\documentclass[10pt,twoside,reqno,final]{manuscript}
\usepackage[numbers,sort&compress]{natbib}

\usepackage[notcite,notref]{showkeys}

\usepackage{cleveref}
\usepackage{indentfirst} 
\usepackage{pgfplots}
\pgfplotsset{compat=1.18}
\usepackage{forest}
\usepackage{tikz}
\usepackage{layout} % see the size of the paper used
\usepackage{transparent} % `input pdf_tex (induced by Inkescape)' needed

\usepackage{graphicx} % adding `pdf' as pictures

\usepackage{bm} % better bold-symbol

\usepackage{leftidx} % left subindex
\usepackage{cases} % multilabel in numcases enviroment
\usepackage{mathabx} % for `\widebar'
\usepackage{amssymb} % for `\symdif'

\hypersetup{
  colorlinks=true,
  linkcolor=blue,
  citecolor=blue,
  filecolor=blue,
  urlcolor=blue
}
\graphicspath{{./figures/}}

\newcommand{\dd}{\mathop{}\!\mathrm{d}}

\newcommand{\floor}[1]{\lfloor {#1} \rfloor}

    \newcommand{\mat}[2][p]{%
      \ifx#1p \begin{pmatrix} #2 \end{pmatrix}%   ( )
      \else\ifx#1b \begin{bmatrix} #2 \end{bmatrix}%   [ ]
      \else\ifx#1B \begin{Bmatrix} #2 \end{Bmatrix}%   { }
      \else\ifx#1v \begin{vmatrix} #2 \end{vmatrix}%   | |
      \else\ifx#1V \begin{Vmatrix} #2 \end{Vmatrix}%   || ||
      \else \begin{matrix} #2 \end{matrix}%   no surrondings
      \fi\fi\fi\fi\fi
    }

\DeclareMathOperator{\spn}{span}
\newcommand{\inner}[1]{\left\langle #1 \right\rangle}

\newcommand{\abs}[1]{\left\lvert#1\right\rvert}

\newcommand{\R}{\mathbb{R}} % real numbers
\newcommand{\Q}{\mathbb{Q}} % rational numbers
\newcommand{\N}{\mathbb{N}} % natural numbers
\renewcommand{\S}{\mathbb{S}} % sphere 
\renewcommand{\P}{\mathbb{P}} % projective space
\newcommand{\ii}{\mathbf{i}} \newcommand{\jj}{\mathbf{j}} \newcommand{\kk}{\mathbf{k}}% for binary index
\newcommand{\HH}{\mathcal{H}} % \mathcal{H}
\newcommand{\LL}{\mathcal{L}} % \mathcal{L}
\newcommand{\GG}{\mathcal{G}} % \mathcal{G} for grassmannian
\renewcommand{\AA}{\mathcal{A}} % \mathcal{A} for affine grassmannian %%%%%% renewcommand!
\newcommand{\RR}{\mathcal{R}} % for rotation set
\newcommand{\II}{\mathcal{I}} % for isometry set
\newcommand{\aA}{\mathfrak{a}} % for projective axis
\newcommand{\cC}{\mathfrak{c}} % for projective center
\newcommand{\wt}[1]{\widetilde{#1}} % for widetilde
\newcommand{\wh}[1]{\widehat{#1}} % for widetilde

\newcommand{\Tt}{\mathsf{T}} % transpose operator
\newcommand{\Rect}{\mathrm{Rect}^{n-1}(\mathbb{R}^n)} %(n-1)-rectifiable set
\newcommand{\inte}{\mathrm{int}}
\newcommand{\cl}{\mathrm{cl}}
\newcommand{\SO}{\mathrm{SO}} % special orthogonal group
\newcommand{\Isom}{\mathrm{Isom}^{+}} % oriention-preserving isometry group
\newcommand{\GL}{\mathrm{GL}} % general linear group
\newcommand{\id}{\mathrm{id}} % identity transform
\newcommand{\aff}{\mathrm{aff}} % affine part
\newcommand{\Aff}{\mathrm{Aff}^{+}} % affine transforms
\newcommand{\Sim}{\mathrm{Sim}^{+}} % similarity transforms

\newcommand{\sign}{\operatorname{sign}} % sign function

\newcommand{\mydot}{\raisebox{0.3ex}{\scalebox{0.75}{\textbullet}}}

\let\oldthebibliography\thebibliography
\renewcommand{\thebibliography}[1]{%
  \oldthebibliography{#1}%
  \setlength{\parskip}{1pt}%       % extra verticle between items
  \setlength{\itemsep}{1pt}%       % verticle between items
  \setlength{\parsep}{1pt}%        % verticle between paragraph in an item
}
\pdfoutput=1

\title{A Study on Kakeya Needle Problem for $(n-1)$-Rectifiable Sets\thanks{
        This work is supported by the National Key R\&D Program of China (No.2023YFA1010800); the Natural Science Foundation of China (No.12271435, No.12301113).

        The authors would like to thank Prof. Bochen Liu for his valuable discussions and suggestions, which greatly improved the quality of this paper.}}

\shorttitle{KAKEYA NEEDLE PROBLEM FOR $(n-1)$-RECTIFIABLE SETS}

\author[1]{Wenjuan Li}
\author[1]{Hongqiang Wang}
\author[2]{Huiju Wang}
\affil[1]{School of Mathematics and Statistics, Northwestern Polytechnical University, Xi'an, 710129, China}
\affil[2]{School of Mathematics and Statistics, Henan University, Kaifeng, 475000, China}

\shortauthor{Wenjuan Li, Hongqiang Wang, and Huiju Wang}

\date{}

\begin{document}

\maketitle
\thispagestyle{firststyle}

\begin{abstract}
In this article we study the analog of Kakeya needle problem for $(n-1)$-rectifiable sets in $\mathbb{R}^n$ and construct the related Nikodym type sets.

The novelty of our approach lies in combining three ingredients: 
the two-dimensional Venetian blind-type construction for isometries in $\mathbb{R}^n$;
the normal geometry of a $(n-1)$-rectifiable set viewed in the projective setting;  
and measure estimates of moving $(n-1)$-rectifiable sets by isometries in $\mathbb{R}^n$.
Together, these ingredients enable us to move $(n-1)$-rectifiable sets along paths of isometries in $\mathbb{R}^n$ and cover a Lebesgue null set.

\keywords{Kakeya needle problem, Nikodym type set, $(n-1)$-rectifiable set}
\MSC{28A75}
\end{abstract}

%%%%%%%%% Contents %%%%%%%%%

\hypersetup{linkcolor=black}
\tableofcontents
\hypersetup{linkcolor=blue}

%%%%%%%%%% main body %%%%%%%%%%

\section{Introduction}

We introduce the analog of Kakeya needle problem and related Nikodym type sets associated with $(n-1)$-rectifiable sets in $\R^n$.
We always assume $n \ge 2$ in this article.

\subsection{Kakeya needle problem}\label{ssect:intro-Kakeya}

The classical Kakeya needle problem in $\R^2$, as stated in \cite{Kak1917}, asks whether a needle (a line segment) can be continuously turned around using translations and rotations such that the covered set during the motion has arbitrarily small $\HH^2$-measure.
Besicovitch \cite{Bes1919,Bes1927} showed this is possible and constructed a set $K \subset \R^2$ of measure zero containing a unit line segment in every direction of $\S^1$.
Such sets are now well known as planar Besicovitch sets, or Kakeya sets in $\R^2$. 
The Kakeya set in $\R^n$ is defined as an $\HH^n$-null set containing a unit line segment of every direction in $\S^{n-1}$.
An example of the Kakeya set in $\R^n$ is $K\times\R^{n-2}$, where $K$ is a Kakeya set in $\R^2$.
See \cite{RFM2026} for a construction of Kakeya sets in $\R^n (n \ge 2)$ using the generalized Perron tree method.

The Kakeya needle problem can be generalized to the question of continuously moving a given set to a new position by translations and rotations, 
while requiring the covered set to be as small as possible. 
In \cite{CsHeLa2018}, Csörnyei, Héra and Laczkovich proved that if a closed set in $\R^2$ can be continuously moved to a new position 
and covers an area of arbitrarily small measure (which they call the Kakeya property), 
then its non-trivial connected components must be subsets of parallel lines or of concentric circles.
Moreover, they proved that if a closed connected set can be moved to any position with an arbitrarily small area covered (which they define as the strong Kakeya property), 
then it must be a line segment, a circular arc or a singleton.
Under the condition that the whole set is required to move, only very special sets can be moved to new positions such that the sets they covered during the motion are of measure zero.
Consequently, several authors have sought to relax the condition so that more sets can be moved.

In the planar case, Chang and Csörnyei \cite[Section 5]{ChCs2019} showed that all rectifiable curves in $\R^2$ can be moved continuously to any (translated or rotated) position, 
covering a set of zero Lebesgue measure, 
provided that at each stage of the motion one is allowed to delete an $\HH^1$-null subset from the translated or rotated copy of $E$.
This is a measure-theoretic analog of the Kakeya needle problem for planar curves, and only subsets with the same $\HH^1$-measure (equimeasurable subsets) are moved in the process.
In higher dimensions, Chang, Dosidis and Kim \cite{ChDoKi2022NikodymSA} showed that 
the unit sphere $\S^{n-1}$ can be moved, in the measure-theoretic sense, to any translated or rotated position in $\R^n$.
Their construction exploits the rotational symmetry of the sphere; in fact, it suffices to consider translations in a two-dimensional subspace.

Inspired by the pioneering works \cite{ChCs2019} on planar rectifiable curves and \cite{ChDoKi2022NikodymSA} on spheres in $\R^n$,
we consider a similar problem of moving other sets to arbitrary positions in $\R^n$.
The natural higher-dimensional analog of planar rotations and translations is the concept of orientation-preserving isometries in $\R^n$,
which we simply call isometries and denote by $\Isom(\R^n)$. 
More precisely, we seek an $(n-1)$-rectifiable set $E$ other than the sphere 
which can be moved (using isometries) to $\iota(E)$ for any $\iota \in \Isom(\R^n)$ and cover a set $F$ of $\HH^n$-measure zero, 
after possibly deleting an $\HH^{n-1}$-null subset from the isometric copy of $E$ for each position in the motion.
The $\HH^n$-null set $F$ covered during the motion is called a \emph{Besicovitch set from $E$ to $\iota(E)$}.
If for every $\iota \in \Isom(\R^n)$, there exists a Besicovitch set from $E$ to $\iota(E)$, then we call $E$ \emph{$\Isom(\R^n)$-Kakeya movable}.
The $\Isom(\R^n)$-Kakeya movable property is a higher-dimensional measure-theoretic analog of the strong Kakeya property introduced in \cite{CsHeLa2018}.

We prove that every $(n-1)$-rectifiable set with finite $\HH^{n-1}$-measure is $\Isom(\R^n)$-Kakeya movable, 
and construct the corresponding Besicovitch sets.
Our main theorem is the following; the relevant definitions can be found in \autoref{sect:preliminaries}.

\begin{theorem}\label{thm:global_Kak_isom}
  Suppose $E$ is an $(n-1)$-rectifiable set in $\R^n$ with $\HH^{n-1}(E)<\infty$.
  Then for any $\iota \in \Isom(\R^n)$, there exist a continuous path $P_{\iota} \subset \Isom(\R^n)$ connecting $\id$ and $\iota$, 
  and a collection $\{E_p\}_{p \in P_{\iota}}$ of subsets of $E$ satisfying $\HH^{n-1}(E_p) = \HH^{n-1}(E)$ for all $p \in P_{\iota}$, 
  such that 
  \begin{align*}
    \bgg[v]{\bigcup_{p \in P_{\iota}} p(E_p)} = 0.
  \end{align*}
  From the arbitrariness of $\iota$, we have that $E$ is $\Isom(\R^n)$-Kakeya movable.
\end{theorem}

\begin{remark}
  For $n=2$, the result coincides with the planar result of \cite[Theorem 6.6]{ChCs2019}. For $\iota \in \Isom(\R^2)$, they call the set $\bigcup_{p \in P_{\iota}} p(E_p)$ a ``Besicovitch set for rotations''. 
  The result for the case $E = \S^{n-1}$ coincides with \cite[Theorem A.4]{ChDoKi2022NikodymSA}.
  All results for $n \ge 3$ and $E \neq \S^{n-1}$ are new.
\end{remark}

Let $V \subset \R^n$ be a 2-dimensional subspace. 
All planar isometries on $V$ can be naturally lifted to isometries in $\R^n$.
We call them \emph{simple isometries}, which form a subgroup of $\Isom(\R^n)$ denoted by $\Isom(V)$ (see \autoref{ssect:isom_R^n}).
Then $\Isom(V)$ contains both translations with direction in $V$ and rotations around $(n-2)$-dimensional affine subspaces orthogonal to $V$.
A simple but important observation is that any isometry in $\Isom(\R^n)$ can be decomposed as the composition of finitely many simple isometries.
With this observation, 
we can reduce \autoref{thm:global_Kak_isom} (the global version for $\Isom(\R^n)$) to a local version for $\Isom(V)$, where $V$ is an arbitrary 2-plane.
As a basic step of proving \autoref{thm:global_Kak_isom},
we consider the $\Isom(V)$-Kakeya movable property for $(n-1)$-rectifiable sets, and prove the following theorem.

\begin{theorem}\label{thm:local_Kak_isom}
  Let $V \subset \R^n$ be a $2$-dimensional subspace, and $E$ be an $(n-1)$-rectifiable set with $\HH^{n-1}(E)<\infty$.
  Then for any $\iota \in \Isom(V)$, there exist a continuous path $P_{\iota}  \subset \Isom(V)$ from $\id$ to $\iota$, 
  and a subset $E_p \subset E$ for all $p \in P_{\iota}$,
  such that
  \begin{align}\label{eq:proof_local_kak_isom.1}
    \bgg[v]{\bigcup_{p \in P_{\iota}} p(E_p)} = 0,
  \end{align}
  and
  \begin{align}\label{eq:proof_local_kak_isom.2}
    \HH^{n-1}(E_p) = \HH^{n-1}(E), \quad \forall p \in P_{\iota}.
  \end{align}
  In other word, $E$ is $\Isom(V)$-Kakeya movable.
\end{theorem}

\begin{remark}
Recall that an $(n-1)$-rectifiable set $E\subset \R^n$ can be decomposed as $E = E_0 \cup \wt{E}$, 
where $\HH^{n-1}(E_0)=0$ and $\wt{E}$ is a union of countably many $C^1$ hypersurfaces, called the regular part of $E$.
Then the normal direction and normal line to $E$ are continuous on $\wt{E}$, and are well-defined $\HH^{n-1}$-almost everywhere on $E$.
Using these geometric concepts, we can determine the moved set $E_p$ for all $p \in P_{\iota}$.
Roughly speaking, $E_p$ is obtained from $E$ by deleting some points satisfying special geometric properties:
  \begin{enumerate}[label=(\alph*)]
    \item If $p \in P_{\iota}$ is a translation with direction $\theta_p $, 
          then $E\setminus E_p$ contains points whose normal directions to $E$ are orthogonal to $\theta_p$.
    \item If $p \in P_{\iota}$ is a rotation,
          then $E\setminus E_p$ contains points whose normal lines to $E$ intersect the axis of the rotation $p \in \Isom(V)$.
  \end{enumerate}
  For more details, see the proof in \autoref{ssect:proof_loc_Kak_isom}.
\end{remark}

For $n \ge 3$, we can also consider the group of all orientation-preserving affine transformations, denoted by $\Aff(\R^n)$.
This group contains isometries, anisotropic dilations and their compositions.
We can generalize \autoref{thm:global_Kak_isom} to the setting of $\Aff(\R^n)$, which gives the following theorem (see \autoref{sect:Aff}).

\begin{theorem}\label{thm:global_Kak_aff}
  Suppose $n \ge 3$, $E$ is an $(n-1)$-rectifiable set in $\R^n$ with $\HH^{n-1}(E)<\infty$.
  Then for any affine transformation $\sigma \in \Aff(\R^n)$, there exist a continuous path $P_{\sigma} \subset \Aff(\R^n)$ connecting $\id$ and $\sigma$, 
  and a collection $\{E_p\}_{p \in P_{\sigma}}$ of subsets of $E$ satisfying $\HH^{n-1}(E_p) = \HH^{n-1}(E)$ for all $p \in P_{\sigma}$, 
  such that 
  \begin{align*}
    \bgg[v]{\bigcup_{p \in P_{\sigma}} p(E_p)} = 0.
  \end{align*}
  In other word, $E$ is $\Aff(\R^n)$-Kakeya movable.
\end{theorem}

\subsection{Nikodym type set}\label{ssect:intro-Nikodymtype}

The classical Nikodym set $F$ in $\R^n$ is an $\HH^n$-null set such that 
for every $x \in \R^n$, there exists a line $l$ through $x$ with $l \cap F$ containing a unit segment.
It is well known that the classical Nikodym set (or classical Kakeya set) plays an important role in modern harmonic analysis and geometric measure theory.
The classical Nikodym set in $\R^2$ can be obtained from the classical Kakeya set via the point-line duality method.
Similar to the Kakeya set, the planar Nikodym set $F \subset\R^2$ can be used to construct the Nikodym set in higher-dimensional space, 
by simply taking $F\times \R^{n-2}$ as the classical Nikodym set in $\R^n$.
More details about recent studies on the classical Nikodym set can be found in \cite{Mat2015} and the references listed therein.

The classical Nikodym set involves packing lines into an $\HH^n$-null set, 
in the sense that through every point there is a line whose intersection with the set contains a segment of unit length.
The problem of packing other types of sets into an $\HH^n$-null set while imposing a suitable covering condition has also attracted considerable interest.
In \cite{Fal1986}, Falconer considered the problem of packing $k$-dimensional affine subspaces (or $k$-planes) into an $\HH^n$-null set.
For a $k$-dimensional linear subspace $W \in \GG^k(\R^n)$ with $1 \le k < n$, Falconer constructed a set $F\subset \R^n$ with $\HH^n(F)=0$ 
such that for every $x \in \R^n$, there exists a k-plane $W_x$ passing through $x$ with $W_x \setminus \{x\} \subset F$ (see \cite{Fal1986} by Falconer, or \cite[Theorem 11.7]{Mat2015} by Mattila).
Falconer's set can be regarded as a Nikodym type set associated with $k$-planes.

Notice that the $k$-plane $W_x$ used in Falconer's construction is essentially an isometric copy of the given subspace $W \in \GG^k(\R^n)$.
It is therefore natural to consider Nikodym type sets associated with other sets under isometries.
We say that $F \subset \R^n$ is a \emph{Nikodym type set associated with a $k$-dimensional set $E \subset \R^n$} ($k \in [1,n)$)
if $\HH^n(F)=0$ and, for every $y \in \R^n$, there exist an orientation-preserving isometry $\iota_y$ and a subset $E_y \subset E$ such that $\iota_y(E)$ passes through $y$, $\HH^k(E_y) = \HH^k(E)$, and $\iota_y(E_y) \subset F$.
This definition is a measure-theoretic analog of the Nikodym type set associated with the $k$-plane considered by Falconer in \cite{Fal1986}.
Chang and Csörnyei, in \cite[Theorem 6.9]{ChCs2019}, constructed Nikodym type sets associated with rectifiable curves in $\R^2$.
In \cite[Theorem 1.1]{ChDoKi2022NikodymSA}, Chang, Dosidis and Kim constructed the Nikodym type set associated with the unit sphere $\S^{n-1}$ in $\R^n$.

In this article, we consider the Nikodym type sets associated with $(n-1)$-rectifiable sets in $\R^n$.
The classical Nikodym set for lines in $\mathbb{R}^n$, which is a linear-type construction without curvature information, 
can be obtained from the planar case ($n=2$) by taking the Cartesian product.
But for a curved hypersurface in $\R^n$, we cannot construct the Nikodym type set directly from lower-dimensional results.
However, combining the isometric Kakeya movable results established in \autoref{thm:local_Kak_isom} with a natural property of the embedded $C^1$ hypersurfaces
(namely, that many isometries acting on such a hypersurface will cover a set with nonempty interior; see \autoref{lem:rot-create-interior}), 
we can construct the Nikodym type set associated with an $(n-1)$-rectifiable set in $\R^n$, as stated in the following theorem.

\begin{theorem}\label{thm:isom_Nikodym}
  Suppose $E$ is an $(n-1)$-rectifiable set in $\R^n$ with $\HH^{n-1}(E)<\infty$.
  Then there exists a set $F \subset \R^n$ satisfying:
  \begin{enumerate}
    \item $F$ has Lebesgue measure zero;
    \item For any $y \in \R^n$, there exist an orientation-preserving isometry $\iota_y$ and a set $E_y \subset E$ such that 
          \begin{enumerate}[label=(\alph*)]
            \item $\iota_y(E)$ passes through $y$;
            \item $\HH^{n-1}(E_y) = \HH^{n-1}(E)$ and $\iota_y(E_y) \subset F$.
          \end{enumerate}
  \end{enumerate}
\end{theorem}

\begin{remark}
  The case $n=2$ has been considered in \cite[Theorem 6.10]{ChCs2019}. 
  The result for the case $E = \S^{n-1}$ coincides with \cite[Theorem 1.1]{ChDoKi2022NikodymSA}.
  All results for $n \ge 3$ and $E \neq \S^{n-1}$ are new.
\end{remark}

\begin{remark}\label{rmk:specialcase}
  Condition (a) states that, for every point $y\in\R^n$, the isometric image $\iota_y(E)$ passes through $y$; 
  this is the central covering property of the Nikodym type set $F$ constructed above.
  In fact, the same theorem remains valid if we replace $E$ in condition (a) by any countable disjoint union of embedded $C^1$ hypersurfaces.
\end{remark}

\begin{remark}
A classical result shows that no $\HH^n$-null set in $\R^n$ can contain, for every center $x \in \R^n$, a positive measure subset of the sphere $\S^{n-1}(x,1)$ 
(for $n \ge 3$, see Stein \cite{St1976}; for $n = 2$, see Bourgain \cite{Bou1986} and Marstrand \cite{Marstrand1987}).

These non-existence results for $\HH^n$-null sets concern placing a copy of $E$ around every point in $\R^n$. 
When $n = 2$, Chang and Csörnyei \cite{ChCs2019} instead placed an equimeasurable isometric copy of $E$ through every point of $\R^2$ and constructed related Lebesgue null sets.
Our theorem shows that this phenomenon also holds in higher-dimensional spaces.
More studies on the possibility (and impossibility) of packing sets into an $\HH^n$-null set can be found, for example, in \cite{Wol1997, Mit1999, HZ2025, YZ2025CurvedKakeyaset, ILT2026}.
\end{remark}

For $n \ge 3$, we also consider the Nikodym type set associated with $(n-1)$-rectifiable sets in the setting of $\Aff(\R^n)$, 
which generalizes \autoref{thm:isom_Nikodym} to the following theorem.

\begin{theorem}\label{thm:aff_Nikodym}
  Suppose $E$ is an $(n-1)$-rectifiable set in $\R^n$ with $\HH^{n-1}(E)<\infty$, 
  and $\Gamma \subset \R^n$ is a countable disjoint union of embedded $C^1$ hypersurfaces.
  Then there exists a set $F \subset \R^n$ satisfying:
  \begin{enumerate}
    \item $F$ has Lebesgue measure zero;
    \item For any $y \in \R^n$, there exist an orientation-preserving affine transformation $\sigma_y$ and a set $E_y \subset E$ such that 
          \begin{enumerate}[label=(\alph*)]
            \item $\sigma_y(\Gamma)$ passes through $y$;
            \item $\HH^{n-1}(E_y) = \HH^{n-1}(E)$ and $\sigma_y(E_y) \subset F$.
          \end{enumerate}
  \end{enumerate}
\end{theorem}

\begin{remark}
  In \cite[Section 7]{ChCs2019}, the authors considered planar Nikodym type sets related to similarity transformations of $\R^2$ (i.e., isotropic dilations and isometries).
  We generalize their results to affine transformations in higher dimensions.
\end{remark}

For $n \ge 3$, Hickman and Zahl \cite{HZ2025} considered the $L^p$ boundedness of strong spherical maximal operators in $\R^n$. 
Their result implies that no $\HH^n$-null set in $\R^n$ can contain, for every center $x \in \R^n$, a positive measure subset of some ellipsoid centered at $x$ whose axes are parallel to the coordinate axes.
By taking $E = \mathbb{S}^{n-1}$ in \autoref{thm:aff_Nikodym}, our theorem implies the existence of an $\HH^n$-null set in $\R^n$ ($n \ge 3$) with the following property:
for every $x \in \R^n$, the set contains $\HH^{n-1}$-almost all of some ellipsoid passing through $x$, 
where the axes are not necessarily parallel to the coordinate axes.

\subsection{Proof sketches of main theorems and paper organization}

We first give some preliminaries in \autoref{sect:preliminaries} related to the projective space, $(n-1)$-rectifiable sets, and the decomposition and geometric description of isometries.
In \autoref{sect:lem} we give the basic lemmas on moving sets along isometries, which play a very important role in the proofs of \autoref{thm:global_Kak_isom} and \autoref{thm:isom_Nikodym}.
We describe the construction of zigzag paths of isometries in \autoref{sect:VBandzigzag}.

\autoref{sect:Kak_isom} is devoted to the $\Isom(\R^n)$-Kakeya movable property for $(n-1)$-rectifiable sets, 
which contains the proof of \autoref{thm:global_Kak_isom} and \autoref{thm:local_Kak_isom}.
We reduce \autoref{thm:global_Kak_isom} to \autoref{thm:local_Kak_isom} using the fact that $\Isom(\R^n)$ can be generated by $\Isom(V)$ for suitable 2-planes $V$.
The proof of \autoref{thm:local_Kak_isom} depends on \autoref{lem:measisom}, \autoref{lem:smallmeasisom}, \autoref{lem:smallneighborhood} and the Venetian blind-type constructions described in \autoref{sect:VBandzigzag}.
For any simple isometry $\iota$, we introduce the projective axis to describe its geometric properties (see \autoref{def:proaxis_isom}).
In the projective setting, we can describe our basic observation on moving a bounded $C^1$ hypersurface $E$ by simple isometries in a uniform way.
For a given simple isometry $\iota \in \Isom(\R^n)$, let $E^{\iota,\delta}$ be the set of points in $E$ 
whose projective normal line intersects the $\delta$-neighborhood of the projective axis of $\iota$ (under the metric of $\P^n$).
Then moving $E^{\iota,\delta}$ along $\iota$ covers a set of small Lebesgue measure controlled by $\delta$.
This simple but powerful observation leads to the key measure estimate given in \autoref{lem:smallmeasisom}.

We study Nikodym type set in \autoref{sect:Nikodym} and prove \autoref{thm:isom_Nikodym} there.
In \autoref{sect:Aff}, we generalize the main theorems from $\Isom(\R^n)$ to $\Aff(\R^n)$ for $n\ge 3$, 
and describe the main tools used to prove \autoref{thm:global_Kak_aff} and \autoref{thm:aff_Nikodym}.

\subsection{Notations}

Let $X=\R^n$ or $X=\P^n$. We regard $(X,d_X)$ as a metric space equipped with the topology induced by its metric. 
For $E\subset X$, $x\in X$, and $y\in E$, we use the following notation:
\begin{itemize}[label=\mydot]
  \item $\mathscr{P}(X)$ and $\mathscr{F}(X)$ denote the collections of all subsets and all closed subsets of $X$, respectively;
  \item $B_X(x,r)$ and $\widebar{B}_X(x,r)$ denote the open and closed balls centered at $x\in X$ of radius $r>0$;
  \item $B_X(E,r) := \bigcup_{x\in E} B_X(x,r)$ is the (open) $r$-neighborhood of a set $E\subset X$;
  \item $B_E(y,r) := B_X(y,r)\cap E$ and $\widebar{B}_E(y,r) := \widebar{B}_X(y,r)\cap E$ are the open and closed balls in $E$;
  \item for an integer $k\in[0,\dim X]$, $\GG^k(X)$ and $\AA^k(X)$ denote the collections of $k$-dimensional linear subspaces and affine subspaces of $X$, respectively;
  \item for $A,B\subset X$, $\spn_X\{A,B\}$ denotes the smallest subspace of $X$ containing $A\cup B$;
  \item for a subspace $W\subset \R^n$, $P_W:\R^n\to W$ denotes the orthogonal projection onto $W$.
\end{itemize}
When $X=\R^n$, we omit the subscript and simply write $B(x,r)$, $B(E,r)$, and $\spn\{A,B\}$.

For $E\subset\R^n$ and $s\in[0,n]$, we denote by $\HH^s(E)$ the $s$-dimensional Hausdorff measure of $E$, 
and by $\dim(E)$ its Hausdorff dimension. 
The Lebesgue measure $\HH^n(E)$ is denoted simply by $\abs{E}$.

We follow the standard convention that for a map $f:X\to Y$, the image of $A\subset X$ is $f(A):=\{f(x):x\in A\}$, 
and the preimage of $B\subset Y$ is $f^{-1}(B):=\{x\in X:f(x)\in B\}$.

We write $a \lesssim_B b$ if $a \le C b$ for a constant $C>0$ depending only on $B$, 
and $a \approx_B b$ if both $a\lesssim_B b$ and $b\lesssim_B a$ hold. 
Also, we use $M\gg 1$ to indicate that $M$ is a sufficiently large number, 
and write $o_M(1)$ for a quantity that tends to $0$ as $M\to \infty$.

\section{Preliminaries}\label{sect:preliminaries}

In this section, we introduce some definitions that will be used in the proofs later.
In order to treat rotations and translations in a unified manner, we shall work in the projective space $\P^n$ rather than in $\R^n$.

\subsection{Embedding $\R^n$ into $\P^n$}\label{ssect:embedinP^n}

To study the geometric properties of rotations and translations in $\R^n$, 
it is more convenient to work in the projective space $\P^n$.

One way to visualize $\P^n$ is to regard it as $\R^{n+1} \setminus \{0\}$ modulo the scalar multiplication, i.e., 
the quotient space $\pi(\R^{n+1} \setminus \{0\})$, where 
\begin{align*}
\pi \colon 
&\R^{n+1} \setminus \{0\} \rightarrow \P^n, \\ 
&x=(x_1,\ \cdots,\ x_n,\ x_{n+1}) \mapsto \pi(x) =[x] \coloneq [x_1: \cdots: x_n : x_{n+1}]
\end{align*}
is the natural quotient map, and the homogeneous coordinate satisfying $[\lambda x] = [x]$ for all $x \in \R^{n+1}\setminus\{0\}$ and $\lambda \in \R\setminus \{0\}$.
We write $[\xi:\xi_{n+1}] \coloneq [\xi_1: \cdots :\xi_n:\xi_{n+1}]$ for $(\xi,\xi_{n+1})=(\xi_1,\ \dots,\ \xi_n,\ \xi_{n+1}) \in (\R^n\times\R) \setminus \{0\}$.
Using this quotient map, all linear subspaces of $\R^{n+1}$ naturally correspond to projective subspaces of $\P^n$. 
More precisely, for any integer $k \in [1,n+1]$, we have the canonical correspondence
\begin{align*}
  \pi \colon \GG^k(\R^{n+1}) \to \GG^{k-1}(\P^n),\quad  W \mapsto \pi(W) = \{\pi(x) \colon x \in W\} \eqcolon \P^{k-1}_{W},
\end{align*}
where, abusing notation, we also write $\pi$ for the corresponding map.
Equipped with the standard Euclidean inner product $\inner{\cdot,\cdot}$ on $\R^{n+1}$, 
the notion of orthogonality extends naturally to $\P^n$:
we say two projective subspaces $\pi(V)$ and $\pi(W)$ are orthogonal in $\P^n$, denoted by $\pi(V)\perp \pi(W)$, 
if and only if the subspaces $V$ and $W$ are orthogonal in $\R^{n+1}$.
The orthogonal complement in $\R^{n+1}$ also induces a natural duality between $\GG^{k}(\P^n)$ and $\GG^{n-k-1}(\P^n)$:
\begin{align*}
  (\cdot)^{\perp}  \colon \GG^{k}(\P^n) \rightarrow \GG^{n-k-1}(\P^n), \quad A = \P^k_{W} \mapsto A^{\perp} \coloneq \P^{n-k-1}_{W^{\perp}}
\end{align*}
for $k=0, 1, \dots, n-1$.
For example, the orthogonal complement of a point $x = [v] \in \P^n$ is defined as $x^{\perp} \coloneq \pi(v^\perp) \subset \P^n$,
where $v^{\perp}$ is the hyperplane in $\R^{n+1}$ orthogonal to $v$.
This orthogonal complement operator can be defined for any subset of $\P^n$:
\begin{align*}
  (\cdot)^{\perp} \colon 
              & \mathscr{P}(\P^n) \rightarrow \mathscr{F}(\P^n),\\
              &  A \mapsto A^{\perp} \coloneq \bigcap_{x \in A} x^\perp,
\end{align*}
which maps any subset of $\P^n$ to a closed projective subspace (possibly empty).
In particular, we have $({\P^n})^{\perp} = \emptyset$, $\emptyset^{\perp} = \P^n$
and $(A^\perp)^\perp = A$ for any projective subspace $A \subset \P^n$.

Geometrically, the projective space $\P^n$ is obtained from $\R^n$ by adding points at infinity.
In homogeneous coordinates,
\begin{align*}
  \P^n = \{[x_1:\cdots:x_{n+1}] : (x_1,\dots,x_{n+1}) \neq (0,\dots,0)\} = \P^n_{\aff} \cup \P^{n-1}_{\infty} ,
\end{align*}
where $\P^n_{\aff} := \{[x:1] : x \in \R^n\}$ 
and $\P^{n-1}_{\infty} := \{[x:0] : x \in \R^n\setminus\{0\}\}$ are the affine part and the part at infinity of $\P^n$.

To formalize this, we embed $\R^n$ into $\R^{n+1}$ as the hyperplane $x_{n+1}=1$ for affine points, 
and $x_{n+1} =0$ for the hyperplane at infinity (corresponding to directions of $\R^n$).
Let $\mu_1:\R^n\hookrightarrow\R^{n+1}$, $x\mapsto(x,1)$ and $\mu_0:\R^n\setminus \{0\} \hookrightarrow \R^{n+1}\setminus\{0\}$, $x\mapsto(x,0)$.
We define $\pi_1 := \pi\circ\mu_1$ and $\pi_0 := \pi\circ\mu_0$, so that
\begin{align*}
  \pi_1 &\colon \R^n \to \P^n_{\aff},\quad x\mapsto [x:1], \\
  \pi_0 &\colon \R^n\setminus\{0\} \to \P^{n-1}_{\infty},\quad x\mapsto [x:0],
\end{align*}
then we have $\P^n = \pi_1(\R^n) \cup \pi_0(\R^n\setminus\{0\})$.
We call $\pi_0(x)\in\P^{n-1}_{\infty}$ (equivalently, $[x]\in\P^{n-1}$) the projective direction (or simply direction) of $x \in \R^n\setminus\{0\}$.
Thus, we identify $\P^{n-1}$ with $\P^{n-1}_{\infty}\subset\P^n$ via $[x]\mapsto[x:0]$, i.e., $\P^{n-1} \cong \P^{n-1}_{\infty}$, and use it as the direction set of $\R^n$.

For an arbitrary set $E \subset \R^n$, we define the related projective set by
\begin{align}
  [E:1] := \cl_{\P^n} \pi_1(E), \label{eq:projset}
\end{align}
where $\cl_{\P^n}$ is the closure with respect to the standard topology of $\P^n$.
Notice that for a singleton $\{x\}\subset\R^n$, we have $\pi_1(x) = [\{x\}:1]$, 
thus \eqref{eq:projset} can be regarded as a natural extension of $\pi_1$.

For an affine subspace $A = y + A_0 \in \AA^k(\R^n)$ with $k \in \{1,\dots,n\}$, $y\in\R^n$ and $A_0\in\GG^k(\R^n)$, simple calculation gives
\begin{align}\label{eq:equation1}
    [A_0:1]& = \P^{k}_{A_0\times\R}, \quad [A:1] = \{[x:1] : x\in A\} \cup \{[x:0] : x\in A_0\setminus\{0\}\}.
\end{align}
We adopt the following notation for the projective direction set of $A$:
\begin{align*}
  [A:0] \coloneq [A:1] \cap \P^{n-1}_{\infty}.
\end{align*}
A simple calculation leads to
\begin{align}\label{eq:equation2}
 [A:0] = [A_0:0] =\P^{k-1}_{A_0\times \{0\}}.
\end{align}

These notations, while not standard, are adopted for convenience in analyzing affine subspaces from the projective perspective.

\begin{remark}\label{rem:V^{perp}}
  We will frequently use $\P^2_{V\times\R}$ as the natural embedding image of $V \in \GG^2(\R^n)$ into $\P^n$.
  For $n \ge 3$, the orthogonal complement of $\P^2_{V\times\R}$ in $\P^n$ is
  \begin{align}\label{eq:Vperp}
    (\P^2_{V\times\R})^{\perp} = \P^{n-3}_{V^{\perp}\times\{0\}},
  \end{align}
  which is nonempty and consists precisely of all the projective directions of $V^\perp$ (recall that $\P^{n-3}_{V^{\perp}\times\{0\}} = [V^\perp:0]$).
  For the planar case $n = 2$, $V \in \GG^2(\R^2)$ implies $V=\R^2$, 
  thus we have $(\P^2_{V\times\R})^{\perp} = \emptyset$.
  In this case we interpret the right-hand side of \eqref{eq:Vperp} as the empty set.
\end{remark}

\subsection{$(n-1)$-rectifiable set in $\R^n$}\label{sect:Rect}

In this subsection we recall the notion of an $(n-1)$-rectifiable set and its associated normal map.

\begin{definition}[$(n-1)$-rectifiable set and its normal map]\label{def:rect_normal}
  A set $E \subset \R^n$ is called countably $(n-1)$-rectifiable if it admits a disjoint decomposition
  \begin{align}\label{eq:rect}
    E = E_0 \sqcup\wt{E},\ \text{and}\  \wt{E}=\bigsqcup_{i=1}^\infty E_i,
  \end{align}
  where $\HH^{n-1}(E_0)=0$ and each $E_i$ is an embedded $C^1$ submanifold of dimension $n-1$ in $\R^n$.
  We call $\wt{E}$ the regular part of $E$.

  For such an $E$, each $E_i$ admits a continuous classical normal map into the unit sphere $\S^{n-1}$; 
  hence the regular part $\wt{E}$ admits a piecewise continuous classical normal map 
  $\wh{N} \colon \wt{E} \to \S^{n-1}$.
  The normal map $N \colon E \to \P^{n-1}_{\infty}$ of $E$ is then defined by
  \begin{align*}
    N(x) = \pi_0(\wh{N}(x)) \quad \text{for } x \in \wt{E},
  \end{align*}
  while $N(x)$ may be chosen arbitrarily in $\P^{n-1}_{\infty}$ for $x \in E_0$.
  Thus, $N$ is continuous $\HH^{n-1}$-almost everywhere on $E$.

  For brevity, we omit the term ``countably'' and simply call such sets $(n-1)$-rectifiable sets.
  We denote by $\Rect$ the collection of all $(n-1)$-rectifiable sets in $\R^n$.
\end{definition}

\begin{remark}
Although the decomposition in \eqref{eq:rect} is not unique (the sets $\{E_i\}$ may vary), 
the regular part $\wt{E}$ and the normal map $N$ are uniquely determined up to an $\HH^{n-1}$-null set. 
Hence the definition of $N$ is independent of the chosen decomposition in the sense of $\HH^{n-1}$-a.e. equivalence.
\end{remark}

\begin{remark}\label{rmk:recthypersurface}
The questions discussed in this paper are invariant under modifications on $\HH^{n-1}$-null sets. 
In particular, since $\HH^{n-1}(E)=\HH^{n-1}(\wt{E})$, we may, 
without loss of generality, identify $E\in\Rect$ with its regular part $\wt{E}$, 
which is a countable disjoint union of embedded $C^1$ hypersurfaces equipped with a piecewise continuous normal map.
\end{remark}

For more details about rectifiable sets, see \cite[Chapter 3]{LY2002} or \cite[Chapter 3]{Fed1969}.

We use $n_x \in \AA^1(\R^n)$ as the normal line to $E$ at $x$, and set 
\begin{align*}
  \nu_x=[n_x:1] \in \GG^1(\P^n)
\end{align*}
as the projective normal line to $E$ at $x$.

\subsection{Orientation-preserving isometries in $\R^n$}\label{ssect:isom_R^n}
This section includes the notations related to translations, rotations and their compositions, known as orientation-preserving isometries, 
which we simply call isometries.

\subsubsection{Basic definitions}
\textbf{Translations.} 
All translations in $\R^n$ form the translation group, which is isomorphic to the usual additive group $\R^n$.
Each translation is a transformation of $\R^n$ and can be parameterized by a vector $z \in \R^n$ as
\begin{align*}
  \tau_z \colon \R^n \to \R^n,\ x \mapsto z+x.
\end{align*}

\textbf{Rotations and simple rotations.} 
Let $\SO(n)=\{A\in \GL_n(\R) \colon A^{\Tt}A=I, \det(A)=1\}$ be the group of special orthogonal matrices.
Then each pair $(z,A) \in \R^n\times\SO(n)$ uniquely determines a rotation around $z$,
which is
\begin{align*}
  \rho_{z,A} \colon \R^n \to \R^n,\ x \mapsto z+A(x-z).
\end{align*}

To simplify an arbitrary rotation, we first consider the real spectral decomposition of $A \in \SO(n)$ given by
\begin{align*}
  A =
  P^{\Tt}\begin{pmatrix}
          R(\theta_1)& & & \\
           & \ddots& &  \\
           & & R(\theta_k)& \\
           & & & I_{n-2k} 
  \end{pmatrix}P 
  = \sum_{i=1}^k P_i^{\Tt} R(\theta_i) P_i + Q^{\Tt} Q,
\end{align*} 
where $k \le \floor{\frac{n}{2}}$, $P^{\Tt}= \mat[p]{P_1^{\Tt}& \cdots & P_k^{\Tt} & Q^{\Tt} }$ is an orthogonal matrix,
$R(\theta_i) = \mat[p]{\cos\theta_i & -\sin\theta_i \\ \sin\theta_i & \cos\theta_i}$ and $\theta_i \in (-2\pi,2\pi)$ for $i=1, \cdots,k$.
Then the arbitrary rotation can be rewritten as 
\begin{align*}
  \rho_{z,A}(x) 
  =z + \sum_{i=1}^k P_i^{\Tt} R(\theta_i) P_i(x-z) + (I-\sum_{i=1}^kP_i^{\Tt}P_i)(x-z),
\end{align*}
where we use the fact $Q^{\Tt}Q=I-\sum_{i=1}^kP_i^{\Tt}P_i$.
Notice that $P_iP_j^{\Tt} = 0$ for $i \neq j$, we have
\begin{align}
  \rho_{z,A}(x) 
  = z + \prod_{i=1}^k (P_i^{\Tt}R(\theta_i) P_i +I-P_i^{\Tt}P_i)(x-z). \label{eq:deco_rotation1}
\end{align}

The rotation $\rho_{z,A}$ with $k=1$ in \eqref{eq:deco_rotation1} plays an important role in the decomposition of arbitrary rotation.
Thus, we define the following concept.

\begin{definition}[Simple rotation in $V$]\label{def:V_rotat}
  Let $V \in \GG^2(\R^n)$ with orthonormal basis $\{u,v\} \subset V$, $z \in \R^n$ and $\theta \in \R$.
  A simple rotation in $V$ of angle $\theta$ around $z$, denoted by $\RR_{z, V}^\theta$, 
  is the affine transformation
  \begin{align*}
    \RR_{z, V}^\theta \colon 
    &\R^n \to \R^n, \\
    &x \mapsto z+(P^{\Tt}R(\theta)P + I-P^{\Tt}P)(x-z),
  \end{align*}
  where $P^{\Tt}= 
            \begin{pmatrix}
              u & v
            \end{pmatrix}$ 
  and  $R(\theta)=
                  \begin{pmatrix}
                    \cos\theta & -\sin\theta \\ 
                    \sin\theta & \cos\theta
                  \end{pmatrix}$.
  This operator actually defines a rotation which is essentially the planar rotation of angle $\theta$ in $V$ 
  that keeps the affine subspace $z+V^\perp$ invariant.
  We then regard $z+V^\perp$ as the axis of the rotation $\RR_{z, V}^\theta$.
\end{definition}

\begin{remark}\label{rmk:equi-rotat}
  For fixed $\theta$ and $V$, the simple rotation $\RR_{z,V}^{\theta}$ depends on its axis $z+V^\perp$, rather than on the particular choice of $z\in\R^n$.
  Hence, for any $z\in\R^n$, we may represent the same rotation by $\RR_{z',V}^{\theta}$, 
  where $z'=P_V(z)\in V$ is the orthogonal projection of $z$ onto $V$.
\end{remark}

\begin{remark}
  We use the same notation $\RR_{z, V}^\theta$ both for the affine transformation $\RR_{z, V}^\theta \colon \R^n \to \R^n$ 
  and for the continuous motion that rotates any point $x \in \R^n$ around the axis $z+V^\perp$.
  We always regard simple rotations as continuous motions, thus $\RR_{z, V}^\theta$ and $\RR_{z, V}^{\theta+2\pi}$ are different.
\end{remark}

Now go back to the decomposition \eqref{eq:deco_rotation1}, and let $V_i \in \GG^2(\R^n)$ be the subspace spanned by the rows of $P_i$. 
Using the above definition, \eqref{eq:deco_rotation1} can be restated as decomposing a given rotation into a sequence of simple rotations in the 2-planes $\{V_i\}_{i=1}^k$, 
which leads to the following lemma.
\begin{lemma}\label{lem:decomp_rotat}
  Suppose $\rho_{z,A}$ is an arbitrary rotation around $z \in \R^n$, 
  then there exist $k \le \floor{\frac{n}{2}}$, orthogonal planes $\{V_i\}_{i =1}^{k} \subset \GG^2(\R^n)$ and angle $\theta_i \in (-2\pi,2\pi)$ for each $i=1\cdots,k$, 
  such that 
  \begin{align*}
    \rho_{z,A} = \prod_{i=1}^k \RR_{z, V_i}^{\theta_i}.
  \end{align*}
\end{lemma}

\textbf{Orientation-preserving isometries.} 
All rotations, translations and their ordered compositions are called orientation-preserving isometries, or simply isometries.
All orientation-preserving isometries in $\R^n$ form a transformation group denoted by $\Isom(\R^n)$, 
which is canonically isomorphic to the semidirect product $\R^n \rtimes \SO(n)$.
Each element $(z, A) \in \R^n \times \SO(n)$ corresponds to an isometry $\iota_{(z,A)} \in \Isom(\R^n)$ defined as
\begin{align}\label{eq:decompo_isometry-trans-rot}
  \iota_{(z,A)} \colon \R^n \to \R^n,\ x \mapsto z+Ax.
\end{align}

For a given $V \in \GG^2(\R^n)$, the simple rotations and translations in $V$ (including the identity $\id$) are examples of isometries.
They will play an important role in our discussions, and we define them as follows:

\begin{definition}[Simple isometry in $V$]\label{def:Isom(V)}
  For $V \in \GG^2(\R^n)$, all simple rotations and translations in $V$ are called simple isometries in $V$, and they form the set
  \begin{align*}
    \Isom(V) := \{\RR_{z,V}^{\theta} : z \in V,\ \theta \in \R\setminus\{0\}\} \cup \{\tau_z : z \in V\},
  \end{align*}
  where we use the convention that the axis of $\RR_{z,V}^{\theta}$ is represented by $z\in V$ (see \autoref{rmk:equi-rotat}).
\end{definition}

It is convenient to treat all simple rotations and translations uniformly, which leads to the following:

\begin{definition}[Simple isometry in $\R^n$]
  All simple rotations and translations in $\R^n$ are called simple isometries in $\R^n$.
  They form the set
  \begin{align*}
    \II(\R^n) := \bigcup_{V \in \GG^2(\R^n)} \Isom(V),
  \end{align*} 
  which is a proper subset of $\Isom(\R^n)$ whenever $n\ge 3$.
\end{definition}

Since any translation can be decomposed as the composition of two simple rotations, 
and any rotation can be factorized into at most $\floor{\frac{n}{2}}$ simple rotations (see \autoref{lem:decomp_rotat}),
we can simplify a general isometry by simple rotations, which leads to the following lemma.
\begin{lemma}\label{lem:decomp_isometry}
  Any orientation-preserving isometry in $\R^n$ can be decomposed as the composition of at most $\floor{\frac{n}{2}}+2$ simple rotations.
\end{lemma}

\subsubsection{Simple isometry from the projective perspective}\label{ssect:simple projective rotations}

In this section we describe the simple isometries uniformly from the projective perspective.

Recall that a simple rotation $\RR_{z,V}^{\theta} \ (\theta \neq 0)$ is uniquely determined by its axis $z+V^{\perp} \in \AA^{n-2}(\R^n)$ and its directed rotation angle $\theta \in \R$.
Likewise, a translation $\tau_z \ (z\neq 0)$ is uniquely determined by the hyperplane $z^{\perp} \in \GG^{n-1}(\R^n)$ and the directed translation length $\sign(z)\abs{z}$,
where $\sign(z) \in \{\pm1\}$ denotes the sign of the first non-zero coordinate of $z$,
distinguishing $\tau_z$ from $\tau_{-z}$.

By focusing on the quantitative information (directed rotation angle or directed translation length) and the positional information (rotation axis or translation direction) of the simple isometries, 
we introduce the following concepts.

\begin{definition}[Projective axis and projective angle of simple isometry]\label{def:proaxis_isom}
  For a simple isometry $\iota \in \II(\R^n)\setminus\{\id\}$, we define its projective axis $\aA_{\iota} \in \GG^{n-2}(\P^n)$ as
  \begin{align}\label{eq:projaxis}
    \aA_{\iota} = 
    \begin{cases}
      [z+V^{\perp}:1], &\text{if} \ \iota = \RR_{z,V}^{\theta}, \\
      [z^{\perp}:0], &\text{if} \ \iota = \tau_z.
    \end{cases}
  \end{align}
  The projective angle $\theta_{\iota} \in \R\setminus\{0\}$ is defined as
  \begin{align*}
    \theta_{\iota} = 
    \begin{cases}
      \theta, &\text{if} \ \iota = \RR_{z,V}^{\theta}, \\
      \sign(z)\abs{z} , &\text{if} \ \iota = \tau_z.
    \end{cases} 
  \end{align*}
\end{definition}

These notations unify the position and quantitative data of simple isometries, 
and allow us to describe the simple isometries in the language of $\P^n$,
which leads to the following definition.
\begin{definition}\label{def:prorota_R^n}
  Any non-identity simple isometry in $\R^n$ can be parameterized by the space $(\R\setminus\{0\}) \times \GG^{n-2}(\P^n)$, i.e., 
  each simple isometry $\iota$ can be interpreted as a simple projective rotation in $\P^n$ with angle $\theta_{\iota} \in \R\setminus\{0\}$ and axis $\aA_\iota \in \GG^{n-2}(\P^n)$,
  denoted by $\iota = \RR_{\aA_{\iota}}^{\theta_{\iota}}$.
  Thus, we have
  \begin{align*}
    \II(\R^n)\setminus\{\id\} =\{\RR_{\aA}^{\theta} \colon \theta \in \R\setminus\{0\}, \aA \in \GG^{n-2}(\P^n)\}.
  \end{align*}
\end{definition}

\subsubsection{Simple isometry in $V$, Part $1$: projective description}

In this section we fix a 2-plane $V$ and discuss the non-identity isometries in $\Isom(V)$ from the projective perspective.
First we show the notion of projective axis can be further simplified in the setting of $\Isom(V)$.

For a simple rotation $\iota = \RR_{z,V}^{\theta}$ with $\theta\neq0$ and $z \in V$, 
applying \eqref{eq:equation1} and \eqref{eq:equation2} to \eqref{eq:projaxis} gives 
\begin{align*}
  \aA_{\iota} = [z+V^\perp:1] = \{[x:1]\colon x \in z+V^\perp\}\cup \P^{n-3}_{V^\perp \times\{0\}} \supsetneq \P^{n-3}_{V^\perp \times\{0\}},
\end{align*}
Similarly, for a translation $\iota=\tau_z$ with $z \in V$, the combination of \eqref{eq:equation2} and \eqref{eq:projaxis} yields
\begin{align*}
  \aA_{\iota} = [z^\perp:0] = \P^{n-2}_{z^\perp \times\{0\}}\supsetneq \P^{n-3}_{V^\perp \times\{0\}}.
\end{align*}
To summarize, for $\iota \in \Isom(V)\setminus\{\id\}$, we have 
\begin{align*}
  \aA_{\iota} \supsetneq \P^{n-3}_{V^\perp \times\{0\}} = (\P^2_{V\times \R})^\perp,
\end{align*}
where the last equality follows from \autoref{rem:V^{perp}}.
Thus, the intersection of $\aA_{\iota}$ with $\P^2_{V\times \R}$ is a single projective point, 
which we define as the projective center of $\iota$.
(For $n=2$, recall \autoref{rem:V^{perp}}, we interpret $\P^{n-3}_{V^{\perp}\times\{0\}}$ as the empty set, and the statements above remain valid.)

\begin{definition}\label{def:prorota_V}
  For $\iota \in \Isom(V)\setminus\{\id\}$, we define its projective center as
  \begin{align}\label{eq:proj_center2}
    \cC_{\iota} \coloneq \aA_{\iota} \cap \P^2_{V\times \R}  \in \P^2_{V\times \R}.
  \end{align}
\end{definition}

The projective center and projective axis are closely related by the following concept of axis-lifting operator:

\begin{definition}[Axis-lifting operator]\label{def:center_axis_corrspd}
  For fixed $V \in \GG^2(\R^n)$, there is a canonical correspondence between $\P^2_{V\times \R}$ (the space of projective centers) 
  and $\{Y \in \GG^{n-2}(\P^n) \colon Y \supset \P^{n-3}_{V^\perp\times\{0\}}\}$ (the space of projective axes) given by
  \begin{align*}
    \aA \colon &\P^2_{V\times \R} \to \{Y \in \GG^{n-2}(\P^n) \colon Y \supset \P^{n-3}_{V^\perp\times\{0\}}\} \\
    &x \mapsto \aA(x) \coloneq \spn_{\P^n} \{x, \P^{n-3}_{V^\perp\times \{0\}} \},
  \end{align*} 
  which we call the axis-lifting operator.
  This operator extends naturally to any nonempty subset $A \subset \P^2_{V\times \R}$ by
    \begin{align*}
      \aA(A) \coloneq   \bigcup_{x \in A}\aA(x).
    \end{align*}
  And for a connected closed set $A \in \mathscr{F}(\P^2_{V\times \R})$, we have $\aA(A)= \spn_{\P^n}\{A,\P^{n-3}_{V^\perp\times \{0\}}\}$.
\end{definition}

Thus, for any $\iota \in \Isom(V)\setminus\{\id\}$, the projective axis $\aA_\iota$ and
the projective center $\cC_{\iota}$ from \eqref{eq:projaxis} and \eqref{eq:proj_center2}
can be transformed into each other naturally:
\begin{align*}
  \cC_{\iota}  = \aA_{\iota} \cap \P^2_{V\times \R}; \quad
  \aA_{\iota} = \aA(\cC_{\iota}).
\end{align*} 
For $\iota \in \Isom(V)\setminus\{\id\}$, we abuse the notation $\iota=\RR_{\aA_{\iota}}^{\theta_{\iota}}$ from \autoref{def:prorota_R^n} and write 
\begin{align*}
  \iota=\RR_{\cC_{\iota}}^{\theta_{\iota}},
\end{align*}
which also yields that
  \begin{align*}
    \Isom(V)\setminus\{\id\} = \{\RR_{\cC}^{\theta} \colon \theta \in \R\setminus\{0\} , \cC \in \P^2_{V\times\R} \}.
  \end{align*}

The axis-lifting operator $\aA$ connects projective centers and projective axes, and lifts information from $\P^2_{V\times\R}$ to $\P^n$.
It has the following properties:
  \begin{enumerate}
    \item If $A$ is a closed subset of $\P^2_{V\times\R}$, then $\aA(A)$ is closed in $\P^n$;
    \item For any family $\{A_m\}_m \subset \mathscr{P}(\P^2_{V\times\R})$, we have 
          \begin{enumerate}[label=(\alph*)]
            \item $\bigcup_m \aA(A_m) = \aA(\bigcup_m A_m)$,
            \item $\bigcap_m \aA(A_m) \supset \aA(\bigcap_m A_m)$,
          \end{enumerate}
    \item For $W\in \GG^2(V\times\R)$, we have $\aA(\P^1_{W}) = \spn_{\P^n} \{\P^1_{W}, \P^{n-3}_{V^{\perp}\times\{0\}} \} \in \GG^{n-1}(\P^n)$,
  \end{enumerate}
which follow from the elementary properties of linear subspaces under intersection and union.

\subsubsection{Simple isometry in $V$, Part $2$: $V\times\R$-parametrization}

The projective description of non-identity simple isometries in $V$ — namely, their projective axes, centers, and angles — 
is useful for analyzing the geometry of moving a set along an isometry, 
but reveals little about the composition structure of $\Isom(V)$. 
We therefore introduce another parametrization of $\Isom(V)$, 
which comes from a detailed expression of the projective center and is better suited for studying compositions.

For any non-identity isometry $\iota \in \Isom(V)\setminus\{\id\}$, a direct computation yields
  \begin{align*}
    \cC_{\iota} =
    \begin{cases}
      [z:1], &\text{if} \ \iota = \RR_{z,V}^{\theta}\ \text{for} \ \theta\in \R\setminus\{0\} \ \text{and} \ z \in V,\\
      [\RR_{0,V}^{\frac{\pi}{2}}(z):0], &\text{if} \ \iota = \tau_z \ \text{for} \ z\in V\setminus\{0\}.
    \end{cases}
  \end{align*}
Then the vector $(\RR_{0,V}^{\frac{\pi}{2}}(z),0)$ can uniquely determine the translation $\tau_z$.
However, for the simple rotation $\RR_{z,V}^{\theta}$, the vector $(z,1)$ encodes only the positional information (the projective center) 
and carries no information about the angle.
Since homogeneous coordinates are invariant under multiplication by non-zero scalars, 
we may multiply by $\theta$ to obtain the vector $(\theta z,\theta)$, 
which determines both the projective center (i.e., $[\theta z:\theta] = [z:1] =\cC_{\iota}$) and the projective angle $\theta \in \R\setminus\{0\}$.
All the newly constructed vectors lie in $V\times\R$.
And we can regard $\id$ as the zero vector $0\in V\times\R$.

Motivated by the above observation, 
we use $V\times\R$ to parameterize the isometries in $\Isom(V)$, which leads to the following definition.
 
 \begin{definition}[$V\times \R$-parametrization]\label{def:labelby_V}
  For a given $V \in \GG^2(\R^n)$, we define a bijection between $\Isom(V)$ and $V\times\R$ as 
  \begin{align*}
    \eta \colon \Isom(V) \to V\times\R, \ \iota \mapsto \eta(\iota),
  \end{align*}
  which satisfies
  \begin{align}\label{eq:nonproj_center}
      \eta(\iota) 
      \coloneq 
        \begin{cases}
          (\theta z,\theta), &\text{if} \ \iota = \RR_{z,V}^{\theta}, \\
          (\RR_{0,V}^{\frac{\pi}{2}}(z),0), &\text{if} \ \iota = \tau_z,
        \end{cases}
    \end{align}
where $\theta \in \R\setminus\{0\}$ and $z \in V$.
\end{definition}

\begin{remark}
  The idea of using Euclidean space to parameterize rotations and translations via \eqref{eq:nonproj_center} and \autoref{def:labelby_V} has been used for $n=2$ in \cite[Section 3.2]{ChCs2019},
  where complex numbers are used to express planar rotations.
  Our approach is inspired by their method; the operator $\RR_{0,V}^{\frac{\pi}{2}}$, the simple rotation around the axis $V^\perp$ by angle $\frac{\pi}{2}$ appeared in \eqref{eq:nonproj_center}, 
  serves as the imaginary unit ``$\mathrm{i}\mkern1mu$'' in the 2-plane $V$. 
\end{remark}

As stated above, the vector $\eta(\iota)$ encodes the full geometric information of $\iota \in \Isom(V)$.
Especially, for $\iota \in \Isom(V)\setminus\{\id\}$, we have 
\begin{align*}
  \cC_{\iota} = [\eta(\iota)].
\end{align*}
The vector norm $\abs{\eta(\iota)}$ will be used to quantify the Lebesgue measure of the set swept out by moving a set along $\iota$,
which will be introduced in \autoref{sect:lem} (see \autoref{lem:measisom} for more details).
We say $\iota \in \Isom(V)$ is a small isometry, if the vector norm $\abs{\eta(\iota)}$ is small.

Next we describe the composition property of $\Isom(V)$ under the $V\times\R$-parametrization.
For any $\iota \in \Isom(V)$ and integer $N \ge 1$, we have
  \begin{align}\label{eq:decomp_N_terms}
    \iota = [\eta^{-1}(\frac{\eta(\iota)}{N})]^N 
    \coloneq \underbrace{\eta^{-1}(\frac{\eta(\iota)}{N}) \circ \cdots \circ \eta^{-1}(\frac{\eta(\iota)}{N})}_{N \ \text{terms}},
  \end{align}
  where the composition is taken from right to left.
  This identity corresponds to splitting the isometry $\iota=\RR^{\theta}_{\cC}$ into $N$ smaller isometries $\RR^{\theta/N}_{\cC}$ with the same projective center $\cC \in \P^2_{V\times\R}$.

  More importantly, the $V\times \R$-parametrization can be used to approximate the composition of small isometries in $\Isom(V)$ by the vector addition in $V\times\R$.
  This property is analogous to \cite[Lemma 5.4, Lemma 5.5]{ChCs2019}; we summarize it as the following lemma and omit the proof, which is similar to the one given there.

  \begin{lemma}\label{lem:approx_N_divides}
    Suppose $x_0, \wt{x}_1 \in (V\times\R)\setminus\{0\}$.
    For any $\epsilon >0$, if $N$ is a sufficiently large integer, 
    then there exists $x_1 \in (V\times\R)\setminus\{0\}$, such that
    \begin{enumerate}[label=(\arabic*)]
      \item $\eta\bg{\eta^{-1}(x_0/N)\circ \eta^{-1}(x_1/N)} = x_0/N + \wt{x}_1/N$,
      \item $\abs{x_1-\wt{x}_1} \le \epsilon$.
    \end{enumerate}
    If $x_0, \wt{x}_1 \in V\times\{0\}$, then we can take $x_1 = \wt{x}_1 \in V\times\{0\}$ which still satisfies the first condition.
  \end{lemma}

  The combination of \eqref{eq:decomp_N_terms} and \autoref{lem:approx_N_divides} leads to the following result:
   \begin{lemma}\label{lem:approx_N_term}
    Suppose $\iota \in \Isom(V)\setminus \{\id\}$ with $\eta(\iota) = y_0+\wt{y}_1$ for some $y_0,\wt{y}_1 \in (V\times\R)\setminus\{0\}$.
    Then for any $\epsilon >0$ and sufficiently large integer $N$, there exists $y_1 \in B_{V\times\R}(\wt{y}_1,\epsilon)$ such that
    \begin{align*}
      \iota = \underbrace{(\iota_0 \circ\iota_1) \circ \cdots \circ (\iota_0 \circ\iota_1)}_{N \ \text{terms}},
    \end{align*}
   where the isometries are
   \begin{align*}
      \iota_0 \coloneq \eta^{-1}(\frac{y_0}{N}), \quad  
      \iota_1 \coloneq \eta^{-1}(\frac{y_1}{N}).
   \end{align*}
  \end{lemma}

  This lemma decomposes a given non-identity simple isometry $\iota$ into many small isometries 
  $\iota_0 $ and $\iota_1$ whose $V\times \R$-parametrization vectors satisfy special properties.
  This will be useful for the construction of zigzag path in \autoref{sect:VBandzigzag}.

\subsubsection{Paths of isometries}

An isometry path in $\R^n$ is a continuous map 
\begin{align*}
  \gamma \colon [0,1] \to \Isom(\R^n)
\end{align*}
with $\gamma(0) = \id$. 
We denote its image by $P := \gamma([0,1]) \subset \Isom(\R^n)$, 
and by slight abuse of notation, we refer to $P$ itself as an isometry path when the parametrization is clear from context.
Given two isometry paths $\gamma_1,\gamma_2:[0,1]\to\Isom(\R^n)$ with $\gamma_i(1)=\iota_i \in \Isom(\R^n)$ for $i =1,2$, 
their concatenation $\gamma_3 \coloneq \gamma_1 * \gamma_2$ is the path defined by
\begin{align*}
  (\gamma_1 * \gamma_2)(t) \coloneq
    \begin{cases}
    \gamma_1(2t), & 0\le t\le \frac12, \\
    \gamma_2(2t-1) \circ \iota_1, & \frac12 \le t \le 1.
    \end{cases}
\end{align*}
This gives a continuous path from $\id$ to $\iota_2 \circ \iota_1$.
Let $P_i = \gamma_i([0,1])$ for $i=1,2,3$,
we write $P_3 = P_2\circ P_1$ for the above concatenation of paths.

For a simple isometry $\iota = \RR_{\aA}^{\theta} \in \II(\R^n)$, 
we especially take $\gamma \colon [0,1] \to \Isom(\R^n),\ t \mapsto \RR_{\aA}^{t\theta}$,
which defines a path from $\id$ to $\iota$. 
This path is called the natural path of $\iota$ and is denoted by 
\begin{align*}
  L^{\iota}\coloneq \gamma([0,1]) \subset \Isom(\R^n).
\end{align*}

A path $P$ is called a polygonal path if it can be obtained as the concatenation of finitely many natural paths of simple isometries. 
Equivalently, there exist simple isometries $\iota_1,\ \dots,\ \iota_m \in \II(\R^n)$ such that $P = L^{\iota_m} \circ \cdots \circ L^{\iota_1}$, 
and each path $L^{\iota_j}$ is regarded as a segment of the polygonal path.

For $V \in \GG^2(\R^n)$, recall the $V\times\R$ parametrization given in \autoref{def:labelby_V}, we can regard simple isometries in $\Isom(V)$ as vectors in $V\times\R$.
Then any collection of parallel line segments in $V\times\R$ can be interpreted as a collection of simple isometry in $\Isom(V)$ with the same projective center.
Thus, a polygonal isometry path $P$ in $\Isom(V)$ can be visualized as a collection of line segments in $V\times\R$.
We use $\LL$ for a collection of parallel line segments. These ``line segment'' notations will be used to visualize the construction of paths in $\Isom(V)$.

For an isometry path $P \subset \Isom(\R^n)$ from $\id$ to $\iota \in \Isom(\R^n)$, 
the set covered during the procedure of moving $E$ to $\iota(E)$ along $P$ is $\bigcup_{p \in P}p(E)$.

\section{Useful lemmas on moving $(n-1)$-rectifiable sets}\label{sect:lem}
In this section we give some basic lemmas on moving an $(n-1)$-rectifiable set by simple isometries,
which will be used in \autoref{sect:Kak_isom} to prove \autoref{thm:global_Kak_isom} and \autoref{thm:local_Kak_isom}.  
We state our lemmas first and leave the proofs in the latter part of this section.

First we give two lemmas to estimate the Lebesgue measure of the set covered by moving an $(n-1)$-rectifiable set along a simple isometry.
The $V\times\R$-parametrization defined in \autoref{def:labelby_V} turns a simple isometry into a vector, whose length can be used to measure the effect of this simplified isometry.
Recall that for a simple isometry $\iota$, we use $L^{\iota}$ for the natural path from $\id$ to $\iota$.

\begin{lemma}[trivial measure estimate for simple isometry]\label{lem:measisom}
  Let $V \in \GG^2(\R^n)$ and $\iota \in \Isom(V)$ be a simple isometry. 
  Suppose $E \subset B(0,r)$ is a bounded $(n-1)$-rectifiable set in $\R^n$,
  then the Lebesgue measure of the set covered by moving $E$ along $\iota$ can be estimated as
  \begin{align}\label{eq:lem_measisom}
    \abs{\bigcup_{p \in L^{\iota}} p(E)} \lesssim_r \abs{\eta(\iota)} \HH^{n-1}(E),
  \end{align}
  where $\eta(\iota) \in V\times \R$ is defined in \eqref{eq:nonproj_center} and can be used to measure the effect of moving along $\iota$.
\end{lemma}

\begin{lemma}[small measure estimate for simple isometry]\label{lem:smallmeasisom}
  Let $V \in \GG^2(\R^n)$ and $\iota \in \Isom(V)$ be a simple isometry.
  Suppose $E \subset B(0,r)$ is a bounded $(n-1)$-rectifiable set in $\R^n$, and $\delta >0$ be sufficiently small (depending on $r$).
  Let $\aA_{\iota} \in \GG^{n-2}(\P^n)$ be the projective axis of $\iota$, $n_x \in \AA^1(\R^n)$ be the normal line at $x \in E$ and $\nu_x=[n_x:1]$ be the corresponding projective normal line.
  Denoting 
  \begin{align*}
    E^{\iota, \delta} \coloneq \{x \in E \colon \nu_x \cap B_{\P^n}(\aA_{\iota},\delta) \neq \emptyset\} ,
  \end{align*}
  then we have
  \begin{align}\label{eq:thmr3.1.3}
    \abs{\bigcup_{p \in L^{\iota}} p(E^{\iota, \delta})} 
     \lesssim_r \delta \abs{\eta(\iota)} \HH^{n-1}(E^{\iota, \delta}).
  \end{align}
\end{lemma}

\begin{remark}\label{rmk:zero_isom_P_V^perp}
  Let $N \colon E \to \P^{n-1}$ be the normal map of $E$. Notice that 
  \begin{align*}
    \P^{n-3}_{V^{\perp}\times\{0\}} = \bigcap_{\iota \in \Isom(V)} \bigcap_{\delta>0} B_{\P^n}(\aA_{\iota},\delta)
    \implies
    N^{-1}(\P^{n-3}_{V^{\perp}}) = \bigcap_{\iota \in \Isom(V)} \bigcap_{\delta>0} E^{\iota,\delta},
  \end{align*}
  where we identify $\P^{n-3}_{V^{\perp}\times\{0\}} \subset \P^{n-1}_{\infty}$ with the direction set $\P^{n-3}_{V^{\perp}} \subset \P^{n-1}$.
  Thus, applying \autoref{lem:smallmeasisom} and letting $\delta \to 0$ yields
  \begin{align*}
    \abs{\bigcup_{p \in L^{\iota}} p(N^{-1}(\P^{n-3}_{V^{\perp}}))} = 0, \quad \forall \iota \in \Isom(V).
  \end{align*}
\end{remark}

Finally, we give a simple lemma which essentially shows that small perturbation of a path of isometries can only cause small change to the measure of the covered set.
This result is the higher-dimensional counterpart of \cite[Lemma 2.1]{ChCs2019}; since the proof is analogous, we omit it here.

\begin{lemma}[small neighborhood lemma for isometries]\label{lem:smallneighborhood}
  Suppose $V \in \GG^2(\R^n)$, $E \subset \R^n$ is a compact set, and $P \subset \Isom(V)$ is a path of isometries. 
  For any $\epsilon >0$, there exist $\delta >0$ and a neighborhood $B_{\Isom(V)}(P,\delta)$ of path $P$ in $\Isom(V)$, 
  such that for all path $P' \subset B_{\Isom(V)}(P,\delta)$, 
  we have 
  \begin{align*}
    \abs{\bigcup_{p \in P'}p(E)} \le \abs{\bigcup_{p \in P}p(E)} + \epsilon.
  \end{align*}
\end{lemma}

Next we prove \autoref{lem:measisom} and \autoref{lem:smallmeasisom}.
First we introduce a basic lemma that reflects the geometry nature of moving an $(n-1)$-rectifiable set along a simple isometry.
This result depends on two structures: the normal information of $E$ (i.e., the $\HH^{n-1}$-a.e. well-defined normal directions and normal lines),
and the geometric information of the simple isometry (i.e., the direction of the translation or the axis of the simple rotation). 

\begin{lemma}[basic measure estimate for moving $(n-1)$-rectifiable set]\label{lem:bettermeas_isom}
  Let $V \in \GG^2(\R^n)$, $z \in V$ and $\theta \in \R$.
  Let $\tau_z$ be a translation with direction $\wh{z} = \frac{z}{\abs{z}} \in \S^{n-1}$, and $\RR^\theta_{z,V}$ be a simple rotation with axis $z+V^{\perp}$.
  For $E \in \Rect$ and $x \in E$, let $\wh{N}(x) \in \S^{n-1}$ and $n_x \in \AA^1(\R^n)$ be the classical normal direction and normal line to $E$ at $x$, respectively.
  Then the measure of the set covered by moving $E$ along simple isometries can be estimated as
  \begin{numcases} {\abs{\bigcup_{p \in L^{\iota}} p(E)} \le}
    \abs{z} \int_{E} \abs{\inner{\wh{z}, \wh{N}(x)}}  \dd \HH^{n-1}(x), & if $\iota = \tau_z$, \label{eq:l2.0.1} \\
    \abs{\theta} \int_{E} \abs{P_V(\wh{N}(x))}\cdot d_{\R^n}(n_x, z+V^\perp)  \dd \HH^{n-1}(x), & if $\iota = \RR^{\theta}_{z,V}$. \label{eq:l2.0.2}
  \end{numcases}

\end{lemma}

\begin{proof}
  Let $\iota_1 = \tau_z$ and $\iota_2 = \RR^\theta_{z,V}$.
  Without loss of generality, we can assume that $\theta \in [0, 2\pi]$ and only consider the regular part $\wt{E}=\bigsqcup_{i=1}^{\infty} E_i$.
  We can further assume that each $E_i$ is parameterized by 
  $$x: U_i \subset \R^{n-1} \to E_i,\  s \mapsto x(s).$$
  Fix $i \ge 1$ and define 
  \begin{align*}
    \Phi_1 &\colon U_i\times [0,1] \to \R^n,\  (s,t) \mapsto \tau_{tz}(x(s)),  \\
    \Phi_2 &\colon U_i\times [0,1] \to \R^n,\  (s,t) \mapsto \RR^{t\theta}_{z,V}(x(s)), 
  \end{align*}
  which are Lipschitz mappings from $U_i\times [0,1]$ to $\R^n$.
  From the area formula (see, for example, \cite[Theorem 3.2.3]{Fed1969}), for $j= 1,2$, we have
  \begin{align}\label{eq:l2.1}
    \int_{\R^n} \mathrm{N}(\Phi_j,y) \dd \HH^{n}(y) = \int_{U_i\times [0,1]} \mathrm{J}_n \Phi_j(s,t) \dd (s,t),  
  \end{align}
  where $\mathrm{N}(\Phi,y) = \# \{(s,t) \in U_i\times [0,1] \colon \Phi(s,t)=y \}$ is the multiplicity function, 
  and the left-hand side corresponds to the swept-out measure of moving $E_i$ along $\iota_j$, counting multiplicity.
  The Jacobian appeared on the right-hand side is 
  \begin{align}\label{eq:l2.2}
    \mathrm{J}_n \Phi_j(s,t) = \abs{\det((\mathrm{D}\Phi_j)^{\Tt}_{(s,t)} \mathrm{D}\Phi_{(s,t)})}^{\frac{1}{2}} =\abs{\det((\mathrm{D}\Phi_j)_{(s,t)})},
  \end{align}
  and the linear mapping $(\mathrm{D}\Phi_j)_{(s,t)}$ is the derivative of $\Phi_j$ at $(s,t)$.

  Next we need to calculate the term $\abs{\det((\mathrm{D}\Phi_j)_{(s,t)})}$. For $(\tilde{s},\tilde{t}) \in \R^{n-1} \times \R$, from 
  \begin{align*}
    \Phi_1(s,t) &= \tau_{tz}(x(s)) = x(s)+tz, \\
    \Phi_2(s,t) &=\RR^{t\theta}_{z,V} (x(s)) =z+ \RR^{t\theta}_{0,V}(x(s)-z),
  \end{align*} 
  simple computation yields that
  \begin{align*}
    (\mathrm{D}\Phi_1)_{(s,t)} (\tilde{s},\tilde{t})
    &= \frac{\partial \Phi_1}{\partial s}\tilde{s} +\frac{\partial \Phi_1}{\partial t}\tilde{t} 
    = \frac{\partial x}{\partial s} \tilde{s} + z \tilde{t} 
    \eqcolon L_1 (\tilde{s},\tilde{t}), \\
    (\mathrm{D}\Phi_2)_{(s,t)} (\tilde{s},\tilde{t})
    &= \frac{\partial \Phi_2}{\partial s}\tilde{s} +\frac{\partial \Phi_2}{\partial t}\tilde{t} 
    = \RR^{t\theta}_{0,V} \Bg{\frac{\partial x}{\partial s} \tilde{s} + \theta J \bg{x(s)-z}\tilde{t}} 
    \eqcolon \RR^{t\theta}_{0,V} \circ L_2 (\tilde{s},\tilde{t}), 
 \end{align*}
 where $J=\RR^{\frac{\pi}{2}}_{0,V} \circ P_V $,
 $L_1 (\tilde{s},\tilde{t})
    = \begin{pmatrix} \frac{\partial x}{\partial s} &  z \end{pmatrix} 
      \begin{pmatrix}
        \tilde{s}\\ \tilde{t}
      \end{pmatrix}$, 
  and
  $L_2(\tilde{s},\tilde{t}) 
    = \begin{pmatrix} \frac{\partial x}{\partial s} &  \theta J\bg{x(s)-z} \end{pmatrix} 
      \begin{pmatrix}
        \tilde{s}\\ \tilde{t}
      \end{pmatrix}$.
Combining these with \eqref{eq:l2.2}, and notice that $\RR^{t\theta}_{0,V}$ corresponds to a matrix with determinant 1, we have
\begin{numcases} {\mathrm{J}_n \Phi_j(s,t) = \abs{\det((\mathrm{D}\Phi_j)_{(s,t)})} = \abs{\det(L_j)}= }
    \abs{\det \begin{pmatrix} \frac{\partial x}{\partial s} &  z \end{pmatrix} }, & if $j=1$, \label{eq:l2.3.1} \\
    \abs{\det \begin{pmatrix} \frac{\partial x}{\partial s} &  \theta J\bg{x(s)-z} \end{pmatrix} }, & if $j=2$. \label{eq:l2.3.2}
\end{numcases}

 Let $\{u_1,\dots,u_{n-1}\}$ be an orthonormal basis of the tangent hyperplane of $E_i$ at $x(s)$, and let $u_n$ be a unit vector parallel to $\wh{N}(x(s))$; these vectors form an orthonormal basis of $\R^n$.
 From the fact that the column vectors of $\frac{\partial x}{\partial s}$ generate the tangent hyperplane of $E_i$ at $x(s)$,
 we can find an $(n-1)\times(n-1)$ invertible matrix $A(s)$ such that
  \begin{equation*}
    \frac{\partial x}{\partial s} 
    = \begin{pmatrix}
        u_1& \cdots & u_{n-1}      
    \end{pmatrix} 
    A(s),
  \end{equation*}
  which also implies
   \begin{equation*}
   \abs{\det A(s)} = \sqrt{A(s)^\Tt A(s)} = \sqrt{\Bg{\frac{\partial x}{\partial s}}^\Tt \frac{\partial x}{\partial s} } = \mathrm{J}_{n-1} x(s),
  \end{equation*}
  where $\mathrm{J}_{n-1} x(s)$ is the $(n-1)$-Jacobian of the parametrization $x$.
  Next, since $u_1,\ \cdots,\ u_{n-1},\ u_n$ form an orthonormal basis of $\R^n$, there exist $a = (a_1, \dots, a_{n-1}),\ b = (b_1, \dots, b_{n-1}) \in \R^{n-1}$ such that
  \begin{align*}
     z &= \sum_{i=1}^{n-1} a_iu_i +\inner{z, u_n}u_n, \\
     J\bg{x(s)-z} &= \sum_{i=1}^{n-1} b_iu_i +\inner{J\bg{x(s)-z}, u_n}u_n.
  \end{align*}
  With these equations, \eqref{eq:l2.3.1} and \eqref{eq:l2.3.2} can be simplified as
  \begin{align}
    \mathrm{J}_n \Phi_1(s,t)
    &= \abs{\det
      \begin{pmatrix} A(s)  &  a^{\Tt}  \\ 
                    0_{n-1} & \inner{z, u_n}\end{pmatrix}} 
    = \abs{\inner{ z, u_n }} \cdot \mathrm{J}_{n-1} x(s), \label{eq:l2.4.1} \\
    \mathrm{J}_n \Phi_2(s,t)
    &= \abs{\det
      \begin{pmatrix} A(s)  &  \theta b^{\Tt}  \\ 
                    0_{n-1} & \theta \inner{J\bg{x(s)-z}, u_n}\end{pmatrix}}
    =  \theta \abs{\inner{ J\bg{x(s)-z}, u_n }} \cdot \mathrm{J}_{n-1} x(s). \label{eq:l2.4.2}
  \end{align}
  Finally, applying \eqref{eq:l2.4.1} to \eqref{eq:l2.1}, the measure (counting by multiplicity) of the set covered by translating is
  \begin{align} 
    \int_{\R^n} \mathrm{N}(\Phi_1,y) \dd \HH^{n}(y)   
    &= \int_{U_i\times [0,1]} \abs{\inner{ z, \wh{N}(x(s)) }} \cdot \mathrm{J}_{n-1} x(s) \dd (s,t) \notag \\
    &= \int_{U_i} \abs{\inner{ z, \wh{N}(x(s)) }} \cdot \mathrm{J}_{n-1} x(s) \dd s 
    = \int_{E_i} \abs{\inner{ z, \wh{N}(x) }}  \dd \HH^{n-1}(x),  \label{eq:l2.5.1}
  \end{align}
  where the last equality uses the area formula for the parametrization $x$.
  Similarly, applying \eqref{eq:l2.4.2} to \eqref{eq:l2.1} leads to
  \begin{align} \label{eq:l2.5.2}
   \int_{\R^n} \mathrm{N}(\Phi_2,y) \dd \HH^{n}(y)  = \theta \int_{E_i} \abs{\inner{ J\bg{x(s)-z}, u_n }} \dd \HH^{n-1}(x).
  \end{align}
  Notice that for the translation $\iota_1 = \tau_z$ with direction $\wh{z} = \frac{z}{\abs{z}} \in \S^{n-1}$, we have 
  \begin{align} \label{eq:l2.5.3}
    \abs{\inner{z, \wh{N}(x)}} = \abs{z}\abs{\inner{\wh{z}, \wh{N}(x)}}.
  \end{align}
  For the rotation $\iota_2 =\RR^\theta_{z,V}$, recall that $J=\RR^{\frac{\pi}{2}}_{0,V} \circ P_V $, we have
  \begin{align} 
    \abs{\inner{J(x-z), \wh{N}(x)}} 
    &= \abs{\inner{\RR^{\frac{\pi}{2}}_{0,V} \circ P_V(x-z), P_V(\wh{N}(x))}}  \notag \\
    &= \abs{P_V(\wh{N}(x))} \abs{P_V(x-z)} \abs{\sin \Bg{\angle\bg{P_V(x-z),P_V(\wh{N}(x))}}}  \notag \\
    &= \abs{P_V(\wh{N}(x))} d_{\R^n} (P_V(n_x), z+V^\perp) \le \abs{P_V(\wh{N}(x))}\cdot d_{\R^n} (n_x, z+V^\perp). \label{eq:l2.5.4}
  \end{align}
  Applying  \eqref{eq:l2.5.3} and \eqref{eq:l2.5.4} to \eqref{eq:l2.5.1} and \eqref{eq:l2.5.2}, respectively,
  and notice that $\abs{\bigcup_{p \in L^{\iota_j}} p(E_i)} \le \int_{\R^n} \mathrm{N}(\Phi_j,y) \dd \HH^{n}(y)$ holds for $j=1,2$, 
  we have 
  \begin{align*}
    \abs{\bigcup_{p \in L^{\iota}} p(E)}  
    \le \sum_{i=1}^{\infty} \abs{\bigcup_{p \in L^{\iota}} p(E_i)}  
    \le
      \begin{cases}
        \displaystyle \abs{z} \int_{E} \abs{\inner{\wh{z}, \wh{N}(x)}}  \dd \HH^{n-1}(x), & \text{if}\ \iota = \tau_z, \\
        \displaystyle \abs{\theta} \int_{E} \abs{P_V(\wh{N}(x))}\cdot d_{\R^n}(n_x, z+V^\perp)  \dd \HH^{n-1}(x), & \text{if}\ \iota = \RR^{\theta}_{z,V},
      \end{cases}
  \end{align*}
  which proves \eqref{eq:l2.0.1} and \eqref{eq:l2.0.2}.
\end{proof}

Now we turn to the proof of \autoref{lem:measisom} and \autoref{lem:smallmeasisom}.

\begin{proof}[\textbf{Proof of \autoref*{lem:measisom}}]
  For $\iota = \tau_z$ with $z \in V$, we have
  \begin{align}\label{eq:lengthcheck_trans}
     \eta(\tau_z) = (\RR_{0,V}^{\frac{\pi}{2}}(z),0) \ \implies \abs{\eta(\tau_z)}=\abs{z},
  \end{align}
  then we can use \eqref{eq:l2.0.1} in \autoref{lem:smallmeasisom} to get 
  \begin{align}\label{eq:lem_measisom1}
    \abs{\bigcup_{p \in L^{\iota}} p(E)} \le
    \abs{z} \int_{E} \abs{\inner{\wh{z}, \wh{N}(x)}}  \dd \HH^{n-1}(x) \le \abs{z} \HH^{n-1}(E) = \abs{\eta(\iota)} \HH^{n-1}(E).
  \end{align} 

  For $\iota = \RR^{\theta}_{z,V}$ with $z \in V$, 
  we have 
  \begin{align}\label{eq:lengthcheck_rotat}
    \eta(\RR^{\theta}_{z,V}) = (\theta z,\theta) \ \implies \abs{\eta(\RR^{\theta}_{z,V})} = \abs{\theta}\sqrt{1+\abs{z}^2} \approx \abs{\theta}(1+\abs{z}).
  \end{align}
  Then \eqref{eq:l2.0.2} in \autoref{lem:bettermeas_isom} leads to
  \begin{align}\label{eq:lem_measisom2}
    \abs{\bigcup_{p \in L^{\iota}} p(E)} \lesssim \abs{\theta} (r+\abs{z})\HH^{n-1}(E) \lesssim_r \abs{\theta}(1+\abs{z}) \HH^{n-1}(E) \lesssim_r \abs{\eta(\iota)} \HH^{n-1}(E).
  \end{align}
  
  Finally, the combination of \eqref{eq:lem_measisom1} and \eqref{eq:lem_measisom2} proves \eqref{eq:lem_measisom}.
\end{proof}

\begin{proof}[\textbf{Proof of \autoref*{lem:smallmeasisom}}]
  Let $\theta \in \R$ and $z \in V$. 
  To prove \eqref{eq:thmr3.1.3} in \autoref{lem:smallmeasisom}, we only need to prove
  \begin{numcases} {\abs{\bigcup_{p \in L^{\iota}} p(E^{\iota, \delta})} \le}
     C_n \delta \abs{z} \HH^{n-1}(E^{\iota, \delta})  , & if $\iota = \tau_z$, \label{eq:thmr3.1.2} \\
     C_{n, r} \delta \abs{\theta} (1 + \abs{z}) \HH^{n-1}(E^{\iota, \delta}) , & if $\iota = \RR^{\theta}_{z,V}$, \label{eq:thmr3.1.1}
  \end{numcases}
  where the constants $C_n, C_{n,r} >1$.
  Then, applying \eqref{eq:lengthcheck_trans} and \eqref{eq:lengthcheck_rotat} to \eqref{eq:thmr3.1.1} and \eqref{eq:thmr3.1.2}, respectively, leads to \eqref{eq:thmr3.1.3}.
  We now focus on the proof of \eqref{eq:thmr3.1.1} and \eqref{eq:thmr3.1.2}.

  Let $\wh{N}(x) \in \S^{n-1}$ be the classical normal direction at $x \in E$, and $N(x) = \pi_0(\wh{N}(x)) \in \P^{n-1}$.
  For simplicity, we use
  \begin{align*}
    B_{\P^n_{\aff}}(\aA_{\iota},\delta) \coloneq B_{\P^n}(\aA_{\iota},\delta) \cap \P^n_{\aff}, 
    \quad B_{\P^{n-1}_{\infty}}(\aA_{\iota},\delta) \coloneq B_{\P^n}(\aA_{\iota},\delta) \cap \P^{n-1}_{\infty}
  \end{align*}
  for the finite part and the part at infinity of $B_{\P^n}(\aA_{\iota},\delta)$, respectively.

  \textbf{$\bullet$ Case 1:}
  Suppose $\iota = \tau_z$ with $z \in \R^n\setminus\{0\}$, then $\aA_{\iota} = [z^\perp :0] \subset \P^{n-1}_{\infty}$,
  i.e., the projective axis is totally contained in the part at infinity of $\P^n$.
  Let $\wh{z} = \frac{z}{\abs{z}} \in \S^{n-1}$, and $\tilde{z} = [z] \in \P^{n-1}$ be the projective direction of $z \in \R^n\setminus \{0\}$.

  To prove \eqref{eq:thmr3.1.2}, we claim 
  \begin{align}\label{eq:r3.1.1}
    E^{\iota,\delta} \subset N^{-1}(B_{\P^{n-1}}(\tilde{z}^\perp, C_n\delta)),
  \end{align} 
  for sufficiently small $\delta>0$ (depending on $r$),
  where $N \colon E \to \P^{n-1}$ is the normal map of $E$, and $C_n>1$ is a constant.
  Thus, for $x \in E^{\iota,\delta}$ with classical normal direction $\wh{N}(x) \in \S^{n-1}$, 
  \eqref{eq:r3.1.1} implies 
  \begin{align*}
    \abs{\inner{\wh{z}, \wh{N}(x)}} = \abs{\cos \angle(\wh{z}, \wh{N}(x))} =\cos (d_{\P^{n-1}}(\wt{z},N(x))) \le \sin (C_n\delta) \lesssim C_n\delta.
  \end{align*}
  Then combining with \eqref{eq:l2.0.1} in \autoref{lem:bettermeas_isom}, we have
  \begin{align*}
    \abs{\bigcup_{p \in L^{\iota}} p(E^{\iota,\delta})}  
    \le \abs{z} \int_{E^{\iota,\delta}} \abs{\inner{\wh{z}, \wh{N}(x)}}  \dd \HH^{n-1}(x)
    \lesssim C_n \delta \abs{z} \HH^{n-1}(E^{\iota,\delta}), 
  \end{align*}
  which proves \eqref{eq:thmr3.1.2}.

  Now we prove the claim \eqref{eq:r3.1.1}.
  Let's divide $E^{\iota,\delta}$ as $E^{\iota,\delta} = E_1\cup E_2$, where
  \begin{align*}
    E_1 &\coloneq \{x \in E \colon \nu_x \cap B_{\P^n_{\aff}}([z^\perp :0],\delta) \neq \emptyset\},  \\
    E_2 &\coloneq \{x \in E \colon \nu_x \cap B_{\P^{n-1}_{\infty}}([z^\perp :0],\delta) \neq \emptyset\}.
  \end{align*}
  For the set $E_1$, noticing that $B_{\P^n_{\aff}}([z^\perp :0],\delta) = \{[y:1] \colon y\in A(z^\perp,\delta)\}$, 
  where $A(z^\perp,\delta) \subset \R^n$ is far away from the origin and its points have projective direction in  $B_{\P^{n-1}}(\tilde{z}^\perp,\delta)$. 
  Since $E \subset B(0,r)$, for sufficiently small $\delta$ (depending on $r$), we can make $A(z^\perp,\delta)$ far away from $E$, 
  such that 
  \begin{align*}
    \nu_x \cap B_{\P^n_{\aff}}([z^\perp :0],\delta) \neq \emptyset 
    \implies n_x \cap A(z^\perp,\delta) \neq \emptyset
    \implies N(x) \in B_{\P^{n-1}}(\tilde{z}^\perp, C_n\delta),
  \end{align*}
  where $C_n >1$ doesn't depend on $r$.
  Thus, for sufficiently small $\delta$, we have 
  \begin{align}\label{eq:r3.1.2}
    E_1 
    \subset \{x \in E \colon N(x) \in B_{\P^{n-1}}(	\tilde{z}^\perp, C_n\delta)\}
    = N^{-1}(B_{\P^{n-1}}(\tilde{z}^\perp, C_n\delta)).
  \end{align}
  For the set $E_2$, using the fact $B_{\P^{n-1}_{\infty}}([z^\perp :0],\delta) = \pi_0(B_{\S^{n-1}}(\wh{z}^\perp, \delta))$,
  we have 
  \begin{align*}
    \nu_x \cap B_{\P^{n-1}_{\infty}}([z^\perp :0],\delta) \neq \emptyset
    \iff \wh{N}(x) \in B_{\S^{n-1}}(\wh{z}^\perp,\delta)
    \iff N(x) \in B_{\P^{n-1}}(\tilde{z}^\perp, \delta),
  \end{align*}
  which easily leads to 
  \begin{align}\label{eq:r3.1.3}
    E_2 = N^{-1}(B_{\P^{n-1}}(\tilde{z}^\perp, \delta)) \subset N^{-1}(B_{\P^{n-1}}(\tilde{z}^\perp, C_n\delta)).
  \end{align}
  Then the combination of \eqref{eq:r3.1.2}, \eqref{eq:r3.1.3} directly leads to \eqref{eq:r3.1.1}.

  \textbf{$\bullet$ Case 2:}
  Suppose $\iota = \RR^{\theta}_{z,V}$, then the projective axis is $\aA_{\iota} = [z+V^\perp :1]$.
  To prove \eqref{eq:thmr3.1.1}, we claim
  \begin{align}\label{eq:r3.1}
    x \in E^{\iota,\delta}  \implies \abs{P_V(\wh{N}(x))} d_{\R^n}(P_V(n_x), z+V^\perp) \le C_{n,r} \delta (1+\abs{z})
  \end{align}
  for some constant $C_{n,r} >1$.
  Combining \eqref{eq:r3.1} with \eqref{eq:l2.0.2} in \autoref{lem:bettermeas_isom}, we have 
  \begin{align*}
    \abs{\bigcup_{p \in L^{\iota}} p(E^{\iota,\delta})}
    \le \abs{\theta} \int_{E} \abs{P_V(\wh{N}(x))} d_{\R^n}(P_V(n_x), z+V^\perp)  \dd \HH^{n-1}(x) 
    \le C_{n,r} \delta \abs{\theta}   (1+\abs{z}) \HH^{n-1}(E),
  \end{align*}
  which proves \eqref{eq:thmr3.1.1}.

  Next we prove the claim \eqref{eq:r3.1}.
  First we try to simplify the set 
  \begin{align*}
    E^{\iota,\delta} = \{x \in E \colon \nu_x \cap B_{\P^n}([z+V^\perp :1],\delta) \neq \emptyset \}.
  \end{align*}
  From \eqref{eq:equation1}, we know that $[z+V^\perp :1] = \{[u :1] \colon u \in z+V^\perp\} \cup [V^\perp :0]$, 
  which leads to
  \begin{align} \label{eq:r3.2}
    B_{\P^n}([z+V^\perp :1],\delta) 
    = \Big(\bigcup_{u \in z+V^\perp} B_{\P^n}( [u :1], \delta)\Big) \cup B_{\P^n}([V^\perp :0], \delta) .
  \end{align}
  Let $\delta$ be sufficiently small (depending on $r$) as explained in Case 1.
  Since $B_{\P^n}([V^\perp :0], \delta)$ is the neighborhood of points at infinity, 
  we can find a constant $D_r$ depending on $\delta$ (and hence on $r$),
  such that $B_{\P^n}([u :1], \delta) \subset B_{\P^n}([V^\perp :0], \delta)$ for all $u \in z+V^\perp$ with $\abs{u} > D_r$.
  Then \eqref{eq:r3.2} can be simplified as
  \begin{align*}
    B_{\P^n}([z+V^\perp :1],\delta) = \Big(\bigcup_{u \in z+V^\perp, \abs{u} \le D_r} B_{\P^n}( [u :1], \delta)\Big) \cup B_{\P^n}([V^\perp :0], \delta),
  \end{align*}
  which leads to a decomposition of $E^{\iota,\delta}$ denoted by $E^{\iota,\delta} = E_3 \cup E_4$, where
  \begin{align*}
     E_3 &\coloneq \bigcup_{u \in z+V^\perp, \abs{u} \le D_r} \{x \in E \colon \nu_x \cap B_{\P^n}( [u :1], \delta) \neq \emptyset\} , \\
     E_4 &\coloneq \{x \in E \colon \nu_x \cap B_{\P^n}([V^\perp :0], \delta) \neq \emptyset\}. 
  \end{align*}
  
  For $x \in E_4$, since $B_{\P^n}([V^\perp :0], \delta)$ is the projective neighborhood of points at infinity, 
  with similar discussion as the Case 1, 
  we can find a constant $C_n$ (the same as Case 1) such that $E_4 \subset N^{-1}(B_{\P^{n-1}}([V^\perp :0], C_n\delta))$, 
  which further leads to
  \begin{align}\label{eq:r3.3}
    \abs{P_V(\wh{N}(x))} \lesssim C_n\delta, \quad  \forall x\in E_4.
  \end{align}
  
  For $x \in E_3$, we can find $u \in z+V^\perp$ with $\abs{u} \le D_r$ such that
  \begin{align}\label{eq:r3.4}
    \nu_x \cap B_{\P^n}( [u :1], \delta) \neq \emptyset.
  \end{align}
  From the construction of projective space, we know that $B_{\P^n}( [u :1], \delta) = \{[w:1] \colon w \in U_u\}$, 
  where $U_u \subset \R^n$ is a bounded neighborhood satisfying $U_u \subset B_{\R^n}(u, c_u \delta)$ 
  for a constant $c_u > 0$ depending on $u$ ($c_u$ is determined by the local Lipschitz constant of the map $x\mapsto[x:1]$ at $u$).
  Then we have 
  \begin{align*}
    \eqref{eq:r3.4} \iff n_x \cap U_u \neq \emptyset \implies n_x \cap B_{\R^n}(u, c_u \delta) \neq \emptyset 
    \implies n_x \cap B_{\R^n}(z+V^\perp, c_u \delta) \neq \emptyset,
  \end{align*}
  which implies $d_{\R^n}(n_x, z+V^\perp) \le c_u\delta$ for $u \in z+V^\perp$ with $\abs{u} \le D_r$. 
  By taking 
  \begin{align*}
    c_r \coloneq \sup\{c_u \colon \abs{u} \le D_r\} \lesssim_r 1,
  \end{align*}  
  we have 
  \begin{align}\label{eq:r3.5}
    d_{\R^n}(n_x, z+V^\perp) \le c_r\delta, \quad \forall x \in E_3.
  \end{align}
  
  The combination of \eqref{eq:r3.3} and \eqref{eq:r3.5} naturally leads to
  \begin{align*}
    \abs{P_V(\wh{N}(x))} d_{\R^n}(n_x, z+V^\perp)
    \le c_r  \delta + C_n \delta (r+\abs{z})
    \le C_{n,r} \delta (1+\abs{z})
  \end{align*}
  for $x \in E^{\iota,\delta}$, which proves the claim \eqref{eq:r3.1}.
\end{proof}

\section{Venetian blind-type construction and zigzag path}\label{sect:VBandzigzag}

In this section we describe the Venetian blind-type construction, which was introduced in \cite[Section 4, Section 5]{ChCs2019} for planar isometries.
We adapt their method to construct paths in $\Isom(V)$ for any $V \in \GG^2(\R^n)$. 
The classical Venetian blind construction is related to the projection problems and Nikodym sets, which can be found in Falconer's paper \cite{Fal1986}, see also \cite[Theorem 6.9]{Fal1990FractalGM} and \cite[Lemma 11.8]{Mat2015} for related discussions.

\subsection{Zigzag path in $\Isom(V)$}\label{ssect:zigzag_isom}

Roughly speaking, the Venetian blind-type construction is an iteration method of 
replacing a line segment of a given direction by many families of shorter parallel line segments with different directions.
The newly constructed line segments stay in a small neighborhood of the original segment, and their lengths and directions are well controlled.
With the $V\times\R$-parametrization of $\Isom(V)$, see \autoref{def:labelby_V}, we can interpret the line segments as simple isometries in $\Isom(V)$,
and then construct paths in $\Isom(V)$. 

We now fix some notation used in the construction.
For simplicity, we use line segment families labeled by binary indices to describe the construction. 
Let $\{0,1\}^*$ be the collection of all finite sequences of ``0'' and ``1'', 
whose elements are binary indices denoted by boldface letters ``$\mathbf{i, j, k, \cdots}$''.
By convention, $\varnothing$ is used for the empty sequence. 
The notation $\abs{\ii}$ denotes the length of $\ii$ (i.e., the number of ``0''s and ``1''s in $\ii$), 
and $\abs{\ii}_0$ denotes the number of ``1''s in $\ii$.
The notation $\ii^{(k)}$ denotes the binary sequence obtained from $\ii$ by deleting its last $k$ symbols. 
Usually we use $\mathbf{i'}$ for $\ii^{(1)}$. 
We give an example to explain these notations, for $\ii = 0001011$, abbreviated as $\ii = 0^3101^2$, we have
\begin{align*}
  &\mathbf{i'} = \ii^{(1)} = 000101, \ \ii^{(3)} = 0001, \ \ii = \ii^{(0)} =\ii'1 = \ii^{(3)}011, \\
  &\abs{\ii} = 7, \ \abs{\ii}_0 = 3, \ \abs{\ii^{(7)}} =\abs{\varnothing}=0.
\end{align*} 

We can regard $\{0,1\}^*$ as a binary tree rooted at $\varnothing$, and introduce the natural partial order $\preceq$ as
\begin{align*}
  \ii \preceq \jj 
  &\iff  \jj \  \text{is in the subtree rooted at} \ \ii, \\
  &\iff \exists k \in \N, \ s.t. \ \ii = \jj^{(k)}.
\end{align*}
When $\ii \preceq \jj$, we say they lie on the same family line. 
We write $\ii \prec \jj$ for $\ii \preceq \jj$ and $\ii \neq \jj$.
Each $\ii$ has exactly one $k$-th ancestor and $2^k$ $k$-th children in $\{0,1\}^*$.
For any subset $D \subset \{0,1\}^*$, one can define its maximal set as
$$
\max D \coloneq \{\ii \in D  \colon \forall \jj \in D,\ \ii \preceq \jj \implies \ii = \jj \}.
$$ 

We say an index is a ``bad'' index if it ends with ``0'', otherwise we say it is ``good''. 
We also write $g_1(\ii)$ for a good index which is closest to $\ii$ among $\ii$ and its ancestors.
Similarly, $g_k(\ii)$ is defined as the good index which is $k$-th closest to $\ii$ among $\ii$ and its ancestors,
For example, we have $g_1(01010)=0101$, $g_2(01010)=01$ and $g_3(01010)=\varnothing$.

Let $\epsilon >0$ be a small parameter, and $\iota \in \Isom(V)\setminus\{\id\}$.
Our aim is to construct a path in $\Isom(V)$ to connect $\id$ and $\iota$.

From the $V\times\R$-parametrization, $\iota$ corresponds to a point $\eta(\iota) \in (V\times\R)\setminus\{0\}$.
Take a 2-plane $W \in \GG^2(V\times\R)$ passing through $\eta(\iota)$, 
we denote $l = \P^1_W$ to be the direction set of $W$, and $l \in \GG^1(\P^2_{V\times\R})$ is also a projective line.  
In the following constructions, we will replace any constructed line segment by shorter segments in its neighborhood. 
We show this procedure as three parts: 
$$
\text{basic zigzag} \xrightarrow{\text{iteration}} \text{Venetian blind zigzag} \xrightarrow{\text{iteration}} \text{zigzag path}. 
$$

\subsubsection{Basic zigzag}\label{ssect:basiczigzag}

Let $\sigma \in \{1,-1\}$, and $\gamma, \beta$ be angles satisfying
\begin{align}\label{eq:angle_initial}
  \gamma, \beta \in (0, \frac{\pi}{4}), \quad \gamma \le \min \{\beta, \frac{\pi}{2}-2\beta\}.
\end{align}
We first describe a geometric construction related to the basic zigzag.

Let $y \in V\times\R$, then
the subspace $\spn\{y,P_W(y)\}$ divides $\spn\{y,W\}$ into two half-spaces, the left part and the right part, which we distinguish by the sign parameter $\sigma \in \{1,-1\}$.
We define the oriented angle $\angle(y,z)$ to be positive for vectors $z$ in the left part and negative for those in the right part.

If $y$ satisfies
\begin{align}\label{eq:basiczigzag_angle_assumption}
  \angle(y, W) < \beta,
\end{align}
where $\angle(y, W)$ is the angle between $y$ and $W$,
then we can find $y_0 \in W$ with 
\begin{align}\label{eq:y_0}
  \angle(y,y_0) = -\sigma\beta.
\end{align}
With this fixed $y_0$ as the input, we can find $\wt{y}_1 \in V\times\R$, such that
\begin{align*}
  y = y_0+\wt{y}_1,\ \angle(y,\wt{y}_1) = \sigma\gamma.
\end{align*}
Then, by \autoref{lem:approx_N_term}, we can find $M$ sufficiently large and $y_1 \in V\times\R$ sufficiently close to $\wt{y}_1$, i.e.,
\begin{align*}
  \abs{y_1-\wt{y}_1} \le o_M(1),
\end{align*} 
such that
\begin{align}
  \eta^{-1}(y) =\eta^{-1}(y_0/M) \circ \eta^{-1}(y_1&/M) \circ \cdots \circ  \eta^{-1}(y_0/M) \circ \eta^{-1}(y_1/M), \label{item_a}  \\
  \angle(y,y_1) &= \sigma\gamma + o_M(1), \label{item_b} 
\end{align}
where the righthand side of \eqref{item_a} is a composition of $2M$ terms from right to left.
From \eqref{eq:y_0} and \eqref{item_b}, and the fact that $o_M(1)$ can be arbitrarily small, we know that $\{y,y_0,y_1\}$ almost forms a triangle with two interior angles $\beta$ and (almost) $\gamma$.

For the given $y$, fix a direction $\theta \in l$ associated with $y$ (we will determine it later).
We denote by $\LL(y,\theta)$ a collection of line segments parallel to the vector $y$, and assigned the direction $\theta \in l$.
Let $y_0 \in W$ be fixed with $\angle(y,y_0) = -\sigma\beta$.
From the above geometric construction,
we can find a sufficiently large integer $M$ and a vector $y_1 \in V\times\R$ to decompose $y$ by $y_0/M$ and $y_1/M$, as suggested by \eqref{item_a}.
We first divide each line segment $L \in \LL$ equally into $M$ shorter subsegments, 
and then replace each subsegment by two segments parallel to $y_{0}, y_{1}$, respectively, see \autoref{fig:zigzag0}. 
All subsegments parallel to $y_0$ are associated with direction $\theta_0 \coloneq \theta - \sigma \beta \in l$ and collected in $\LL(y_0,\theta_0)$. 
Similarly, $\LL(y_1,\theta_1)$ consists of subsegments parallel to $y_1$, assigned the direction $\theta_1= \theta + \sigma \gamma \in l$.

%%%%%%%%%  % `\fontsize{1}{1.2}\selectfont' to control small text, just put it befor `\smash'
\begin{figure}[htbp]
    \centering
    \def\svgwidth{\textwidth}
    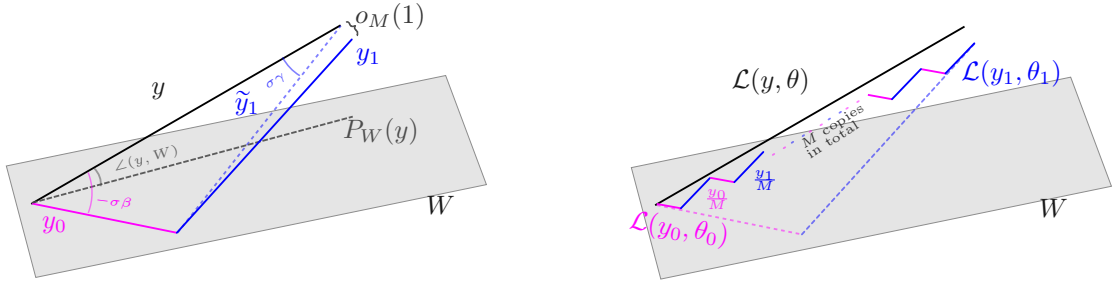
    \caption{Basic zigzag}
    \label{fig:zigzag0}
\end{figure}

%%%%%%%%%

This procedure is summarized as a $(\gamma, \beta, \sigma, M, y_{0})$-basic zigzag construction on $\LL(y,\theta)$, which we denote by the operator $Z_{\gamma, \beta, \sigma, M, y_{0}}$:
\begin{align*}
  \LL(y,\theta) \xlongrightarrow{Z_{\gamma, \beta, \sigma, M, y_{0}}} \LL(y_0,\theta_0) \cup \LL(y_1,\theta_1) \eqcolon Z_{\gamma, \beta, \sigma, M, y_{0}}(\LL(y,\theta)).
\end{align*}
The resulting collection of line segments has many ``Z''-turns, which explains the name ``Zigzag''.
The sufficiently large parameter $M$ is called the fineness parameter and to be determined later. 
It is worth noting that $Z_{\gamma, \beta, \sigma, M, y_{0}}(\LL(y,\theta))$ always stays in a neighborhood of $\LL$, 
and the neighborhood shrinks as the fineness $M$ increases.

For simplicity, we describe the following constructions mainly using the language of line segments in $V\times\R$,
and the results are always understood as isometry paths in $\Isom(V)$, just as the $V\times\R$ parametrization of $\Isom(V)$ suggests.

\subsubsection{Venetian blind zigzag}\label{ssect:Vb}

Let $\gamma, \beta, \sigma$ be the same parameters as stated above, 
and $k \in \N$ satisfy 
\begin{align*}
  \frac{\pi}{2}- 2\beta- \gamma < k\gamma \le \frac{\pi}{2}- 2\beta.
\end{align*}

Now for a fixed binary index $\ii \in \{0,1\}^*$, let $y_{\ii} \in V\times \R$, $\theta_{\ii} \in l$,
and $\LL(y_{\ii},\theta_{\ii})$ be a family of line segments parallel to $y_{\ii}$ and associated with $\theta_{\ii} \in l$.

We will iterate the basic zigzag construction many times on $\LL(y_{\ii},\theta_{\ii})$ to get two families $\LL(y_{\ii0},\theta_{\ii0})$ and $\LL(y_{\ii1},\theta_{\ii1})$ 
corresponding to line segments parallel to $y_{\ii0}$ and $y_{\ii1}$, respectively.
The vectors $y_{\ii0}, y_{\ii1} \in V\times\R$ and directions $\theta_{\ii0}, \theta_{\ii1} \in l$ will be determined in later discussion.

To describe the iterations of basic zigzags, 
we introduce an additional binary index $\jj$ to form a finer poset $\{(\ii, \jj)\}_{\jj}$ between $\ii$ and $\{\ii0,\ii1\}$.
More precisely, we introduce the finer poset as
\begin{align}\label{eq:finner_poset}
  \ii \coloneq \ii,\varnothing \prec 
    \begin{cases}
      \ii,0 \prec  \ii0  \\
      \ii,1 \prec 
        \begin{cases}
          \ii,10 \prec \ii0 \\
          \ii,11 \prec \cdots 
            \begin{cases}
              \ii,1^{k-1}0 \prec \ii0 \\
              \ii,1^{k} \prec \ii1,
            \end{cases}
        \end{cases}
    \end{cases} 
\end{align} 
where we call $\{(\ii,\jj)\}$ and $\{\ii\}$ the small indexes (for describing the basic zigzag iterations) 
and big indexes (for describing the Venetian blind iterations), respectively. 
As before, we say a small index $(\ii,\jj)$ is ``bad'' if $\jj$ ends with ``0'', and ``good'' otherwise.

We define $y_{\ii, \varnothing} = y_{\ii}$ and $\theta_{\ii,\varnothing} = \theta_{\ii}$.
If, similar to \eqref{eq:basiczigzag_angle_assumption}, we have
\begin{align}\label{eq:cond_VBzigzag}
  \angle(y_{\ii}, W) = \angle(y_{\ii, \varnothing}, W) < \beta,
\end{align}
then we can find $y_{\ii, 0} \in W$ such that $\angle(y_{\ii,\varnothing},y_{\ii, 0}) = -\sigma\beta$.

Iteratively, we use the $(\gamma, \beta, \sigma, M, y_{\ii, 0})$-basic zigzag construction introduced in \autoref{ssect:basiczigzag} at any index $(\ii,\jj)$ with $\jj \in \{\varnothing, 1, 11,\dots, 1^{k-1}\}$ (see \eqref{eq:finner_poset}), and choose different fineness parameters as we need.
In a typical iteration at $(\ii,\jj)$, we set $y_{\ii,\jj0} = y_{\ii,0}$, and choose $M_{\ii,\jj} \gg 1$ and $y_{\ii,\jj1}, \wt{y}_{\ii,\jj1} \in V\times\R$ such that $ y_{\ii,\jj} = y_{\ii,\jj0} +\wt{y}_{\ii,\jj1}$ and
\begin{align}\label{eq:vectorclose_j1}
\abs{y_{\ii,\jj1} - \wt{y}_{\ii,\jj1}} \le o_{M_{\ii,\jj}}(1),  
\end{align} 
then the basic zigzag iteration procedure  gives
\begin{align*}
  \LL(y_{\ii,\jj},\theta_{\ii,\jj}) \xlongrightarrow{Z_{\gamma, \beta, \sigma, M_{\ii,\jj}, y_{\ii, 0}}}  \LL(y_{\ii,\jj0},\theta_{\ii,\jj0}) \cup \LL(y_{\ii,\jj1},\theta_{\ii,\jj1}),
\end{align*}
where $y_{\ii,\jj0} =y_{\ii, 0}$, and 
\begin{align}
  &\theta_{\ii,\jj0} = \theta_{\ii,0} = \theta_{\ii,\varnothing} -\sigma \beta \label{eq:thetaangle_j_j0} \\
  &\theta_{\ii,\jj1} = \theta_{\ii,\jj} +\sigma \gamma=\cdots = \theta_{\ii,\varnothing} + \abs{\jj1}\sigma \gamma \label{eq:thetaangle_j_j1}\\ 
  &\angle (y_{\ii,\jj},y_{\ii,\jj1}) = \sigma\gamma +o_{M_{\ii,\jj}}(1). \label{eq:vectorangle_i_i1}
\end{align}

After the iterations described above, each resulting line segment parallels to either $y_{\ii,1^k}$ or $y_{\ii,0}$. 
Let 
\begin{align*}
  &y_{\ii0} \coloneq y_{\ii,0} , \quad \theta_{\ii0} \coloneq \theta_{\ii,0} = \theta_{\ii}-\sigma \beta, \\
  &y_{\ii1} \coloneq y_{\ii,1^{k}},\quad \theta_{\ii1} \coloneq \theta_{\ii,1^{k}} =\theta_{\ii}+k\sigma \gamma.
\end{align*}
Then the resulting line segments of the above iterations can be divided into two new families:
\begin{align*}
  \LL(y_{\ii0},\theta_{\ii0}) &\coloneq \LL(y_{\ii, 0},\theta_{\ii, 0})\cup \cdots \cup \LL(y_{\ii, 1^{k-1}0},\theta_{\ii,1^{k-1}0}), \\
  \LL(y_{\ii1},\theta_{\ii0}) &\coloneq \LL(y_{\ii, i^k},\theta_{\ii, 1^k}).
\end{align*}
Let $M_{\ii} \coloneq \{M_{\ii,\varnothing}, M_{\ii,1}, \dots, M_{\ii,1^{k-1}}\}$, which contains all the fineness parameters used in this procedure.
We say $M_{\ii}$ is sufficiently large and denote it by $M_{\ii}\gg1$, to mean that the used fineness parameters are sufficiently large.

We call the above iteration procedure as the $(\gamma, \beta, \sigma, M_{\ii})$-Venetian blind zigzag construction on $\LL(y_{\ii},\theta_{\ii})$, 
and summarize it as an operator $V_{\gamma, \beta, \sigma, M_{\ii}}$:
\begin{align*}
      \LL(y_{\ii},\theta_{\ii}) \xlongrightarrow{V_{\gamma, \beta, \sigma, M_{\ii}}}  
      \LL(y_{\ii0},\theta_{\ii0}) \cup \LL(y_{\ii1},\theta_{\ii1})  \eqcolon V_{\gamma, \beta, \sigma, M_{\ii}}(\LL(y_{\ii},\theta_{\ii})).
\end{align*}
For $\kk \in \{\ii0,\ii1\}$, any line segment $L \in \LL(y_{\kk},\theta_{\kk})$ is parallel to $y_{\kk}$ and assigned the direction $\theta_{\kk} \in l$.

Now we focus on some geometric properties between $y_{\ii}, y_{\ii0}, y_{\ii1}$ and $W$. 
We already know that 
\begin{align*}
  \angle(y_{\ii}, y_{\ii0}) = -\sigma\beta, \quad  \angle(y_{\ii0},W) =0.
\end{align*}
More importantly, we will show the following results:
  \begin{align}\label{eq:zigzag}
    \angle (y_{\ii},y_{\ii1}) = k\sigma\gamma +o_{M_{\ii}}(1), 
    \quad \sin\angle(y_{\ii1},W) \le \frac{\sin \angle(y_{\ii},W)}{\sin \beta} + o_{M_{\ii}}(1),
  \end{align}
where $o_{M_{\ii}}(1) \coloneq o_{M_{\ii,\varnothing}}(1) + \cdots + o_{M_{\ii,1^k}}(1)$.

Actually, for each $\jj \in \{\varnothing, 1, 11,\dots, 1^{k-1}\}$, from $y_{\ii,\jj}-y_{\ii,0} -y_{\ii,\jj1} = \wt{y}_{\ii,\jj1}-y_{\ii,\jj1}$ and \eqref{eq:vectorclose_j1}, we have
\begin{align*}
  \abs{y_{\ii,\varnothing}-(\abs{\jj1}y_{\ii,0}+y_{\ii,\jj1})}
  &=\abs{[y_{\ii,\varnothing}-(\abs{\jj}y_{\ii,0}+y_{\ii,\jj})] +[y_{\ii,\jj}-y_{\ii,0} -y_{\ii,\jj1}]} \\
  &\le \abs{y_{\ii,\varnothing}-(\abs{\jj}y_{\ii,0}+y_{\ii,\jj})} + o_{M_{\ii,\jj}}(1).
\end{align*}
By iterating the above inequality, we have
\begin{align}\label{eq:vectorclose_0_jj1}
  \abs{y_{\ii,\varnothing}-(\abs{\jj1}y_{\ii,0}+y_{\ii,\jj1})} \le o_{M_{\ii,\varnothing}}(1) + o_{M_{\ii,1}}(1)+ \cdots + o_{M_{\ii,\jj}}(1) \le o_{M_{\ii}}(1).
\end{align}
Thus, $\abs{y_{\ii,\varnothing}-(\abs{\jj1}y_{\ii,0}+y_{\ii,\jj1})}$ can be arbitrarily small provided that $M_{\ii} \gg 1$,
which implies that for each $\jj$, the vectors $\{y_{\ii,\varnothing}, \abs{\jj1}y_{\ii,0}, y_{\ii,\jj1}\}$ almost form a triangle.
This leads to
\begin{enumerate}[label=(\arabic*)]
  \item the vectors $\{y_{\ii,1},\ y_{\ii,11},\ \dots,\ y_{\ii,1^k} \}$ almost lie in the 2-plane generated by $\{y_{\ii,\varnothing},y_{\ii,0}\}$, which, combining with \eqref{eq:vectorangle_i_i1}, implies
        \begin{align}\label{eq:angle_empty_i1}
          \angle (y_{\ii,\varnothing},y_{\ii,\jj1}) = \sum_{\kk \prec \jj1} \angle (y_{\ii,\kk},y_{\ii,\kk1}) +o_{M_{\ii}}(1) =\abs{\jj1}\sigma\gamma +o_{M_{\ii}}(1);
        \end{align}
  \item for any vector $z_r = y_{\ii,\varnothing}-ry_{\ii,0}$ with $r \in \R$, from $y_{\ii,0} \in W$ and $\abs{\angle(y_{\ii,\varnothing},y_{\ii,0})} = \beta$, a simple geometric calculation yields 
        \begin{align*}
          \sin \angle(z_r, W) \le \frac{\sin\angle(y_{\ii}, W)}{\sin \beta}, \quad \forall r \in \R.
        \end{align*}
        Thus, combination with \eqref{eq:vectorclose_0_jj1}, 
        we can find $r_{\jj1} \in \R$ close to $\abs{\jj1}$ and vector $z_{\ii,\jj1}= y_{\ii,\varnothing}-r_{\jj1}y_{\ii,0}$ satisfying $\abs{z_{\ii,\jj1}-y_{\ii,\jj1}} \le o_{M_{\ii}}(1)$,
        such that
        \begin{align}\label{eq:close_W}
          \sin \angle(y_{\ii,\jj1}, W) \le \sin\angle(z_{\ii,\jj1}, W) + \abs{\angle(y_{\ii,\jj1},z_{\ii,\jj1})} \le \frac{\sin\angle(y_{\ii,\varnothing}, W)}{\sin \beta} +o_{M_{\ii}}(1). 
        \end{align}
\end{enumerate}
Applying \eqref{eq:angle_empty_i1} and \eqref{eq:close_W} for $\jj = 1^{k-1}$ naturally leads to \eqref{eq:zigzag}.

In the following discussions, we abbreviate the families $\LL(y_{\ii},\theta_{\ii})$ as $\LL_{\ii}(\theta_{\ii})$ or $\LL_{\ii}$ for simplicity.
We give a schematic diagram of the basic zigzag iteration procedure in \autoref{fig:zigzag1}, where the line segments are understood to lie close to a 2-plane $W$.

%%%%%%%%%
\begin{figure}[htbp]
    \centering
    \def\svgwidth{\textwidth}
    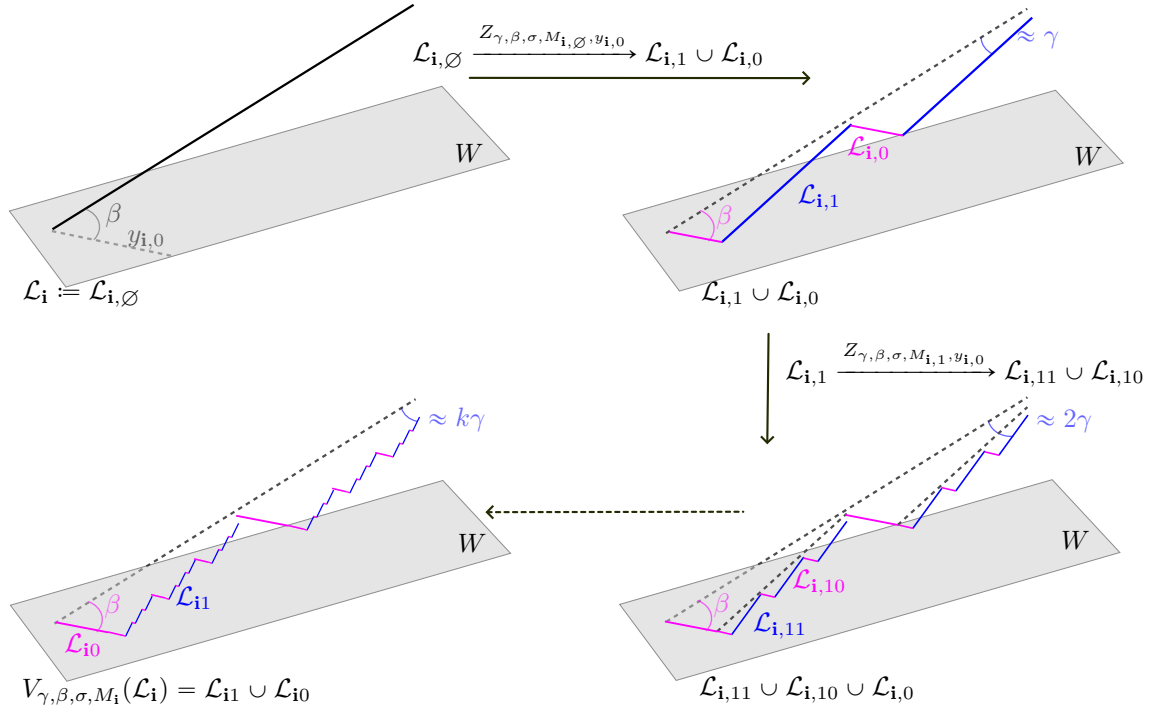
    \caption{Venetian-blind zigzag}
    \label{fig:zigzag1}
\end{figure}

%%%%%%%%%

From $\angle(y_{\ii},y_{\ii0}) = -\sigma \beta$ and $\angle (y_{\ii},y_{\ii1}) = k\sigma\gamma +o_{M_{\ii}}(1)$ (see \eqref{eq:zigzag}), a simple computation leads to
\begin{align}\label{eq:length_control}
  \HH^1(\LL_{\ii1}) \lesssim \frac{\sin \beta}{\sin (k\gamma + \beta)} \HH^1(\LL_{\ii}), \quad
  \HH^1(\LL_{\ii0}) \lesssim \frac{\sin k\gamma}{\sin (k\gamma + \beta)} \HH^1(\LL_{\ii}),
\end{align}
where the implicit constant is close to 1 and uniform to all indexes $\ii$ (coming from the fact that $\{y_{\ii},y_{\ii0},y_{\ii1}\}$ forms a shape very close to a triangle).
This shows that $\LL_{\ii1}$ is quantitatively short and $\LL_{\ii0}$ is quantitatively long.
We regard $\gamma$ as a good angle and $\LL_{\ii1}$ as a good line segments family labeled by good index. 
Similarly, we define the bad angle $\beta$ and the bad line segments family $\LL_{\ii0}$ (see \cite[Section 4]{ChCs2019} for more details).
From \eqref{eq:thetaangle_j_j0} and \eqref{eq:thetaangle_j_j1}, all the directions in $l$ that have been used in this procedure are 
\begin{align}\label{eq:directions}
  \{\theta_{\ii, 1}, \theta_{\ii, 11},\dots, \theta_{\ii, 1^k}=\theta_{\ii1}, \theta_{\ii, 0}=\theta_{\ii0}\}= \{\theta_{\ii} + \sigma\gamma, \dots, \theta_{\ii} + \sigma k\gamma,\theta_{\ii} - \sigma \beta\}.
\end{align}.

\subsubsection{Zigzag path}\label{ssect:zigzagpath}

We set the line segment from origin to $y_\varnothing =\eta(\iota)$ as $\LL_\varnothing(\theta_{\varnothing})$, 
which is associated with the direction $\theta_{\varnothing} = [\eta(\iota)] \in l$.
This original line segment corresponds to the natural isometry path from $\id$ to $\iota$.
Recall that we chose $W$ to pass through $\eta(\iota)$, thus we have $y_{\varnothing} \in W$ and $\theta_{\varnothing} \in l$.

Let $\epsilon$ be a small positive number.
We will start from $\LL_\varnothing(\theta_{\varnothing})$ to iterate the Venetian blind construction with different parameters to obtain a zigzag path $P\subset \Isom(V)$, 
which contains many ``good'' segments and ``bad'' segments associated with different directions in $l$.
The whole iteration procedure takes place within a small neighborhood of $\LL_\varnothing(\theta_{\varnothing})$, 
and the projective centers of the isometries used lie in $B_{\P^2_{V\times\R}}(l,\epsilon)$.

Although each Venetian blind construction contains many basic zigzag iteration repressed by small indexes $\{(\ii,\jj)\}_{\jj}$, we can use only big indexes to describe the result.
Thus, the entire iterative process can be indexed by big indexes in $\{0,1\}^*$ (but we still can transform to small indexes when needed, see \eqref{eq:finner_poset}). 
To each $\ii\in\{0,1\}^*$ we associate a family $\LL_{\ii}(\theta_\ii)$ of parallel line segments (related to a direction $\theta_{\ii} \in l$), 
a direction set $J_{\ii}\subset l$, Venetian blind parameters $\{\gamma_\ii,\beta_\ii,\sigma_\ii,M_{\ii}\}$, 
and a length parameter $\epsilon_\ii>0$.

The positive length parameters $\{\epsilon_{\ii}\}_{\ii \in \{0,1\}^*}$ satisfy $\sum_{\ii \in \{0,1\}^*} \epsilon_{\ii} < \epsilon$ and $\epsilon_{\ii} = \epsilon_{g_1(\ii)}$ for all $\ii \in \{0,1\}^*$.
For each $\ii \in \{0,1\}^*$, the angles $\gamma_{\ii}$ and $\beta_{\ii}$ are chosen to satisfy \eqref{eq:angle_initial} and the following conditions:
\begin{align}
  &\beta_\ii \le \min\{\beta_\mathbf{i'}, \frac{1}{\abs{\ii}_0}\},\ \beta_\ii = \beta_{g_1(\ii)}, \label{cond-beta_decreasing}\\
  &\beta_\ii \HH^1(\LL_\ii) \le \epsilon_\ii,\ \text{for good index } \ii, \label{cond-beta} \\
  &\gamma_\ii \HH^1(\LL_\ii) \le \epsilon_\ii,\ \text{for all index } \ii.  \label{cond-gamma}
\end{align}
These conditions can be satisfied, since $\HH^1(\LL_{\ii})$ is essentially controlled by $\{\gamma_{\jj},\beta_{\jj}\}_{\jj \prec \ii}$, see \eqref{eq:length_control}.
For fixed $\ii$, the direction set $J_{\ii}$ and the parameters $\sigma_{\ii}, M_{\ii}$ will de determined during the iteration process.

\textbf{Beginning step.}
We start with $\LL(y_\varnothing,\theta_{\varnothing})$ and define $J_{{\varnothing}} \coloneq \{\theta_{\varnothing}\}$.
We choose $\sigma_{\varnothing} \in \{1, -1\}$ arbitrarily.
Since $y_{\varnothing} \in W$, we can naturally find $y_0 \in W$ satisfying $\angle(y_{\varnothing},y_0) = -\sigma_{\varnothing}\beta_{\varnothing}$.
For sufficiently large $M_{\varnothing}$, applying the $(\gamma_{\varnothing}, \beta_{\varnothing}, \sigma_{\varnothing},M_\varnothing)$-Venetian blind construction on $\LL(y_\varnothing,\theta_{\varnothing})$ gives 
\begin{align*}
  \LL(y_{\varnothing},\theta_{\varnothing}) \xlongrightarrow{V_{\gamma_{\varnothing}, \beta_{\varnothing}, \sigma_{\varnothing}, M_{\varnothing}}}  \LL(y_{0},\theta_{0}) \cup \LL(y_{1},\theta_{1}).
\end{align*}

Then the direction sets of $\LL(y_{0},\theta_{0})$ and $\LL(y_{1},\theta_{1})$ are defined as 
closed intervals $J_0 \coloneq [\theta_{\varnothing}, \theta_0]$ and $J_1 \coloneq [\theta_{\varnothing}, \theta_1]$, respectively.

\textbf{Iteration step.}
For a general index $\ii \in \{0,1\}^*$, 
assume the families $\{\LL(y_{\jj},\theta_{\jj})\}_{\jj \preceq \ii}$ 
and the related parameters $\{\gamma_{\jj}, \beta_{\jj}, \sigma_{\jj},M_{\jj}\}_{\jj \prec \ii}$ have been constructed.
We also know the angles $\{\gamma_{\ii}, \beta_{\ii}\}$, and the parameters $\sigma_{\ii}$ and $M_{\ii}$ will be determined later.
To apply a $(\gamma_{\ii}, \beta_{\ii}, \sigma_{\ii},M_{\ii})$-Venetian blind zigzag construction at index $\ii$, 
we need to check the condition \eqref{eq:cond_VBzigzag}.
For bad $\ii$, we already have $\angle(y_{\ii},W) =0 < \beta_{\ii}$.
For good $\ii$, we can use \eqref{eq:zigzag} iteratively for the indexes $\jj \prec\ii$ to get
\begin{align*}
  \sin\angle(y_{\ii},W) 
  &\le \frac{\sin \angle(y_{\ii'},W)}{\sin \beta_{\ii'}} + o_{M_{\ii'}}(1) \le \cdots \\
  &\le \frac{\sin \angle(y_{\varnothing},W)}{\prod_{\jj \prec \ii}\sin \beta_{\jj}} + \sum_{\jj\prec\ii} \frac{o_{M_{\jj}}(1)}{\prod_{\jj \prec\kk \prec \ii}\sin\beta_{\kk}} 
   \le \sum_{\jj\prec\ii} \frac{o_{M_{\jj}}(1)}{(\sin\beta_{\ii})^{\abs{\ii}-\abs{\jj}-1}},
\end{align*}
where the last inequality use the fact that $y_{\varnothing} \in W$ and \eqref{cond-beta_decreasing}.
Furthermore, by taking the fineness parameters $\{M_{\jj}\}_{\jj \prec \ii}$ sufficiently large, we can ensure that 
\begin{align}\label{eq:angle_y_W}
  \angle(y_{\ii},W) \le \sum_{\jj\prec\ii} \frac{o_{M_{\jj}}(1)}{(\sin\beta_{\ii})^{\abs{\ii}-\abs{\jj}-1}} < \beta_{\ii}
\end{align}
as desired in \eqref{eq:cond_VBzigzag}.

Thus, we can apply the $(\gamma_{\ii}, \beta_{\ii}, \sigma_{\ii}, M_{\ii})$-Venetian blind zigzag construction on $\LL_{\ii}(\theta_\ii)$ with a sufficiently large $M_{\ii}$, 
and continue the iterating as
\begin{equation*}
  \LL_{\ii}(\theta_\ii) 
  \xrightarrow{V_{\gamma_\ii, \beta_\ii, \sigma_\ii, M_{\ii}}} 
  \LL_{\ii1}(\theta_{\ii1}) \cup \LL_{\ii0}(\theta_{\ii0}),
\end{equation*}
where we use $\LL_{\ii}(\theta_{\ii})$ for $\LL(y_{\ii},\theta_{\ii})$ for simplicity.
For the new index $\jj \in \{\ii1, \ii0\}$, the direction set of $\LL_{\jj}(\theta_{\jj})$ is defined as 
\begin{equation*}
  J_{\jj} = J_{\ii} \cup [\theta_{\ii}, \theta_{\jj}] \subset l,
\end{equation*}
which contains all the directions used in the iterations of $\{\kk \colon \kk \preceq \jj\}$.
The sign parameter $\sigma_{\ii} \in \{1,-1\}$ is chosen so as to maximize the length of the direction set $J_{\ii1}$.

\textbf{Stopping time condition.}
The iterative process terminates according to two conditions, which we refer to as the local and global stopping conditions (see \cite[Subsection 4.7]{ChCs2019} for more discussions).

$\mydot$ Local stopping condition:
The iteration is stopped at the index $\ii$ if 
$$\HH^1(\LL_{\ii}) \le \epsilon_{g_1(\ii)}.$$
And we denote the collection of all such locally stopped indices by $S\subset \{0,1\}^*$, 
whose corresponding line families have short length controlled by parameters $\{\epsilon_{\ii}\}$. 

$\mydot$ Global stopping condition:
We can find a positive integer $k$ such that
\begin{equation*}
  \max\{\HH^1(l \setminus J_{\ii}) \colon \abs{\ii}=k,\ \ii\notin S \} \le \epsilon.
\end{equation*}
And we stop the iteration at $\ii$ if $\abs{\ii} >k$, which is the global stopping condition. 
We collect all the used big indices in the iterations as a set $D \subset \{0,1\}^*$.

\textbf{Final zigzag path.}
The resulting zigzag path is formed from the line families indexed by $\max D$ and can be decomposed as
\vspace{-0.3\baselineskip}
\begin{align*}
  P= (\bigcup_{\ii \in S} \LL_{\ii}) \cup (\bigcup_{\ii \in (\max D) \setminus S} \LL_{\ii} ).
\end{align*}
See \autoref{fig:zigzag2} for an illustration of the iteration procedure.

%%%%%%%%%%%  0 uses '\fontsize{7}{9}\selectfont' ; 01 uses '\fontsize{5}{7}\selectfont' to replace '\lineheight{1.25}'
\begin{figure}[htbp]
    \centering
    \def\svgwidth{0.93\textwidth}
    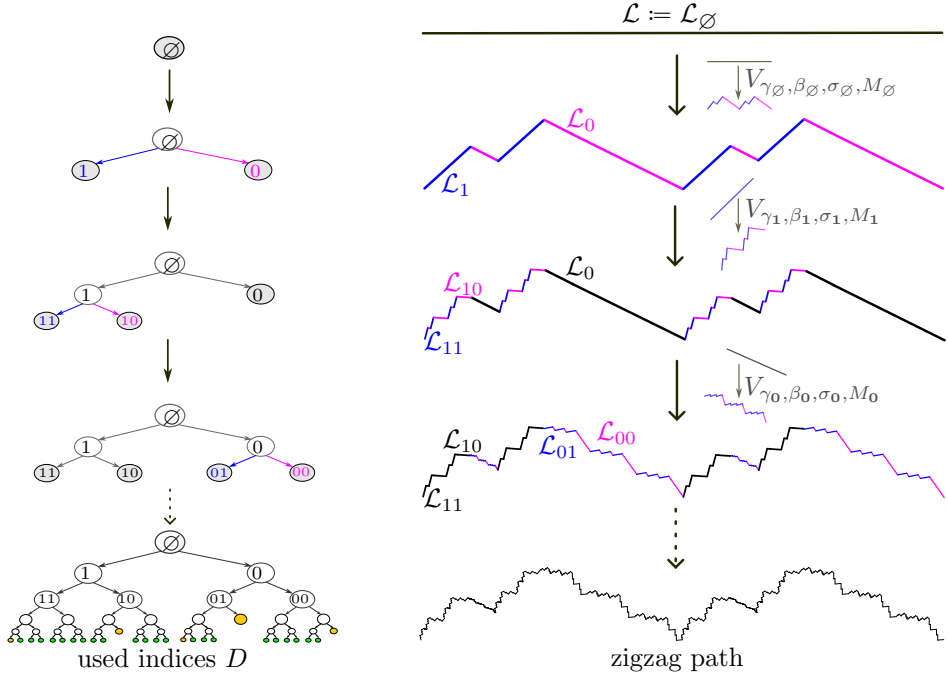
    \captionsetup{justification=centering}
    \caption{Binary index iterations (left) and line segments (right): \\ yellow+green = $\max D$, yellow = $S$; segments indexed by $\max D$ form the zigzag path. \\Segments are close to a 2-plane (drawn planar for simplicity).}
    \label{fig:zigzag2}
\end{figure}

%%%%%%%%%%%

For any $\ii \in D$, the segment family $\LL_{\ii}$ is assigned to a direction $\theta_{\ii} \in l$, 
a directional interval $J_{\ii} \subset l$ and a number $\epsilon_{\ii} >0$.
All the parallel line segments of $\LL_{\ii}$, taken together, correspond to a vector $y_{\ii}$ near $W$ (see \eqref{eq:angle_y_W}).
The $V\times\R$-parametrization guarantees that 
$\LL_{\ii}$ can be interpreted as a path in a sufficiently small neighborhood of the natural path $L^{\iota_{\ii}}$ from $\id$ to $\iota_{\ii} \coloneq \eta^{-1}(y_{\ii}) \in \Isom(V)$.
And we have $\HH^1(\LL_{\ii}) = \abs{y_{\ii}} = \abs{\eta(\iota_{\ii})}$.
For any compact set $E$, from \autoref{lem:smallneighborhood} and \autoref{lem:smallmeasisom}, the Lebesgue measure of the set covered by moving $E$ along $\LL_{\ii}$ can be estimated as
  \begin{align}\label{eq:controlby_iota}
    \bg[v]{\bigcup_{p \in \LL_{\ii}}p(E)} 
    \lesssim  \bg[v]{\bigcup_{p \in L^{\iota_{\ii}}}p(E)} \lesssim_{E} \HH^1(\LL_{\ii}) \HH^{n-1}(E),
  \end{align} 
where the last inequality has an implicit constant related to the boundedness of $E$.

More importantly, for the indexes appeared in the zigzag path (i.e., $\ii \in \max D$), we have 
  \begin{align}
    &\HH^1(\LL_\ii) = \abs{\eta(\iota_{\ii})} \lesssim \epsilon_{g_1(\ii)}, \quad \text{for} \quad \ii \in S; \label{eq:isom_epsilon.2} \\
    &\HH^1(l \setminus J_{\ii}) < \epsilon, \quad \text{for} \quad \ii \in (\max D) \setminus S. \label{eq:isom_epsilon.3}
  \end{align} 
where $g_1(\ii)$ is the last good index among $\ii$ and its ancestors.

All the used directions can be expressed using small indexes as $\{\theta_{\ii,\jj}\} \subset l$ (see \eqref{eq:directions}).
They divide the projective line $l$ into finitely many closed intervals $\{J\}$, whose interiors are pairwise disjoint.
Such closed intervals of $l$ are called the elementary intervals of $P$.
There is an estimate on the length of the elementary interval given by 
\vspace{-0.3\baselineskip}
  \begin{align}\label{eq:elementary_interval_length2}
    \text{elementary interval}\  J \subset [\theta_{\ii}, \theta_{\ii'}] \implies \HH^1(J) \le \alpha_{\ii},
  \end{align}
\vspace{-0.3\baselineskip}
where 
\vspace{-0.3\baselineskip}
  \begin{align}\label{eq:alpha_angle}
    \alpha_{\ii}=
      \begin{cases}
        \gamma_{\ii'}, &\ii \ \text{is good}, \\
        \beta_{\ii'}, &\ii \ \text{is bad}.
      \end{cases}
    \end{align}
\vspace{-0.3\baselineskip}

% From \eqref{eq:isom_epsilon.2}, the family $\LL_\ii$ with $\ii\in S$ satisfies $\HH^1(\LL_\ii)<\epsilon_{g_1(\ii)}$, which, combining with \autoref{lem:measisom}, implies that moving along isometries in $\LL_\ii$ naturally cover a small set.
% The family $\LL_\ii$ with $\ii\in\max D\setminus S$ has a large direction sets $J_{\ii} \subset l$ satisfying \eqref{eq:isom_epsilon.3}, which, combining with the projective axis introduced in \autoref{ssect:simple projective rotations} and \autoref{lem:smallmeasisom}, can be used to determine a large subset of $E$ to be moved along $\LL_\ii$.
% These are the basic ideas of moving subsets of $E$ along $P$ to cover a set with small Lebesgue measure.

\subsection{Zigzag path in $\Isom(\R^n)$}\label{ssect:path_Isom_R^n}

For a given isometry $\iota \in \Isom(\R^n)$,
we need to construct a path in $\Isom(\R^n)$ to connect the identity $\id$ and  $\iota$.
From \autoref{lem:decomp_isometry}, any isometry $\iota \in \Isom(\R^n)$ can be decomposed into $k$ simple rotations in the planes $\{V_i\}_{i=1}^k \subset \GG^2(\R^n)$, with $k \le 2+\floor{\frac{n}{2}}$, say
\begin{align*}
  \iota = \rho_k \circ \cdots \circ \rho_2 \circ \rho_1,
\end{align*}
where $\rho_i$ is a simple rotation related to $V_i$ for $i=1, \dots, k$ (see \autoref{def:V_rotat}).
Suppose we have constructed, for each $i$, a path $P_i \subset \Isom(V_i)$ from $\id$ to $\rho_i$.
Then the (ordered) combination of such paths is
\begin{align*}
  P \coloneq P_k \circ \cdots \circ P_2 \circ P_1,
\end{align*}
which is a zigzag path in $\Isom(\R^n)$ that connects $\id$ and $\iota$.

\section{Kakeya isometry: proof of \autoref*{thm:global_Kak_isom} and \autoref*{thm:local_Kak_isom}}\label{sect:Kak_isom}

We prove \autoref{thm:local_Kak_isom} and \autoref{thm:global_Kak_isom} in \autoref{ssect:proof_loc_Kak_isom} and \autoref{ssect:proof_glo_Kak_isom}, respectively.

\subsection{Proof of \autoref*{thm:local_Kak_isom}}\label{ssect:proof_loc_Kak_isom}

Assume $V \in \GG^2(\R^n)$ is fixed, $\iota \in \Isom(V)$ and $\epsilon >0$ are given.
First we show how to move subsets of a bounded $(n-1)$-rectifiable set $E$ to $\iota(E)$ using only isometries in $\Isom(V)$,
and cover a set with Lebesgue measure smaller than $\epsilon$.
The result is described as the following lemma.

\begin{lemma}[$\epsilon$-isometry path for $(n-1)$-rectifiable set]\label{lem:kak_isom_epsilon}
  
  Suppose $V \in \mathcal{G}(n,2)$ is a $2$-plane, $r, \epsilon, \delta >0$,  and $\delta < \epsilon/2$ are sufficiently small.
  $E \subset B(0,r)$ is a bounded $(n-1)$-rectifiable set with $\HH^{n-1}(E) < \infty$. 

  Let $\iota \in \Isom(V)\setminus\{\id\}$ be a simple isometry with projective center $\cC_{\iota} \in \P^2_{V\times\R}$,
  and $l \in \GG^1(\P^2_{V\times\R})$ be a projective line passing through the projective center.
  Then there is a polygonal isometry path $P = \bigcup_{k} \LL_{k} \subset \Isom(V)$ connecting $\id $ and $\iota$, 
  where each $\LL_k$ consist of small simple isometries with the same projective center near $l$.
  And for each $k$, there exists 
  \begin{align}\label{eq:kak_isom_refine1.0}
    u_k \in l \setminus \widebar{B}_l(\cC_{\iota}, \epsilon)
  \end{align}
  such that
  \begin{align}\label{eq:kak_isom1.1}
    \bgg[v]{\bigcup_k \bigcup_{p \in \LL_k} p(E(u_k, \epsilon, \delta))} \lesssim C_r \HH^{n-1}(E) \epsilon,
  \end{align}
  where the moved set is defined as
  \begin{align}\label{eq:kak_isom1.2}
    E(u_k, \epsilon, \delta) \coloneq \{x \in E \colon \nu_x \cap \aA(B_{l}(u_k, \epsilon) \setminus B_{l}(\cC_{\iota}, \delta) ) = \emptyset\}
  \end{align}
  and $\aA \colon \mathscr{P}(\P^2_{V\times\R}) \to \mathscr{P}(\P^n)$ is the axis-lifting operator.
  The constant $C_r >1$ in \eqref{eq:kak_isom1.1} is independent of $\epsilon$ and $\delta$.

  For simplicity, we write $K_{E, \epsilon, \delta}(\id, \iota)$ for the constructed $\epsilon$-isometry path $P$ from $\id$ to $\iota$.
\end{lemma}

\begin{proof}[Proof of \autoref*{lem:kak_isom_epsilon}]
  By considering the regular part of the rectifiable set, we can assume $E$ is a compact set with continuous normal vector field, and the normal line $n_x$ is well-defined for all $x \in E$.
  Also, since we only use isometries in $\Isom(V)$, from \autoref{rmk:zero_isom_P_V^perp}, we can assume 
  \begin{align}\label{eq:kak_isom1.3}
    N(E) \subset \P^{n-1} \setminus \P^{n-3}_{V^{\perp}}
  \end{align} 
  without loss of generality, where $N(E) = \{N(x)\colon x \in E\}$.

  From the construction described in \autoref{ssect:zigzagpath},  
  for $\epsilon >0$ and the projective line $l \in \GG^1(\P^2_{V\times\R})$ passes through $\cC_{\iota}$,
  we can connect the $\id$ and $\iota$ in $\Isom(V)$ by a zigzag path $P$ of the form
  \begin{align*}
      P= (\bigcup_{\ii \in S} \LL_{\ii}) \cup (\bigcup_{\ii \in (\max D) \setminus S} \LL_{\ii} ),
  \end{align*}
  where each $\LL_{\ii}$ contains a collection of segments corresponding to small pieces of the simple isometry $\iota_{\ii}$ 
  with the projective center $[\eta(\iota_{\ii})]$ lies near the projective line $l$.
  All the needed parameters $\{\gamma_\ii, \beta_\ii, \sigma_\ii, \epsilon_\ii, M_{\ii}\}$ for the zigzag iteration,  
  the directions $\{\theta_{\ii}\}_{\ii} \subset l$, the directional intervals $\{J_{\ii}\}$ related to $\{\LL_{\ii}\}$, 
  and the elementary intervals $\{J\}$ are determined.
  After relabeling the segments of $P$ by indices $k \in \N $, the polygonal isometry path has the form $P =\bigcup_k L_k \subset \Isom(V)$, 
  corresponding to a sequence of simple isometries $\{\iota_k\}_k$ whose composition is the given $\iota$.

  Next we determine the points $\{u_\ii \}_{\ii \in \max D} \subset l$.
  For $\ii \in S$, we choose $u_\ii \in l \setminus \widebar{B}_l(\cC_{\iota}, \epsilon)$ arbitrarily.
  For $\ii \in (\max D) \setminus S$ and $\delta < \epsilon/2$, from \eqref{eq:isom_epsilon.3} and the fact $\cC_{\iota} \in J_{\ii}$, we can choose $u_\ii \in l \setminus \widebar{B}_l(\cC_{\iota}, \delta)$ such that 
  \begin{align}\label{eq:isom_epsilon.4}
    l \setminus J_{\ii} \subset  B_{l}(u_\ii,\epsilon).
  \end{align}
  Again, after relabeling the indices, we get the points $\{u_k\}_k \subset l \setminus \widebar{B}_l(\cC_{\iota}, \epsilon)$ as stated in \eqref{eq:kak_isom_refine1.0}.

  Using the axis-lifting operator $\aA$ introduced in \autoref{def:center_axis_corrspd}, 
  each directional interval $J_{\ii} \subset l$ determines a set $\aA(J_{\ii}) = \spn_{\P^n} \{J_{\ii}, \P^{n-3}_{V^{\perp}\times \{0\}} \} \subset \P^n$.
  And the set $B_{l}(u_\ii,\epsilon) \subset l$ determines $\aA(B_{l}(u_\ii,\epsilon)) \subset \P^n$ which contains many projective axes.
  For any direction interval $J_{\ii}$, we define 
  \begin{align*}
    E_{J_{\ii}} \coloneq \{x \in E \colon \nu_x \cap \aA(J_{\ii}) \neq \emptyset\}.
  \end{align*}
  For any elementary interval $J \subset l$, the notation $E_J$ is defined similarly.
  More importantly, since $J$ is a closed interval, we know that $\aA(J)$ is a closed set in $\P^n$, 
  which, combining with the fact that the normal line to $E$ is continuously changed, leads to 
  \begin{align}\label{eq:compact_set}
    E_{J} = \{x \in E \colon \nu_x \cap \aA(J) \neq \emptyset\}\ \text{is a closed subset of}\ E.
  \end{align}
  Similar to \eqref{eq:kak_isom1.2}, we also define
  \begin{align} \label{eq:moved_set_2}
    E(u_k, \epsilon) \coloneq \{x \in E \colon \nu_x \cap \aA(B_{l}(u_k, \epsilon)) = \emptyset\},
  \end{align}
  which is a subset of $E(u_k, \epsilon, \delta)$ for all index $k$.

  Now we focus on the set covered in the procedure of moving $\{E(u_k, \epsilon, \delta)\}_k$ along $P = \bigcup_{k}\LL_k$,
  and our aim is to prove \eqref{eq:kak_isom1.1}.
  We can simplify the analysis by replacing $\{E(u_k, \epsilon, \delta)\}_k$ with $\{E(u_k, \epsilon)\}_k$ as defined above.
  For $\delta$ sufficiently small, we claim the following estimate
  \begin{align}\label{eq:1_1'}
    \abs{\bigcup_k \bigcup_{p \in \LL_k} p(E(u_k,\epsilon, \delta))} 
    \lesssim \abs{\bigcup_k \bigcup_{p \in \LL_k} p(E(u_k,\epsilon))} + \delta.
  \end{align}
  Actually, for $\delta < \epsilon$, by setting (recall that $\aA_{\iota} = \aA(\cC_{\iota})$)
  \begin{align*}
    R=\{x \in E \colon \nu_x \cap B_{\P^n}(\aA_{\iota}, \delta) \neq \emptyset\},
  \end{align*}
  then for all index $k$, we have
  \begin{align*}
    E(u_k,\epsilon,\delta) \bigtriangleup E(u_k,\epsilon) \subset R,
  \end{align*}
  where $A \bigtriangleup B = (A\setminus B) \cup (B\setminus A)$.
  From \autoref{lem:smallmeasisom}, the initial isometry $\iota$ acting on $R$ only covers a small set, i.e.,
  \begin{align*}
    \abs{\bigcup_{p \in L^{\iota}} p(R)} \lesssim  C_r \delta \abs{\eta(\iota)} \HH^{n-1}(R),
  \end{align*}
  where $L^{\iota}$ is the natural path from $\id$ to $\iota$, and $r$ comes from $E \subset B(0,r)$. 
  Recall that all fineness parameters can be chosen sufficiently large such that 
  the resulting zigzag path $P$ stays in a small neighborhood of the initial path $L^{\iota}$,
  which, combining with the small parameter $\delta$ and the small neighborhood lemma (see \autoref{lem:smallneighborhood}), we have 
  \begin{align*}
    \abs{\bigcup_{p \in P}p(R)} \lesssim \abs{\bigcup_{p \in L^{\iota}} p(R)} \lesssim C_r\delta \abs{\eta(\iota)} \HH^{n-1}(R) \lesssim_E \delta,
  \end{align*}
  where we use the fact $\HH^{n-1}(R) \le \HH^{n-1}(E) < \infty$. Also, for all $k$, we have
  \begin{align*}
    \abs{\bgg[v]{\bigcup_{p \in \LL_k} p(E(u_k,\epsilon, \delta))}-\bgg[v]{\bigcup_{p \in \LL_k} p(E(u_k,\epsilon))} } 
    \le \abs{\bigcup_{p \in \LL_{k}} p\bg{E(u_k,\epsilon, \delta) \bigtriangleup E(u_k,\epsilon)}}
    \le \abs{\bigcup_{p \in \LL_k} p(R)},
  \end{align*}
  which further leads to
  \begin{align*}
    \abs{\bigcup_k \bigcup_{p \in \LL_k} p(E(u_k,\epsilon, \delta))} 
    &\le \abs{\bigcup_k \bigcup_{p \in \LL_k} p(E(u_k,\epsilon))} + \abs{\bgg[v]{\bigcup_k \bigcup_{p \in \LL_k} p(E(u_k,\epsilon, \delta))}-\bgg[v]{\bigcup_k \bigcup_{p \in \LL_k} p(E(u_k,\epsilon))} }  \\
    &\le \abs{\bigcup_k \bigcup_{p \in \LL_k} p(E(u_k,\epsilon))} + \abs{\bigcup_{p \in P}p(R)} \lesssim \abs{\bigcup_k \bigcup_{p \in \LL_k} p(E(u_k,\epsilon))} + \delta,
  \end{align*}
  and proves \eqref{eq:1_1'}.
  
  Thus, from \eqref{eq:1_1'}, to prove \eqref{eq:kak_isom1.1}, it is sufficient to show 
  \begin{align}\label{eq:kak_isom1.1'}
    \bgg[v]{\bigcup_k \bigcup_{p \in \LL_k} p(E(u_k, \epsilon))} \lesssim C_r \HH^{n-1}(E) \epsilon,
  \end{align}
  provided that $\delta$ is chosen sufficiently small.

  \textbf{$\bullet$ Part 1:} 
  For the isometries along $\bigcup_{\ii \in S} \LL_{\ii}$, by using the sub-additivity of $\HH^n$ and \autoref{lem:measisom}, we have
  \begin{align*}
    \bgg[v]{\bigcup_{\ii \in S} \bigcup_{p \in \LL_{\ii}} p(E)} 
    \le \sum_{\ii \in S} \bgg[v]{\bigcup_{p \in \LL_{\ii}} p(E)} 
    \le  C_r \sum_{\ii \in S} \HH^1(\LL_{\ii}) \HH^{n-1}(E),
  \end{align*}
  where the second inequality follows from \autoref{lem:measisom} applied to each $\LL_{\ii}$ 
  (see also \eqref{eq:controlby_iota} in \autoref{ssect:zigzagpath}).
   From \eqref{eq:isom_epsilon.2}, we have
  \begin{align*}
    \sum_{\ii \in S} \HH^1(\LL_{\ii}) \HH^{n-1}(E)
    \le \sum_{\ii \in S} \epsilon_{g_1(\ii)} \HH^{n-1}(E).
  \end{align*}
  Since all indices in $S$ belong to different family lines (otherwise it would contradict the local stopping condition of $S$), 
  we have $g_1(\ii) \neq g_1(\jj)$ for different $\ii, \jj \in S$, 
  which further leads to $\sum_{\ii \in S}\epsilon_{g_1(\ii)} \le \sum_{\ii \in D}\epsilon_{\ii} < \epsilon$.
  Combining all these estimates, we have
  \begin{align}\label{eq:isom_epsilon.4.1 part1}
    \bgg[v]{\bigcup_{\ii \in S} \bigcup_{p \in \LL_{\ii}} p(E)}  \le C_r\HH^{n-1}(E)\epsilon,
  \end{align}
  which shows that all isometries along $\bigcup_{\ii \in S} \LL_{\ii}$ only cover a set of measure controlled by $C_r\HH^{n-1}(E)\epsilon$. 

  \textbf{$\bullet$ Part 2:} 
  For the motion along $\bigcup_{\ii \in (\max D) \setminus S} \LL_{\ii}$, to prove the estimate \eqref{eq:kak_isom1.1'},
  we first enlarge the moved set \eqref{eq:moved_set_2}  by showing
  \begin{align*}
    E(u_{\ii}, \epsilon) = \{x \in E \colon \nu_x \cap \aA(B_{l}(u_\ii,\epsilon)) = \emptyset\} \subset \{x \in E \colon \nu_x \cap \aA(J_{\ii}) \neq \emptyset\} = E_{J_{\ii}}
  \end{align*} 
  for all $\ii \in (\max D)\setminus S$.
  Otherwise, there exists some $x \in E$ satisfying 
  both $\nu_x \cap \aA(B_{l}(u_\ii,\epsilon)) = \emptyset$ and $ \nu_x \cap \aA(J_{\ii}) = \emptyset$,
  which, combining with \eqref{eq:isom_epsilon.4}, leads to $\nu_x \cap \aA(l) = \emptyset$.
  However, this is impossible since $\aA(l) \in \GG^{n-1}(\P^n)$ 
  and $\nu_x \in \GG^1(\P^n)$ are projective subspace in $\P^n$ with dimension $(n-1)$ and 1, respectively, 
  and they must have nonempty intersection.
  Now we have 
  \begin{align}\label{eq:isom_epsilon.5}
    \bgg[v]{\bigcup_{\ii \in (\max D) \setminus S} \bigcup_{p \in \LL_{\ii}} p (E(u_{\ii},\epsilon))} \le \bgg[v]{\bigcup_{\ii \in (\max D) \setminus S} \bigcup_{p \in \LL_{\ii}} p(E_{J_{\ii}})}.
  \end{align} 

  Next we use the elementary intervals $\{J\}$ to decompose the right-hand side of \eqref{eq:isom_epsilon.5}.
  For an elementary interval $J \subset l$, 
  notice that 
  \begin{align}\label{eq:isom_epsilon.6}
    \bigcup\limits_{\ii \in (\max D) \setminus S \colon J_{\ii} \supset J} \bigcup_{p \in \LL_{\ii}} p(E_J)
  \end{align}
  is the covered set related to the elementary interval $J$, we can control \eqref{eq:isom_epsilon.5} by
  \begin{equation}
    \begin{aligned}\label{eq:isom_epsilon.7}
      \bgg[v]{\bigcup_{\ii \in (\max D) \setminus S} \bigcup_{p \in \LL_{\ii}} p(E_{J_{\ii}})} 
      =\bgg[v]{\bigcup_{J} \bigcup_{\ii \in (\max D) \setminus S \colon J_{\ii} \supset J } \bigcup_{p \in \LL_{\ii}} p(E_J)}
      \le \sum_{J}\sum_{\ii \in (\max D) \setminus S \colon J_{\ii} \supset J } \abs{\bigcup_{p \in \LL_{\ii}} p(E_J)},
    \end{aligned}
  \end{equation}
  which leads us to estimate $\abs{\bigcup_{p \in \LL_{\ii}} p(E_J)}$ for $\ii \in (\max D) \setminus S$ and $J \subset J_\ii$.

  First we consider a simpler case: the elementary interval $J \subset [\theta_{\ii}, \theta_{\ii'}] \subset l$ for some $\ii$.
  Then the estimate \eqref{eq:elementary_interval_length2} leads to 
  \begin{align*}
    \HH^1(J) \le \alpha_{\ii},
  \end{align*}
  where $\alpha_{\ii}$ is $\gamma_{\mathbf{i'}}$ for good $\ii$ and $\beta_{\mathbf{i'}}$ for bad $\ii$, see \eqref{eq:alpha_angle}.
  Recall that $E_J$ is compact, see \eqref{eq:compact_set} above.  
  Let $\eta = \HH^1(J) \HH^1(\LL_{\mathbf{i'}})\HH^{n-1}(E_J)$, 
  then \autoref{lem:smallneighborhood} shows that there exists a positive number $\delta(\eta) >0$ such that 
  \begin{align}\label{eq:isom_epsilon.8}
    \abs{\bigcup_{p \in \LL} p(E_J)} \lesssim \abs{\bigcup_{p \in \LL'} p(E_J)} + \eta, \quad \forall \LL \subset B_{\Isom(V)}(\LL', \delta(\eta)),
  \end{align}
  where $\LL'$ and $\LL$ are polygonal isometry paths in $\Isom(V)$.
  Recall the Venetian blind iteration procedure from $\LL_{\ii'}$ to $\LL_{\ii}$ described in \autoref{ssect:Vb},
  by choosing the fineness parameter $M_\mathbf{i'}$ sufficiently large, 
  we can find a family of line segments denoted by $\LL(\theta)$, which corresponds to small isometries with projective center $\theta \in J$,  
  such that 
  \begin{align*}
    \HH^1(\LL_{\ii}) \le \HH^1(\LL(\theta)) \le \HH^1(\LL_{\mathbf{i'}}), \quad \LL_\ii \subset B_{\Isom(V)}(\LL(\theta), \delta(\eta)).
  \end{align*}
  Combining with \autoref{lem:smallmeasisom}, 
  moving $E_J$ along $\LL(\theta)$ covers a set of estimate 
  \begin{align}\label{eq:isom_epsilon.9}
    \abs{\bigcup_{p \in \LL(\theta)} p(E_J)} \lesssim C_r \HH^1(J) \HH^1(\LL_{\mathbf{i'}})\HH^{n-1}(E_J),
  \end{align}
  where the implicit constant is independent of the choice of elementary interval $J$.
  Combining \eqref{eq:isom_epsilon.8} (with $\LL=\LL_{\ii}, \LL'=\LL(\theta)$), \eqref{eq:isom_epsilon.9} and $\HH^1(J) \le \alpha_{\ii}$, we have 
  \begin{align}\label{eq:isom_epsilon.10}
    \abs{\bigcup_{p \in \LL_{\ii}} p(E_J)} \lesssim C_r \alpha_{\ii} \HH^1(\LL_{\ii'}) \HH^{n-1}(E_J),
  \end{align}
  for the case $J \subset [\theta_{\ii}, \theta_{\ii'}] \subset l$.
  
  Next we consider the general case: $\ii \in (\max D) \setminus S$ and $J \subset J_\ii$.
  For a given pair $(\ii, J)$, 
  let $\kk=\kk(\ii,J) \preceq \ii$ be a big index before $\ii$ such that $J_{\kk}\setminus J_{\kk'} =[\theta_{\kk},\theta_{\kk'}] \supset J$,
  then from \eqref{eq:isom_epsilon.10}, one concludes that 
  \begin{align*}
    \abs{\bigcup_{p \in \LL_{\kk(\ii,J)}} p(E_J)} \lesssim C_r \alpha_{\kk(\ii,J)}\HH^1(\LL_{\kk(\ii,J)'})\HH^{n-1}(E_J).
  \end{align*}
   We can further simplify this estimate.
   If $\kk(\ii,I)$ is good, then $\alpha_{\kk(\ii,J)} = \gamma_{\kk(\ii,J)'}$.
   Combining with the condition \eqref{cond-gamma} on the parameters given in \autoref{ssect:zigzagpath},
   one concludes that 
   \begin{align}\label{eq:isom_epsilon.11}
    \abs{\bigcup_{p \in \LL_{\kk(\ii,J)}} p(E_J)} \lesssim C_r \epsilon_{\kk(\ii,J)'}\HH^{n-1}(E_J).
   \end{align}
   Similarly, if $\kk(\ii,J)$ is bad, then $\alpha_{\kk(\ii,J)} = \beta_{\kk(\ii,J)'}$, and the condition \eqref{cond-beta_decreasing} on family $\{\beta_{\ii}\}_{\ii \in \{0,1\}^*}$ guarantees $\beta_{\kk(\ii,J)'} = \beta_{g_1(\kk(\ii,J))}$.
   Combining with \eqref{cond-beta},
   we have 
  \begin{align}\label{eq:isom_epsilon.12}
    \abs{\bigcup_{p \in \LL_{\kk(\ii,J)}} p(E_J)} \lesssim C_r \epsilon_{g_1(\kk(\ii,J))}\HH^{n-1}(E_J).
  \end{align}
  By taking $\{\kk(\ii,J)\}_{\ii \in (\max D) \setminus S}$ in different family lines, 
  all $\kk(\ii,J)$ appeared in \eqref{eq:isom_epsilon.11} can be chosen to be different. 
  Similarly, all $g_1(\kk(\ii,J))$ appeared in \eqref{eq:isom_epsilon.12} can be chosen to be different.  
  Applying \eqref{eq:isom_epsilon.11} and \eqref{eq:isom_epsilon.12} to \eqref{eq:isom_epsilon.6}, we have
  \begin{align*}
    \bgg[v]{\bigcup\limits_{\ii \in (\max D) \setminus S \colon J_{\ii} \supset J } \bigcup_{p \in \LL_{\ii}} p(E_J)}
    &\lesssim  C_r \bgg{\sum_{\text{good}\ \kk(\ii,J)} \epsilon_{\kk(\ii,J)'}\HH^{n-1}(E_J) + \sum_{\text{bad}\ \kk(\ii,J)} \epsilon_{g_1(\kk(\ii,J))}\HH^{n-1}(E_J)} \\
    &\lesssim C_r \HH^{n-1}(E_J) \sum_{\ii} \epsilon_{\ii}
    \lesssim C_r \HH^{n-1}(E_J) \epsilon.
  \end{align*}
  Combining this with \eqref{eq:isom_epsilon.5} and \eqref{eq:isom_epsilon.7}, we have 
  \begin{align*}
    \bgg[v]{\bigcup_{\ii \in (\max D) \setminus S} \bigcup_{p \in \LL_{\ii}} p(E(u_{\ii},\epsilon))}
     \lesssim  C_r \sum_{J} \HH^{n-1}(E_J) \epsilon.
  \end{align*} 
  Recall the assumption $N(E) \subset \P^{n-1} \setminus \P^{n-3}_{V^\perp}$ stated in \eqref{eq:kak_isom1.3}, 
  we know that each $x \in E$ belongs to at most two sets of $\{E_J\}_J$,
  which leads to 
  \begin{align*}
    \sum_{J} \HH^{n-1}(E_J) \le 2\HH^{n-1}(E).
  \end{align*}  
  Hence, we have
  \begin{align}\label{eq:isom_epsilon.13}
    \bgg[v]{\bigcup_{\ii \in (\max D) \setminus S} \bigcup_{p \in \LL_{\ii}} p(E(u_{\ii},\epsilon))}
     \lesssim 2C_r \HH^{n-1}(E_{\bigcup\limits_{J} J}) \epsilon \lesssim_r \HH^{n-1}(E) \epsilon,
  \end{align}
  which, combining with \eqref{eq:isom_epsilon.4.1 part1}, leads to \eqref{eq:kak_isom1.1'} and therefore proves \eqref{eq:kak_isom1.1}. 
\end{proof}

Now we focus on \autoref{thm:local_Kak_isom} and show how to move $\HH^{n-1}$-a.e. of a given $(n-1)$-rectifiable set $E$ to an isometric copy $\iota(E)$ for any $\iota \in \Isom(V)$ such that the covered set has $\HH^n$ measure zero.
The main idea is to iterate the $\epsilon$-isometry path established in \autoref{lem:kak_isom_epsilon} on different scales $\{\epsilon_k\}_{k=1}^{\infty}$ and non-decreasing bounded sets $\{E_k\}_{k=1}^{\infty}$ 
satisfying $\lim\limits_{k \to \infty} \epsilon_k = 0$ and $\lim\limits_{k \to \infty}E_k = E$, respectively. 
And the moved sets are carefully chosen such that they converge to the final moved sets, which cover a Lebesgue-null set in total.
To make the final moved sets have the same $\HH^{n-1}$ measure as $E$, 
we need to do more analysis.

\begin{proof}[\textbf{Proof of \autoref*{thm:local_Kak_isom}}]
To show that $E$ is $\Isom(V)$-Kakeya movable, for any given $\iota \in \Isom(V)$, 
we need to find a continuous path $P_{\iota} \subset \Isom(V)$ from $\id$ to $\iota$, 
and sets $\{E_p\}_{p \in P_{\iota}}$ to prove \eqref{eq:proof_local_kak_isom.1} and \eqref{eq:proof_local_kak_isom.2}.
For all $p \in P_{\iota}$, we construct the set as $E_p = \{x\in E \colon \nu_x \cap \aA_p = \emptyset\}$, 
where $\{\aA_p\}_{p \in P_{\iota}} \subset \GG^{n-2}(\P^n)$ are projective axes to be determined.
 
  Without loss of generality, we can assume that $E$ is a single embedded $C^1$ hypersurface, so that the normal line $n_x$ varies continuously on $E$.
  And the condition \eqref{eq:kak_isom1.3} still holds.
  The key point is that we always start by working on simple rotations.
  If $\iota = \tau_z$ is a translation, then we decompose it as $\iota = \iota_2 \circ \iota_1$ with $\iota_1,\iota_2$ being simple rotations, 
  and then work on $\iota_1,\iota_2$ respectively.
  Thus, we can further assume $\iota  = \RR^{\phi}_{z,V}$ is a simple rotation in the following discussions.
  
  Let $\{r_k\}_{k \ge 1}$ be positive numbers, and $\{E_k\}_{k \ge 1}$ be a non-decreasing sequence of compact sets with $E_k \subset B(0,r_k)$ and $\HH^{n-1}(E_k) <\infty$, 
  such that $E = \bigcup_{k\ge 1} E_k$, and each $E_k$ has a continuous normal map.
  Let $N$ be the continuous normal map on $\bigcup_k E_k$.
  We will construct the needed isometry path in $\Isom(V)$ and the moved sets, by iterating \autoref{lem:kak_isom_epsilon}.

  Since $\iota = \RR^{\phi}_{z,V}$ itself is a simple rotation (as we assumed without loss of generality), 
  we chose $P^0\subset \Isom(V)$ to be the natural path from $\id$ to $\iota$, 
  which corresponds to the transformations $\{\RR^{t\phi}_{z,V}\}_{t \in [0,1]}$.
  We can choose a sequence of positive numbers $\{\epsilon_k\}_{k \ge 1}$, such that 
  \begin{align}\label{eq:proof_local_kak_isom.3}
    \sum_{k \ge 1} \HH^{n-1}(E_k)\epsilon_k < \infty.
  \end{align}
  This is always possible, for example, by taking $\epsilon_k = \frac{1}{k^2 \HH^{n-1}(E_k)}$.
  Moreover, we have $\lim\limits_{k \to \infty} \epsilon_k =0$.
  We will apply \autoref{lem:kak_isom_epsilon} to the line segment $P^0$ to obtain a new path $P^1$.
  Repeat this procedure on each line segment to obtain a family of polygonal isometry paths $\{P^k\}_{k \in \N}$,
  where $\epsilon_k$ controls the Lebesgue measure of the set covered by moving subsets of $E_k$ along $P^k$.              
  For simplicity, we use $A^k$ for the index set of line segments in $P^k$ and denote 
  \begin{align*}
    P^k = \bigcup_{i \in A^k} L^k_i,
  \end{align*} 
  where $L^k_i$ is the $i$-th line segment of the polygonal isometry path $P^k \subset \Isom(V)$.                                                                                                          

  \textbf{$\bullet$ Iteration details:}
  Now we describe the iterations of \autoref{lem:kak_isom_epsilon}.
  Roughly speaking, for each $L \subset P^{k-1}$, we can choose a projective line $l_L \in \GG^1(\P^2_{V\times \R})$ and small numbers $\{\epsilon_L, \delta_L\}$ satisfying $\delta_L \in (0, \epsilon_L/2)$.
  Then we use \autoref{lem:kak_isom_epsilon} on $E_k$ to replace $L$ by $K_{E_k,\epsilon_L, \delta_L}(L)$, 
  whose line segments are associated with points on $l_L$ to determine moved sets of the type \eqref{eq:kak_isom1.2}.
  The parameters $\{\delta_L\}$ are carefully determined later.

  For the starting step $P^0\rightarrow P^1$, we associate $P^0$ with a positive number $\epsilon^0_0  < \epsilon_1$ and a sufficiently small number $\delta^0_0 \in (0,\epsilon^0_0/2)$,
  and use \autoref{lem:kak_isom_epsilon} (with a projective line $l^0$) on $E_1$ and $P^0$ to construct the zigzag isometry path $K_{E_1,\epsilon^{0}_0, \delta^0_0}(P^0) \eqcolon P^1$.
  For $k\ge 1$, suppose the path $P^{k-1}=\bigcup_{i \in A^{k-1}} L^{k-1}_i$ and the number $\epsilon_k >0$ have been constructed, 
  we show the construction of $P^{k-1} \rightarrow P^k$ as the following steps.
  We assign positive numbers $\{\epsilon^{k-1}_i, \delta^{k-1}_i\}$ (here the superscript $k-1$ stands for path $P^{k-1}$) for each line segment $L^{k-1}_i \subset P^{k-1}$ such that 
  \begin{align}
    &\delta^{k-1}_i \in (0, \epsilon^{k-1}_i/2), \label{eq:delta_conditon} \\
    &C_{r_k} \bgg{\sum_{i \in A^{k-1}} \epsilon^{k-1}_i} < \epsilon_k, \label{eq:proof_local_kak_isom.4}
  \end{align}
  where $C_{r_k} >1$ is the constant in \eqref{eq:kak_isom1.1} of \autoref{lem:kak_isom_epsilon} corresponding to the set $E_k \subset B(0,r_k)$.
  Next we apply \autoref{lem:kak_isom_epsilon} (with a projective line $l_i^k$) to every line segment $L^{k-1}_i \subset P^{k-1}$ 
  and replace $L^{k-1}_i$ with a zigzag path $K_{E_k,\epsilon^{k-1}_i, \delta^{k-1}_i}(L^{k-1}_i) = \bigcup_{j \in J^{k-1}_i} L_j^k$,
  where, for $i \in J_{i}^{k-1}$, each line segment $L_j^k$ is associated with a direction $u_j^k \in l_i^k$ 
  and a set
  \begin{align}\label{eq:proof_local_kak_isom.4.set}
    E_k(u_j^k, \epsilon^{k-1}_i, \delta^{k-1}_i) = \{x \in E_k \colon \nu_x \cap \aA(B_{l_i^k}(u_j^k, \epsilon^{k-1}_i) \setminus B_{l_i^k}(\cC_{\iota}, \delta^{k-1}_i)) = \emptyset\}.
  \end{align}
  Then \autoref{lem:kak_isom_epsilon} leads to
  \begin{align}\label{eq:proof_local_kak_isom.5}
   \bgg[v]{\bigcup_{j \in J^{k-1}_i} \bigcup_{p \in L_j^k} p(E_k(u_j^k, \epsilon^{k-1}_i, \delta^{k-1}_i))} \lesssim C_{r_k} \epsilon^{k-1}_i \HH^{n-1}(E_k)
  \end{align}
  for all $i \in A^{k-1}$.
  By collecting all the newly constructed zigzags and relabeling the line segments, we obtain the path $P^k$ as
  \vspace{-0.5\baselineskip}
  \begin{align}\label{eq:proof_local_kak_isom.6}
    P^k = \bigcup_{i \in A^{k-1}} K_{E_k,\epsilon^{k-1}_i, \delta^{k-1}_i}(L_i^{k-1}) \eqcolon \bigcup_{j \in A^k} L_j^k,
  \end{align}
  where the index set is $A^k = \bigcup_{i \in A^{k-1}} J^{k-1}_i$.
  
  Then we focus on the set covered by moving along $P^k$.
  For each $j \in A^k$, the line segment $L_j^k \subset P^k$ has been associated with a direction $u_j^k \in l_{i(j)}^k$,
  where for $j \in J_{i}^{k-1}$, the projective line is $l_{i(j)}^k = l_i^k$ as appeared in the procedure of $L^{k-1}_i \rightarrow \bigcup_{j \in J^{k-1}_i} L_j^k$ described above. 
  We define an open subset
  \vspace{-0.5\baselineskip}
  \begin{align}\label{eq:proof_local_kak_isom.7}
    E_j^k \coloneq \{x \in E_k \colon \nu_x \cap \aA(\widebar{B}_{l_{i(j)}^k}(u_j^k, \epsilon_k)) = \emptyset\} \subset E_k
  \end{align}
  as the moved set along $L_j^k$, where $\widebar{B}_{l_{i(j)}^k}(u_j^k, \epsilon_k)$ is the closure of $B_{l_{i(j)}^k}(u_j^k, \epsilon_k)$ with respect to the natural topology of $l_{i(j)}^k \cong P^1$.
  Then $E_j^k$ is a closed subset of $E_k$ (recall that $\nu_x$ is continuously changed, and $\aA(\widebar{B}_{l_{i(j)}^k}(u_j^k, \epsilon_k))$ is closed in $\P^n$).
  It is easy to see that $E_j^k \subset E_k(u_j^k, \epsilon^{k-1}_i, \delta^{k-1}_i)$ since $\epsilon_k > \epsilon^{k-1}_i$,
  which, together with \eqref{eq:proof_local_kak_isom.5}, leads to
  \vspace{-0.5\baselineskip}
  \begin{align*}
    \bgg[v]{\bigcup_{j \in A^k} \bigcup_{p \in L_j^k} p(E_j^k)} 
    \le \bgg[v]{\bigcup_{i \in A^{k-1}}\bigcup_{j \in J^{k-1}_i} \bigcup_{p \in L_j^k} p(E_k(u_j^k, \epsilon^{k-1}_i, \delta^{k-1}_i))}
    \lesssim \sum_{i \in A^{k-1}} C_{r_k} \epsilon^{k-1}_i \HH^{n-1}(E_k).
  \end{align*}
  Thus, applying \eqref{eq:proof_local_kak_isom.4}, this leads to an estimate of the set covered by moving $\{E^k_j\}_{j \in A^k}$ along $P^k$:
  \vspace{-0.5\baselineskip}
  \begin{align}\label{eq:proof_local_kak_isom.8}
     \bgg[v]{\bigcup_{j \in A^k} \bigcup_{p \in L_j^k} p(E_j^k)}  \lesssim \epsilon_k \HH^{n-1}(E_k).
  \end{align}
 
  By now, starting from $P^{k-1}$, we have obtained a finer path $P^k$, 
  the moved sets $\{E_j^k\}_{j \in A^k}$ and an estimate of the covered set, i.e., \eqref{eq:proof_local_kak_isom.6}-\eqref{eq:proof_local_kak_isom.8}. 
  
  \textbf{$\bullet$ Proof of estimate \eqref{eq:proof_local_kak_isom.1}:}
  Next, we take the limit of the above results to obtain the final path and the moved sets needed for our aim in \eqref{eq:proof_local_kak_isom.1}.
  
  First, we determine the final path. 
  In the construction of $P^{k-1} \rightarrow P^k$, all the fineness used in the zigzags can be chosen sufficiently large 
  to make $P^{k}$ arbitrarily close to $P^{k-1}$.
  Then $\{P^k\}_{k \ge 1}$, as a family of continuous isometry paths in $\Isom(V)$, actually forms a Cauchy sequence.
  Thus, we can find a continuous limit path $P_{\iota} \subset \Isom(V)$ connecting $\id$ and $\iota$ such that 
  \begin{align*}
    \lim_{k \to \infty} P^k = P_{\iota}.
  \end{align*}

  Now, we carefully determine the moved set on each step.
  Fix an integer $k \ge 1$, the construction $P^{k-1} \rightarrow P^k$ gives inequality \eqref{eq:proof_local_kak_isom.8}.
  For the next construction $P^k \rightarrow P^{k+1}$, which replace $L_i^k$ with $\bigcup_{j \in J^k_i}L_j^{k+1}$,
  the related fineness parameters are sufficiently large and $\bigcup_{j \in J^k_i}L_j^{k+1}$ stays in a small neighborhood of $L_i^k$.
  Since each $E_i^k$ is compact and $\bigcup_{j \in J^k_i}L_j^{k+1}$ can be regarded as a perturbation of $L_i^k$, 
  we can use \autoref{lem:smallneighborhood} on \eqref{eq:proof_local_kak_isom.8} to obtain
  \begin{align}\label{eq:proof_local_kak_isom.8.1}
    \bgg[v]{\bigcup_{i \in A^k} \bigcup_{j \in J^k_i} \bigcup_{p \in L_j^{k+1}} p(E_i^k)}  \lesssim \epsilon_k \HH^{n-1}(E_k).
  \end{align}
  By introduce an index $i(p,k,k+1) \in A^k$ for each isometry $p \in P^{k+1}$, the above inequality can be restated as
  \begin{align*}
    \bgg[v]{\bigcup_{p \in P^{k+1}} p (E_{i(p,k,k+1)}^k)}  \lesssim \epsilon_k \HH^{n-1}(E_k).
  \end{align*}
  Using similar analysis for the following constructions $P^{k+1} \rightarrow P^{k+2},\ P^{k+2} \rightarrow P^{k+3}, \cdots$, 
  and for any integer $s \ge 1$, we can naturally define $i(p,k,k+s) \in A^k$ for $p \in P^{k+s}$ such that   
  \begin{align}\label{eq:proof_local_kak_isom.9}
    \bgg[v]{\bigcup_{p \in P^{k+s}} p (E_{i(p,k,k+s)}^k)}  \lesssim \epsilon_k \HH^{n-1}(E_k),
  \end{align}
  which is an estimate of moving ``$E_k$'' along $P^{k+s}$.
  By taking $s \to \infty$ and define $i(p,k)\coloneq i(p,k,\infty)$ for $p \in P_{\iota} =\lim\limits_{k \to \infty} P^k$, \eqref{eq:proof_local_kak_isom.9} implies 
  \begin{align}\label{eq:proof_local_kak_isom.10}
    \bgg[v]{\bigcup_{p \in P_{\iota}} p (E_{i(p,k)}^k)}  \lesssim \epsilon_k \HH^{n-1}(E_k),
  \end{align}
  for all $k \ge 1$, which controls the covered set of moving ``$E_k$'' along the final path $P_{\iota}$.
  The index $i(p,k)$ is used to show that for the final zigzag path $P_{\iota}$, 
  we can move ``$E_k$'' (actually $\{E_{i(p,k)}^k\}_{p \in P_{\iota}}$) and make the covered set still holds a desired measure estimate similar to \eqref{eq:proof_local_kak_isom.8}.

  After all these preparations, we are ready to determine the set $E_p$ used in \eqref{eq:proof_local_kak_isom.1}.
  For all $p \in P_{\iota}$, define the moved set as
  \begin{align}\label{eq:proof_local_kak_isom.11}
    E_p \coloneq \varlimsup_{k \to \infty} E_{i(p,k)}^k = \bigcap_{k\ge1}\bigcup_{m\ge k} E_{i(p,m)}^m,
  \end{align}
  which is a Borel set ($G_{\delta\sigma}$ set) and hence $\HH^{n-1}$-measurable.
  For all $k \ge 1$, a simple calculation shows
  \begin{align*}
    \bgg[v]{\bigcup_{p \in P_{\iota}} p(E_p)}
    \le \bgg[v]{\bigcup_{p \in P_{\iota}} p (\bigcup_{m\ge k} E_{i(p,m)}^m)} 
    \le \sum_{m \ge k}\bgg[v]{\bigcup_{p \in P_{\iota}} p (E_{i(p,m)}^m)}
    \lesssim \sum_{m \ge k} \epsilon_m \HH^{n-1}(E_m),
  \end{align*}
  where the last inequality follows from \eqref{eq:proof_local_kak_isom.10}.
  Applying \eqref{eq:proof_local_kak_isom.3}, we have 
  \begin{align}\label{eq:proof_local_kak_isom.12}
    \bgg[v]{\bigcup_{p \in P_{\iota}} p(E_p)} \lesssim \sum_{m \ge k} \epsilon_m \HH^{n-1}(E_m) \xrightarrow{k \rightarrow \infty} 0,
  \end{align}
  which is of the desired type as estimate \eqref{eq:proof_local_kak_isom.1}.

  \textbf{$\bullet$ Proof of equation \eqref{eq:proof_local_kak_isom.2}:} 
  Applying \eqref{eq:proof_local_kak_isom.7} to \eqref{eq:proof_local_kak_isom.11}, we have
  \begin{align}\label{eq:proof_local_kak_isom.12.1}
    E_p = \varlimsup_{k \to \infty} \{x \in E_k \colon \nu_x \cap \aA(\widebar{B}_{l_j^k}(u_j^k, \epsilon_k)) = \emptyset\}.
  \end{align}
  By taking the complement with respect to $E = \bigcup_k E_k$, this gives
  \begin{align*}
    E_p^{\mathsf{c}} 
    &= E \setminus \Big[\varlimsup_{k \to \infty} \{x \in E_k \colon \nu_x \cap \aA(\widebar{B}_{l_{i(p,k)}^k}(u_{i(p,k)}^k, \epsilon_k)) = \emptyset\} \Big]  \\
    &= \varliminf_{k \to \infty} \{x \in E_k \colon \nu_x \cap \aA(\widebar{B}_{l_{i(p,k)}^k}(u_{i(p,k)}^k, \epsilon_k)) \neq \emptyset\} \\
    &\subset \{x \in E \colon \nu_x \cap \aA(\varliminf_{k \to \infty} \widebar{B}_{l_{i(p,k)}^k}(u_{i(p,k)}^k, \epsilon_k)) \neq \emptyset\},
  \end{align*}
  where the last step uses the properties of the operator $\aA$ (see \autoref{def:center_axis_corrspd} and the discussions there).

  Since $\{\widebar{B}_{l_{i(p,k)}^k}(u_{i(p,k)}^k, \epsilon_k)\}_k$ is a family of compact projective segments with $\lim_{k \to \infty} \epsilon_k =0$ (see \eqref{eq:proof_local_kak_isom.3}),
  we know that $\varliminf_{k \to \infty} \widebar{B}_{l_{i(p,k)}^k}(u_{i(p,k)}^k, \epsilon_k)$ has at most one point.
  We denote $u_p \in \P^2_{V\times \R}$ for this limit inferior (when it's an empty set, we take $u_p$ as an arbitrary point in $\P^2_{V\times \R}$).
  Then we have 
  \begin{align}\label{eq:proof_local_kak_isom.12.2}
    E_p^{\mathsf{c}}  \subset \{x \in E \colon \nu_x \cap \aA(u_p) \neq \emptyset\}.
  \end{align}
  Let $E(u) \coloneq \{x \in E \colon \nu_x \cap \aA(u) \neq \emptyset\}$ for $p \in \P^2_{V\times \R}$.
  To prove \eqref{eq:proof_local_kak_isom.2}, we only need to show 
  \begin{align}\label{eq:proof_local_kak_isom.13}
    \HH^{n-1}(E(u_p)) =0, \quad \forall p \in P_{\iota}. 
  \end{align}
  This is possible for some carefully selected points $\{u_p\}_{p \in P_{\iota}}$, which we explain as the following.

  Consider the set $A \coloneq \{u \in \P^2_{V\times\R} \colon \HH^{n-1}(E(u))>0\}$, 
  which we claim has at most countable points.
  Let's consider the intersection $E(u)\cap E(u')$ for different $u,u' \in \P^2_{V\times \R}$.
  We consider the intersection in two cases:
  \begin{enumerate}[label=(\arabic*)]
    \item If $u,u' \in \P^1_{V\times\{0\}}$, then $\aA(u), \aA(u') \subset \P^{n-1}_{\infty}$ and they satisfy $\aA(u)\cap \aA(u') = \P^{n-3}_{V^{\perp}\times \{0\}}$.
          Since $\nu_x  \cap \P^{n-1}_{\infty}$ is exactly the projective normal direction $N(x)$, we have 
          \begin{align*}
            E(u)\cap E(u') = \{x \in E \colon \nu_x \cap (\aA(u) \cap \aA(u')) \neq \emptyset \} = N^{-1}(\P^{n-3}_{V^{\perp}}).
          \end{align*}
          Together with the assumption \eqref{eq:kak_isom1.3}, this set must be empty.
    \item If at least one of the points belongs to $\P^2_{V\times \R} \setminus \P^1_{V\times \{0\}}$, 
          then $\spn_{\P^n} \{\aA(u),\aA(u')\} = \P^{n-1}_{Y \times \R}$ for some affine subspace $Y \in \AA^{n-1}(\R^n)$ determined by $u,u'$ and $V^{\perp}$.
          Since any two different projective points determine a projective line, we have (recall that $\nu_x = [n_x:1]$)
          \begin{align*}
            E(u)\cap E(u') \subset \{x \in E \colon \nu_x \subset \spn_{\P^n} \{\aA(u),\aA(u')\} \} = \{x \in E \colon n_x \subset Y\},
          \end{align*}
          where the first inclusion uses the assumption \eqref{eq:kak_isom1.3}.
          The righthand side corresponds to a subset of $E$ which is transversally intersected by $Y$, and hence is $\HH^{n-1}$-null.
  \end{enumerate}
  These two cases imply that $\{E(u)\}_{u \in A}$ is a family of positive $\HH^{n-1}$-measure subsets whose pairwise intersections are $\HH^{n-1}$-null.
  Since $E \in \Rect$ has $\sigma$-finite $\HH^{n-1}$-measure, it admits at most countably many subsets of positive $\HH^{n-1}$-measure whose pairwise intersections are $\HH^{n-1}$-null.
  Hence the index set $A =\{u \in \P^2_{V\times\R} \colon \HH^{n-1}(E(u))>0\}$ is at most countable, which we denote by $A = \{u_k\}_{k \in \N}$.
  Next, using condition \eqref{eq:kak_isom_refine1.0} from \autoref{lem:kak_isom_epsilon}, 
  we can select the projective lines $\{l^k_i\}$ and parameters $\{\delta^k_i\}$ in \eqref{eq:delta_conditon} to guarantee that
  \begin{align*}
    u_p \not\in \{u_k\}_{k \in \N},\quad \forall p \in P_{\iota},
  \end{align*}
  see \cite[Section 6.3]{ChCs2019} for more details.
  This special choice of $\{u_p\}_{p \in P_{\iota}}$ also guarantees
  \begin{align*}
    \HH^{n-1}(E(u_p)) = \HH^{n-1}(\{x \in E \colon \nu_x \cap \aA_p \neq \emptyset\}) =0,
  \end{align*} 
  which is exactly \eqref{eq:proof_local_kak_isom.13}, and therefore proves \eqref{eq:proof_local_kak_isom.2} as desired.
\end{proof}

\subsection{Proof of \autoref*{thm:global_Kak_isom}}\label{ssect:proof_glo_Kak_isom}

Recall \autoref{lem:decomp_isometry}, for any isometry $\iota  \in \Isom(\R^n)$ with $\iota \neq \id$, 
we can find a positive integer $k \le 2+\floor{\frac{n}{2}}$, some planes $\{V_i\}_{i=1}^k \subset \GG^2(\R^n)$, and a simple rotation $\rho_i \in \Isom(V_i) \setminus \{\id\}$ for each $i =1,\dots,k$,
such that 
\begin{align*}
  \iota = \rho_k \circ \cdots \circ \rho_1.
\end{align*} 
For each $i =1, \dots, k$, from \autoref{thm:local_Kak_isom}, we can construct a path $P_i \subset \Isom(V_i)$ from $\id$ to $\rho_i$, 
and subsets $\{\wt{E}_p\}_{p \in P_i}$ of $E$, such that 
\begin{align*}
  \abs{\bigcup_{p \in P_i} p(\wt{E}_p)}=0; \quad \HH^{n-1}(\wt{E}_p) = \HH^{n-1}(E), \ \forall p \in P_i.
\end{align*}
Then, using the idea stated in \autoref{ssect:path_Isom_R^n}, 
we can connect these piecewise paths together to obtain a new path $P \coloneq P_k \circ \cdots \circ P_2 \circ P_1 \subset \Isom (\R^n)$, 
and there are subsets $\{E_p\}_{p \in P}$ satisfying
\begin{align*}
  \abs{\bigcup_{p \in P} p(E_p)}=0; \quad \HH^{n-1}(E_p) = \HH^{n-1}(E), \ \forall p \in P,
\end{align*}
which proves \autoref{thm:global_Kak_isom}.

\section{Nikodym type set: proof of \autoref*{thm:isom_Nikodym}}\label{sect:Nikodym}

In this section we consider the Nikodym type set associate with $E \in \Rect$ and prove \autoref{thm:isom_Nikodym}.
We first define our Nikodym type set using the transformations under consideration.

\begin{definition}[$G$-Nikodym type set]\label{def:G-Nikodym}
  Let $E \subset \mathbb{R}^n$ be an $s$-dimensional set with $\HH^s(E) >0$, $G$ be a collection of transformations of $\R^n$.
  A set $F\subset \mathbb{R}^n$ is said to be a $G$-Nikodym type set associated with $E$, if it satisfies the following conditions:
    \begin{enumerate}
          \item $\HH^{n}(F) = 0$;
          \item for every $x \in \mathbb{R}^n$, there exist a transformation $g_x \in G$, and a subset $E_x \subset E$ such that 
                \begin{enumerate}[label=(\alph*)]
                  \item $g_x(E)$ passes through $x$,
                  \item $\HH^s(E_x) = \HH^s(E)$ and $g_x(E_x) \subset F$.
                \end{enumerate}
    \end{enumerate}
\end{definition}

\begin{remark}
  For a $k$-dimensional subspace $W \in \GG^k(\R^n)$, it is easy to check that the Nikodym type set constructed by Falconer \cite{Fal1986} is a $\Isom(\R^n)$-Nikodym type set associate with $W$.
  This also explains why the $G$-Nikodym type set can be regarded as the generalization of classical Nikodym set.
\end{remark}

From the condition ($\mathrm{a}$), the $G$-Nikodym type set associated with $E$ requires some $G$-transformed copy of $E$ to pass through some point in $\R^n$.
This  property suggests that $G$ should contain enough transformations. 
In this section we consider the case $G = \Isom(\R^n)$ and construct the $\Isom(\R^n)$-Nikodym type set associated with $(n-1)$-rectifiable sets.
Later, \autoref{sect:Aff} contains a brief discussion of $\Aff(\R^n)$-Nikodym type set associated with $(n-1)$-rectifiable sets for $n \ge 3$.

For $E \in \Rect$, just as stated in \autoref{rmk:specialcase}, 
we can replace the set $E$ in condition (a) of \autoref{def:G-Nikodym} by any countable disjoint union $\Gamma$ of embedded $C^1$ hypersurfaces.
We will prove the existence of such a slightly more general version of the Nikodym type set.

The argument used here is similar to that in the planar setting of \cite[Theorem 6.9]{ChCs2019}.
Using the result established in \autoref{thm:local_Kak_isom}, we can apply their method to the higher-dimensional setting. 

We first prove that for any countable disjoint union $\Gamma\subset\R^n$ of embedded $C^1$ hypersurfaces, 
the set swept out by $\Gamma$ along a path obtained from \autoref{thm:local_Kak_isom} has nonempty interior.

\begin{lemma}\label{lem:rot-create-interior}
Suppose $\Gamma \subset \R^n$ is a countable disjoint union of embedded $C^1$ hypersurfaces.
Then we can find $V \in \GG^2(\R^n)$, a rotation $\iota \in \Isom(V)$, 
and a path $P_{\iota} \subset \Isom(V)$ connecting $\id$ and $\iota$ as in \autoref{thm:local_Kak_isom},
such that moving the entire set $\Gamma$ along $P_{\iota}$ covers a set with nonempty interior, i.e.,
  \begin{align}\label{eq:nonemptyinter_gamma}
    \inte(\bigcup_{p \in P_{\iota}} p(\Gamma)) \neq \emptyset.
  \end{align}
\end{lemma}

\begin{proof}
Without loss of generality, we can assume $\Gamma$ is an embedded $C^1$ hypersurface with
a pointwise-defined normal map $N \colon \Gamma \to \P^{n-1}$. Let $N(\Gamma)$ denote its set of normal directions.
As before, $n_x$ and $\nu_x = [n_x:1]$ stands for the normal line and projective normal line to $\Gamma$ at $x\in \Gamma$, respectively.

\textbf{$\bullet$ Case 1:}
If $N(\Gamma) = \{\beta\}$, i.e., $\Gamma$ is contained in a hyperplane with normal direction $\beta \in \P^{n-1}$. 
We can take $V$ to be a 2-plane containing the normal direction $\beta$. 
Then any isometry in $\Isom(V)$ (except for the translation along direction $\RR_{0,V}^{\frac{\pi}{2}}(\beta)$ when $n=3$) will cover a set of positive Lebesgue measure.

\textbf{$\bullet$ Case 2:}
If $\Gamma$ has at least two normal directions.
Let $\beta \in N(\Gamma)$, 
and choose $V \in \GG^2(\R^n)$ containing this direction, i.e., $\beta \in \P^1_V$.
Then we have 
\begin{align}\label{eq:inter_analy0}
  N(\Gamma) \setminus \P^{n-3}_{V^\perp} \neq \emptyset.  
\end{align}

Now we consider the set $A = \{\iota' \in \Isom(V) \colon \abs{\bigcup_{p \in P_{\iota'}} p(\Gamma)} = 0\}$.
We claim that, under the condition \eqref{eq:inter_analy0}, the set $A$ contains  isometries of at most one type
(we regard $\tau_z$ and $\tau_{tz}$ ($t \neq 0$) as the same type; similarly for $\RR^{\theta}_{z,V}$ and $\RR^{t\theta}_{z,V}$).
Recall that 
\begin{align}\label{eq:inter_analy1}
  \bgg[v]{\bigcup_{p \in P_{\iota}} p(\Gamma)} = 0 \iff \nu_x \cap \aA_{\iota} \neq \emptyset,\ \forall x\in \Gamma,
\end{align}
which can be seen from the proof of \autoref{lem:smallmeasisom} and \autoref{lem:bettermeas_isom}.
If we can take $\iota_1, \iota_2 \in A$ of different types, we shall lead to contradictions.
We consider this in different cases.

$\bullet$
If $\tau_1, \tau_2 \in A$ are translations with different directions,  
then combining \eqref{eq:inter_analy1} and the fact that $\aA_{\iota_1}$ and $\aA_{\iota_2}$ are actually different projective subspaces in $\P^{n-1}$ with intersection $\P^{n-3}_{V^\perp}$,
we have $N(x) \in \P^{n-3}_{V^\perp}$ holds for all $x \in \Gamma$.
This contradicts to \eqref{eq:inter_analy0}. 
Thus, if $A$ contains translations, then the translations must have same direction, and therefore $A$ has essentially at most one translation.

$\bullet$
If $\iota_1, \iota_2 \in A$ are of different types, then at least one of them is a simple rotation. 
Assume $\iota_1$ is a rotation with projective axis $\aA_{\iota_1} \subsetneq \P^{n-1}_{\infty}$, and $\iota_2$ has a projective axis $\aA_{\iota_2} \neq \aA_{\iota_1}$.
From \eqref{eq:inter_analy1} and the fact ``two distinct points determine a line'', 
we have, for all $x \in \Gamma$, either
\begin{enumerate}[label=(\alph*)]
  \item $\nu_x \subset \spn_{\P^n} \{\aA_{\iota_1}, \aA_{\iota_2}\}$, or
  \item $N(x) \in \P^{n-3}_{V^\perp}$,
\end{enumerate} 
where $\spn_{\P^n} \{\aA_{\iota_1}, \aA_{\iota_2}\} \in \GG^{n-1}(\P^n)$ contains all the projective lines that touch both $\aA(\iota_1)$ and $\aA(\iota_2)$, 
and can be expressed as $\spn_{\P^n} \{\aA_{\iota_1}, \aA_{\iota_2}\} = \P^{n-1}_{Y\times\R} $ for some $(n-1)$-dimensional affine subspace $Y \subset \R^n$.
The condition \eqref{eq:inter_analy0} avoids the case $(b)$, thus $(a)$ holds for all $x \in \Gamma$, i.e.,
\begin{align*}
  \nu_x = [n_x:1] \subset \spn_{\P^n} \{\aA_{\iota_1}, \aA_{\iota_2}\} = \P^{n-1}_{Y\times\R},\ \forall x \in \Gamma.
\end{align*}
This is equivalent to say $n_x \subset Y$ for all $x \in \Gamma$, 
which is impossible since the hyperplane $Y$ can not contain both the whole embedded $C^1$ hypersurface $\Gamma$ and all its normal lines.
Now we have proved that $A$ has essentially one type of isometries in $\Isom(V)$ (translations with fixed direction or rotations with fixed axis).

Combining Case 1 and Case 2, we have shown that for the 2-plane $V$ we selected (i.e.,  $V$ contains some direction $\beta \in N(\Gamma)$),
all isometries in $\Isom(V)$ (with possibly one type being excluded) acting on $\Gamma$ will cover a set with positive Lebesgue measure.
Thus, for a simple rotation $\iota \in \Isom(V)$ and the isometry path $P_{\iota}$ from $\id$ to $\iota$ constructed as in \autoref{thm:local_Kak_isom} (which involves many different simple rotations),
we have 
\vspace{-0.5\baselineskip}
\begin{align*}
 \bgg[v]{\bigcup_{p \in P_{\iota}}p(\Gamma)} >0,
\end{align*}
which implies that $\bigcup_{p \in P_{\iota}}p(\Gamma)$ must have nonempty interior.
\end{proof}

\vspace{-0.5\baselineskip}
Now we are ready to prove \autoref{thm:isom_Nikodym}, 
which essentially relies on \autoref{thm:local_Kak_isom} and \autoref{lem:rot-create-interior}.
\vspace{-0.5\baselineskip}

\begin{proof}[\textbf{proof of \autoref*{thm:isom_Nikodym}}]
  For a countable disjoint union $\Gamma$ of embedded $C^1$ hypersurfaces, from \autoref{lem:rot-create-interior}, 
  we can find a path $P_{\iota}$ (constructed as in \autoref{thm:local_Kak_isom})
  such that \eqref{eq:nonemptyinter_gamma} holds.
  Then we can cover $\R^n$ by countably many translated copies of $\bigcup_{p \in P_{\iota}}p(\Gamma)$, for example,
  \begin{align}\label{eq:6.2}
    \R^n = \bigcup_{q \in \Q^n} \bigcup_{p \in P_{\iota}}(q+p(\Gamma)).
  \end{align}
  For a rectifiable set $E \in \Rect$ with $\HH^{n-1}(E)< \infty$, let
  \begin{align*}
    F \coloneq \bigcup_{q \in \Q^n} \bigcup_{p \in P_{\iota}}(q+p(E_p)),
  \end{align*}
  where $\{E_p\}_{p \in P_{\iota}}$ are equimeasurable subsets of $E$ obtained from \autoref{thm:local_Kak_isom}.
  Thus
  \begin{align*}
    \abs{F} \le \sum_{q \in \Q^n} \bgg[v]{\bigcup_{p \in P_{\iota}}(q+p(E_p))} =0.
  \end{align*}
  For any $y \in \R^n$, from \eqref{eq:6.2}, we can find a point $q_y \in \Q^n$ and a transformation $p_y \in P_{\iota}$, 
  such that $y \in q_y+p_y(\Gamma) \eqcolon \iota_y(\Gamma)$, where $\iota_y = \tau_{q_y}\circ p_y \in \Isom(\R^n)$.
  Thus, $\iota_y(\Gamma)$ is an isometric copy of $\Gamma$ which passes through $y\in \R^n$.
  Now we can define $E_y \coloneq E_{p_y}$, 
  the equimeadurable property of $\{E_p\}_{p \in P_{\iota}}$ stated in \autoref{thm:local_Kak_isom} then leads to
  \begin{align*}
    \HH^{n-1}(\iota_y(E_y)) = \HH^{n-1}(E).
  \end{align*}
  Also, the construction of $F$ guarantees $\iota_y(E_y) = q_y+p_y(E_{p_y}) \subset F$.

  Thus, by taking $\Gamma$ as the regular part $\wt{E}$ of $E$, we have proved \autoref{thm:isom_Nikodym}.
\end{proof}

\section{Further generalizations to affine transformations}\label{sect:Aff}

In this section we consider the affine transformations of $\R^n$ acting on $(n-1)$-rectifiable sets, where $n\ge3$.
This is inspired by the work on planar similarities in \cite[Section 7]{ChCs2019}.
Many of the techniques used in \autoref{sect:preliminaries}, \autoref{sect:Kak_isom} and \autoref{sect:Nikodym} can be adapted to general affine transformations in $\R^n$.

In \cite[Section 7]{ChCs2019}, 
the authors identify $\R^2$ with $\mathbb{C}^1$ and consider the following transformation:
\begin{align*}
  \sigma_{z,\phi,\alpha} \colon 
  \mathbb{C}^1 \to \mathbb{C}^1, \
  u \mapsto e^{-\phi\sin \alpha}e^{\mathrm{i}\mkern1mu \phi\cos\alpha}(u-z)+z,
\end{align*}
where $z\in \mathbb{C}^1$, $\alpha \in [0,2\pi)$ and $\phi \in \R$.
Then they use such transformations to analyze the orientation-preserving similarity group $\Sim(\R^2)$.
Our discussions of affine transformations in $\R^n$ ($n \ge 3$) are inspired by their planar work.
And it turns out that in higher-dimensional spaces, with the flexibility of choosing 2-planes, 
we can generalize their construction to the broader class of affine transformations instead of similarity transformations.

We use $\Aff(\R^n)$ to denote the group of all orientation-preserving affine transformations in $\R^n$, 
whose elements can be parameterized by $\GL^{+}_n(\R)\times \R^n$, where $\GL^{+}_n(\R) =\{A \in \GL_n(\R)\colon \det A >0\}$.
More precisely, a typical affine transformation (or affinity for simplicity) of $\R^n$ can be expressed as
\begin{align*}
  \sigma_{A,z} \colon 
    \R^n \to \R^n, \
    x \mapsto Ax+z,
\end{align*} 
for $(A,z) \in \GL^{+}_n(\R)\times \R^n$.
We also use $\Sim(\R^n)$ to express the group generated by the similarity transformations of $\R^n$, which contains all isotropic dilations, rotations, translations and their compositions.
$\Sim(\R^n)$ is a subgroup of $\Aff(\R^n)$, 
where the latter contains orientation-preserving anisotropic dilations (i.e., $x=(x_1, \dots, x_n) \mapsto (\lambda_1x_1, \dots, \lambda_nx_n)$, where $\lambda_i >0$ for $i=1,\dots,n$).

First we consider the following transformations related to the planar similarities.

\begin{definition}[$V$-spiral transformation]
  Let $V \in \GG^2(\R^n)$, $z \in V$, $\alpha \in [0,2\pi)$ and $\phi \in \R$. 
  We define the $V$-spiral transformation $\sigma^{\phi,\alpha}_{z,V}$ by 
  \begin{align*}
    \sigma^{\phi,\alpha}_{z,V} \colon 
    \R^n \to \R^n, \
    x \mapsto e^{-\phi\sin \alpha} \RR^{\phi\cos \alpha}_{z,V}(x) + (1-e^{-\phi\sin \alpha})P_{V^\perp}(x),
  \end{align*} 
  where $\RR^{\phi\cos \alpha}_{z,V}$ is a simple rotation as defined in \autoref{def:V_rotat}.
  The affine subspace $z+V^\perp$ is fixed pointwise by $\sigma^{\phi,\alpha}_{z,V}$, and thus serves as the axis of the transformation.
\end{definition}

We introduce the following nonlinear transformation to study the $V$-spiral transformation.
\begin{definition}
  For $V \in \GG^2(\R^n)$, $z \in V$ and $\alpha \in (0,2\pi)\setminus\{\pi\}$, we define
  \begin{align*}
    H^{\alpha}_{z,V} \colon 
    \R^n \to \R^n, \
    x \mapsto H^{\alpha}_{z,V}(x)=
    \begin{cases}
      x , &\text{if} \ x \in z+V^{\perp}, \\
      \RR_{z,V}^{\cot\alpha \ln\abs{P_V(x-z)}} , &\text{if} \ x \not\in z+V^{\perp}.
    \end{cases}
  \end{align*}
\end{definition}
This transformation has several useful properties:
\begin{enumerate}
  \item $H^{\alpha}_{z,V}$ is a diffeomorphism from $\R^n\setminus(z+V^{\perp})$ to itself;
  \item $\abs{\det DH^{\alpha}_{z,V}}=1$ for any $x \not\in z+V^{\perp}$, hence $\abs{H^{\alpha}_{z,V}(A)}=\abs{A}$ for any $A \subset \R^n$;
  \item For $\sigma^{\phi,\alpha}_{z,V}$ with $\sin\alpha \neq 0$, we have $H^{\alpha}_{z,V}\circ\sigma^{\phi,\alpha}_{z,V}(x) = z+e^{-\phi\sin\alpha}P_V(H^{\alpha}_{z,V}(x)-z)+P_{V^\perp}(x)$.
\end{enumerate}
The first two properties can be checked by expressing the vector $P_V(x)$ in polar coordinates centered at $z \in V$.
The third property essentially shows that $H^{\alpha}_{z,V}$ transforms the logarithmic spiral trajectory into a line segment.

\begin{remark}
  For $n=2$, in \cite[Section 7]{ChCs2019}, the authors introduced a measure-preserving map $\Psi \colon \R^2 \to \R^2$ 
  that rotates each circle $\abs{u-z} = \text{const}$ around $z$ by angle $\cot\alpha \log \abs{u-z}$.
  The measure-preserving map $H^{\alpha}_{z,V}$ is the higher-dimensional analog of $\Psi$ for the transformation $\sigma^{\phi,\alpha}_{z,V}$.
\end{remark}

Analogously to \autoref{def:Isom(V)}, we define the following transformation groups:

\begin{definition}[Simple affinity in $V$]
  For fixed $\alpha\in[0,2\pi)$, all $V$-spiral transformations with parameter $\alpha$ and translations in $V$ form the set
  \begin{align*}
    \Sim_{\alpha}(V) \coloneq \{\sigma^{\phi,\alpha}_{z,V} \colon z \in V, \phi \in \R\} \cup \{\tau_z \colon z \in V\}.
  \end{align*}
  All $V$-spiral transformations and translations in $V$ are called simple affinities in $V$, and form the group
  $\Sim(V) \coloneq \bigcup_{\alpha \in [0,2\pi)} \Sim_{\alpha}(V)$.
\end{definition}

Note that $\sigma^{\phi,0}_{x,V} =\sigma^{\phi,\pi}_{x,V} = \RR^{\phi}_{x,V}$,
so all $V$-spiral transformations with $\alpha \equiv 0 \pmod{\pi}$ are simple rotations in $V$.
For $\alpha \equiv \frac{\pi}{2} \pmod{\pi}$, the $V$-spiral transformations correspond to isotropic dilations acting in $V$ given by $\sigma^{\phi,\frac{\pi}{2}}_{x,V}(y) = e^{-\phi}P_V(y-x) +P_{V^\perp}(y)$. 
Thus, $\Sim(V)$ is the lifting of planar similarity transformations (in the 2-plane $V$) to affinities of $\R^n$, which also explains the notation ``$\Sim$''.
The set $\Sim_{\alpha}(V)$ can be regarded as the simple affinities in $V$ with similarity factor $\alpha$.
 
Analogously to \autoref{def:proaxis_isom}-\autoref{def:labelby_V}, we can introduce the projective representations and $V\times\R$-parametrization for the simple affinity $\sigma \in \Sim_{\alpha}(V)$. 
More precisely, let $V \in \GG^2(\R^n)$ and $\alpha \in [0,2\pi)$. 
For $\sigma \in \Sim_{\alpha}(V)\setminus\{\id\}$,
we define its projective axis $\wt{\aA}_{\sigma}$ and $V\times\R$-parametrization $\wt{\eta}$ as
  \begin{align*}
    &\wt{\aA}_{\sigma} \coloneq  
    \begin{cases}
      [z+V^{\perp}:1] , &\text{if} \ \sigma = \sigma^{\phi,\alpha}_{z,V}, \\
      [(z)^{\perp}:0], &\text{if} \ \sigma = \tau_z;
    \end{cases} \\
    &\wt{\eta}(\sigma)  \coloneq 
        \begin{cases}
          (\phi z,\phi), &\text{if} \ \sigma = \sigma^{\phi,\alpha}_{z,V}, \\
          (\RR^{\frac{\pi}{2}}_{0,V}(z),0), &\text{if} \ \sigma = \tau_z.
        \end{cases}
  \end{align*}
  Then the projective center is 
  $\wt{\cC}_{\sigma} \coloneq  \wt{\aA}_{\sigma} \cap \P^2_{V\times\R}$.
  The axis-lifting operator $\aA \colon \mathscr{P}(\P^2_{V\times \R}) \to \mathscr{P}(\P^n)$ also applies in this setting.
The parametrization also has a decomposition property analogous to \autoref{lem:approx_N_term}, 
and the Venetian blind iteration method described in \autoref{sect:VBandzigzag} can be applied to construct paths of affinities in $\Sim_{\alpha}(V)$.

Generally, we can reduce any affinity into simple affinities, which leads to the following:

\begin{lemma}\label{lem:decomp_affinity}
  For $n \ge 3$, any orientation-preserving affine transformation in $\Aff(\R^n)$ can be decomposed as the composition of finitely many simple affinities associate with 2-planes and similarity factors.
  Equivalently, $\Aff(\R^n)$ is generated by $\{\Sim_{\alpha}(V)\}_{V\in \GG^2(\R^n), \alpha \in [0,2\pi)}$.
\end{lemma}
This lemma is an analog of \autoref{lem:decomp_isometry}; we give a proof sketch. 
For any $A \in \GL^{+}_n(\R)$, 
by the singular value decomposition of $A$, 
we can find $U,W \in \SO(n)$ and a diagonal matrix $\Lambda$ with positive diagonal elements such that $A = U \Lambda W$.
Since all rotations and translations in $\R^n$ can be expressed by simple rotations (as proved in \autoref{lem:decomp_isometry}),
and all simple rotations are naturally spiral transformations, we only need to handle the dilation part corresponding to the matrix $\Lambda$.
We can always decompose $\Lambda$ into $n$ simpler diagonal matrices, each corresponding to an isotropic dilation in a 2-plane;
then the dilation part can be expressed as the composition of $n$ spiral transformations.
Here we give a concrete decomposition of $\Lambda$ for $n=3$. Suppose the diagonal elements of $\Lambda$ are $\lambda_1, \lambda_2, \lambda_3 >0$, then
\begin{align*}
  \begin{pmatrix}
    \lambda_1 & & \\
    & \lambda_2 & \\
    & & \lambda_3
  \end{pmatrix}
  = 
  \begin{pmatrix}
    \sqrt{\frac{\lambda_1\lambda_2}{\lambda_3}} & & \\
    & \sqrt{\frac{\lambda_1\lambda_2}{\lambda_3}} & \\
    & & 1
  \end{pmatrix}
  \begin{pmatrix}
    1 & & \\
    & \sqrt{\frac{\lambda_2\lambda_3}{\lambda_1}} & \\
    & & \sqrt{\frac{\lambda_2\lambda_3}{\lambda_1}}
  \end{pmatrix}
  \begin{pmatrix}
    \sqrt{\frac{\lambda_1\lambda_3}{\lambda_2}} & & \\
    & 1 & \\
    & & \sqrt{\frac{\lambda_1\lambda_3}{\lambda_2}}
  \end{pmatrix},
\end{align*}
where each matrix on the right-hand side corresponds to a spiral transformation.
This decomposition can be extended to all $n \ge 3$ in a straightforward manner, so we omit the details of the proof.

Now we introduce measure estimates for sets moved by $V$-spiral transformations. 
These are analogs of \autoref{lem:measisom} and \autoref{lem:smallmeasisom}.
For a simple affinity $\sigma$, let $L^{\sigma}$ denote its natural path from $\id$ to $\sigma$.

\begin{lemma}[trivial measure estimate for simple affinity]\label{lem:measaff}
  Let $V \in \GG^2(\R^n)$ and $\sigma\in \Sim(V)$ be a simple affinity in $V$.
  Suppose $E \subset B(0,r)$ is a bounded $(n-1)$-rectifiable set in $\R^n$,
  then the measure of the set covered by moving $E$ along $\sigma$ satisfies
  \begin{align}
    \abs{\bigcup_{p \in L^{\sigma}} p(E)} \lesssim_r \abs{\wt{\eta}(\sigma)} \HH^{n-1}(E),
  \end{align}
  where $\wt{\eta}(\sigma) \in V\times \R$ is defined as above.
\end{lemma}

\begin{lemma}[small measure estimate for simple affinity]\label{lem:smallmeasaff}
  Let $V \in \GG^2(\R^n)$, $\alpha \in [0,2\pi)$. 
  Suppose $E \in \Rect$ satisfies $E \subset B(0,r)$, and let $\delta >0$ be sufficiently small (depending on $r$).
  Let $\sigma \in \Sim_{\alpha}(V)$ be a simple affinity with projective axis $\wt{\aA}_{\sigma} \in \GG^{n-2}(\P^n)$.
  Let $n_x$ be the normal line at $x \in E$, and set $(\nu_x)_{\alpha}=[\RR^{\alpha}_{x,V}(n_x):1]$ to be the rotated projective normal line.
  Denoting 
  \vspace{-0.3\baselineskip}
  \begin{align*}
    E^{\sigma, \delta} \coloneq \{x \in E \colon (\nu_x)_{\alpha} \cap B_{\P^n}(\wt{\aA}_{\sigma},\delta) \neq \emptyset\} ,
  \end{align*}
  \vspace{-0.3\baselineskip}
  then we have
\vspace{-0.3\baselineskip}
  \begin{align}
    \abs{\bigcup_{p \in L^{\sigma}} p(E^{\sigma, \delta})} 
     \lesssim_{\alpha,r} \delta \abs{\wt{\eta}(\sigma)} \HH^{n-1}(E^{\sigma, \delta}).
  \end{align}
\end{lemma}

The proofs of \autoref{lem:measaff} and \autoref{lem:smallmeasaff} require an analog of \autoref{lem:bettermeas_isom}, 
with the new function 
\begin{align*}
  \Phi \colon U_i\times [0,1] \to \R^n,\  (s,t) \mapsto \sigma^{t\phi,\alpha}_{z,V}(x(s))
\end{align*}
replacing $\Phi_2$.
Assume, without loss of generality, that $\sigma = \sigma^{\phi,\alpha}_{z,V}$ with $\sin\alpha \neq 0$; 
otherwise $\sigma$ is an isometry and we are reduced to \autoref{lem:measisom} and \autoref{lem:smallmeasisom}  
(note that $(\nu_x)_\alpha = \nu_x$ when $\alpha = \pi/2$ or $3\pi/2$).
Using the measure-preserving map $H^{\alpha}_{z,V}$ and setting $H \coloneq H^{\alpha}_{z,V}\circ\Phi$, 
the measure of the set covered by moving $E_i$ (a continuous piece of $E$) along $\sigma$ can be estimated as
\begin{align*}
  \abs{\bigcup_{p \in L^{\sigma}}p(E_i)} = \abs{\Phi(U_i\times [0,1])} = \abs{H(U_i\times [0,1])}.
\end{align*}
We therefore study the map
\begin{align*}
  H \colon 
  U_i\times [0,1] \to \R^n, \ 
  (s,t) \mapsto H^{\alpha}_{z,V}\circ\sigma^{t\phi,\alpha}_{z,V}(x(s)) = z+e^{-t\phi\sin\alpha}P_V(H^{\alpha}_{z,V}(x(s))-z)+P_{V^\perp}(x(s)).
\end{align*}
Since $E_i\cap (z+V^{\perp})$ is $\HH^{n-1}$-null, $H$ is differentiable $\HH^n$-almost everywhere.
Then the area formula yields 
\vspace{-0.5\baselineskip}
\begin{align*}
  \abs{H(U_i\times [0,1])} 
  \le \int_{\R^n} N(H,y) \dd \HH^{n}(y) 
  = \int_{U_i\times [0,1]} \mathrm{J}_n H(s,t) \dd(s,t).
\end{align*}
A computation analogous to \eqref{eq:l2.2}-\eqref{eq:l2.4.2} gives
\begin{align*}
    \mathrm{J}_n H(s,t)
    &=\abs{\phi} e^{-2t\phi\sin\alpha} \abs{ \inner{(\sin\alpha \RR^{0}_{0,V}-\cos\alpha\RR^{\frac{\pi}{2}}_{0,V})P_V\bg{x(s)-z}, \wh{N}(x(s))}}\cdot \mathrm{J}_{n-1} x(s) \\
    &=\abs{\phi} e^{-2t\phi\sin\alpha} \abs{ \inner{\RR^{\frac{\pi}{2}-\alpha}_{0,V}P_V\bg{x(s)-z}, \wh{N}(x(s))}}\cdot \mathrm{J}_{n-1} x(s),
  \end{align*}
  where $\wh{N}$ is the classical normal map.
  Then similar to the estimate \eqref{eq:l2.5.4}, we have
  \begin{align*}
    &\abs{ \inner{\RR^{\frac{\pi}{2}-\alpha}_{0,V}P_V\bg{x(s)-z}, \wh{N}(x(s))}} =    \abs{ \inner{\RR^{\frac{\pi}{2}}_{0,V}P_V\bg{x(s)-z}, \RR^{\alpha}_{0,V}P_V\wh{N}(x(s))}} \\
    = &\abs{P_V\wh{N}(x(s))} \abs{P_V(x(s)-z)} \abs{ \sin \angle(P_V\bg{x(s)-z}, \RR^{\alpha}_{0,V}P_V\wh{N}(x(s)))} \\
    \le &\abs{P_V\wh{N}(x(s))} d_{\R^n}(\RR^{\alpha}_{x,V}(n_x), z+V^{\perp}).
  \end{align*}
  By an analysis similar to that in \autoref{lem:bettermeas_isom}, we obtain
    \begin{align*}
    \abs{\bigcup_{p \in L^{\sigma}} p(E)} 
    &\le \abs{\phi} \int_{E} \int_{0}^{1}e^{-2t\phi\sin\alpha}   \abs{P_V(\wh{N}(x))} d_{\R^n}(\RR^{\alpha}_{x,V}(n_x), z+V^\perp)  \dd t \dd \HH^{n-1}(x) \\
    &\le  C_{\alpha,r}\abs{\phi} \int_{E} \abs{P_V(\wh{N}(x))} \cdot d_{\R^n}(\RR^{\alpha}_{x,V}(n_x), z+V^\perp)  \dd \HH^{n-1}(x).
  \end{align*}
  Then, proceeding analogously to the proofs of \autoref{lem:measisom} and \autoref{lem:smallmeasisom}, we may apply this estimate to obtain \autoref{lem:measaff} and \autoref{lem:smallmeasaff}.

  With these basic lemmas, we can use similar methods given in \autoref{sect:Kak_isom} and \autoref{sect:Nikodym} 
  to study the Kakeya needle problem and Nikodym type set in the setting of affine transformations in $\R^n$.
  Finally, we can get \autoref{thm:global_Kak_aff} and \autoref{thm:aff_Nikodym} as desired.

% References
%\bibliographystyle{plain}
\bibliographystyle{myabbrv}

%\phantomsection       %% hyperlink for `Reference'
%\addcontentsline{toc}{section}{\refname}     %% add `reference' to contents

\bibliography{reference}

\newpage

\end{document}